\documentclass[11pt,preprint]{imsart}

\usepackage{graphicx, subfig}
\usepackage[top=1in,bottom=1in,left=1in,right=1in]{geometry}
\usepackage[numbers,sort&compress]{natbib} 
\usepackage{amsfonts, amsmath, amssymb, amsthm, bbm}
\usepackage[normalem]{ulem}
\usepackage{comment}
\usepackage{eufrak}

\usepackage{algorithm}
\usepackage{algorithmic}

\usepackage{color}
\definecolor{darkred}{RGB}{100,0,0}
\definecolor{darkgreen}{RGB}{0,100,0}
\definecolor{darkblue}{RGB}{0,0,150}

\usepackage{hyperref}
\hypersetup{colorlinks=true, linkcolor=darkred, citecolor=darkgreen, urlcolor=darkblue}
\usepackage{url}

\newtheorem{thm}{Theorem}[section]
\newtheorem{prp}[thm]{Proposition}
\newtheorem{lem}[thm]{Lemma}

\newtheorem{cor}[thm]{Corollary}

\def\beq{\begin{equation}}
\def\eeq{\end{equation}}
\def\beqn{\begin{eqnarray*}}
\def\eeqn{\end{eqnarray*}}
\def\bitem{\begin{itemize}}
\def\eitem{\end{itemize}}
\def\benum{\begin{enumerate}}
\def\eenum{\end{enumerate}}
\def\bmult{\begin{multline*}}
\def\emult{\end{multline*}}
\def\bcenter{\begin{center}}
\def\ecenter{\end{center}}

\newcommand{\ol}[1]{\overline{#1}}

\DeclareMathOperator{\rank}{rank}

\DeclareMathOperator{\tr}{tr}

\def\cA{\mathcal{A}}
\def\cB{\mathcal{B}}
\def\bC{\mathcal{C}}

\def\bE{\mathcal{E}}
\def\cF{\mathcal{F}}
\def\cG{\mathcal{G}}

\def\cN{\mathcal{N}}

\def\cP{\mathcal{P}}
\def\cQ{\mathcal{Q}}

\def\cS{\mathcal{S}}

\def\cU{\mathcal{U}}

\def\bA{\mathbf{A}}
\def\bB{\mathbf{B}}
\def\bC{\mathbf{C}}
\def\bD{\mathbf{D}}
\def\bE{\mathbf{E}}

\def\bG{\mathbf{G}}

\def\bI{\mathbf{I}}

\def\bN{\mathbf{N}}

\def\bS{\mathbf{S}}

\def\bU{\mathbf{U}}
\def\bV{\mathbf{V}}
\def\bW{\mathbf{W}}

\def\bY{\mathbf{Y}}

\def\1{{\mathbf 1}}

\newcommand\bSigma{{\boldsymbol\Sigma}}

\newcommand\bPi{{\boldsymbol\Pi}}

\newcommand{\E}{\operatorname{\mathbb{E}}}
\newcommand{\R}{\operatorname{\mathbb{R}}}
\renewcommand{\P}{\operatorname{\mathbb{P}}}
\newcommand{\Var}{\operatorname{Var}}

\newcommand{\var}[1]{\operatorname{Var}\left(#1\right)}

\newcommand{\pa}[1]{\left(#1\right)}

\newcommand{\cro}[1]{\left[#1\right]}
\newcommand{\ac}[1]{\left\{#1\right\}}

\newcommand{\nv}[1]{\textcolor{magenta}{[Nicolas: #1]}}

\begin{document}

\begin{frontmatter}

\title{Optimal Estimation of Schatten Norms of a rectangular Matrix}

\runtitle{Estimation of Schatten norms}
\begin{aug}
  \author{Sol\`ene Th\'epaut and Nicolas Verzelen}
  \runauthor{Th\'epaut, S. and Verzelen, N.}
\affiliation{Universit\'e Paris-Saclay, NukkAI, and INRAE}
\end{aug}

\begin{abstract}
  We consider the twin problems of estimating the effective rank and the Schatten norms $\|\bA\|_{s}$ of a rectangular $p\times q$ matrix $\bA$ from noisy observations. When $s$ is an even integer, we introduce a polynomial-time  estimator of $\|\bA\|_s$ that achieves the minimax rate $(pq)^{1/4}$. Interestingly, this optimal rate does not depend on the underlying rank of the matrix $\bA$. When $s$ is not an even integer, the optimal rate is much slower. A simple thresholding estimator of the singular values achieves the rate $(q\wedge p)(pq)^{1/4}$, which turns out to be optimal up to a logarithmic multiplicative term. The tight minimax rate is achieved by a more involved polynomial approximation method. This allows us to build estimators for a class of effective rank indices. As a byproduct, we also characterize the minimax rate for estimating the sequence of singular values of a matrix. 
  \end{abstract}

    

    \end{frontmatter}

\maketitle

\section{Introduction}
In many modern problems, scientists are faced with a large data matrix $\bY$, which is often assumed to be sum of a signal matrix $\bA$ and a background noise. Under some structural assumptions on the signal such as small rank, it is possible to precisely recover  the signal matrix $\bA$, for instance with  singular value thresholding methods (see e.g.~\cite{chatterjee2015matrix,gavish2014optimal}). Low-rank assumptions are at the heart of unsupervised learning methods, including clustering or biclustering. In this work, we focus our attention on estimating  specific functionals of the matrix $\bA$ that are connection to its rank.  On the one hand, the rank or the  effective rank of the signal help to  assess the relevance of low-rank based procedures. On the other hand, evaluating the rank (or the effective rank) of $\bA$ may also be an objective per se as a characterization of the complexity of the signal matrix $\bA$. We argue below that inferring the effective rank of a matrix $\bA$ mostly boils down to  estimating Schatten norms of $\bA$.  This manuscript is dedicated to the latter estimation problem.

To be more specific, we consider the signal plus noise  model
\beq \label{eq:noise}
\bY = \bA + \bE\ ,
\eeq
where $\bA$ is the unobserved signal and $\bE$ is a $p \times q$ noise matrix with independent entries following a standard normal distribution.  Extension to non-Gaussian noise are dealt with at the end of the manuscript. 
 Without loss of generality the noise variance is set to one. Also,  we assume throughout the manuscript that $p\geq q$.

Given a $p\times q$ matrix $\bA$, we write $\sigma_1(\bA) \geq \sigma_2(\bA)\ldots \geq \sigma_q(\bA)\geq 0$ for its ordered sequence of singular values. 
The rank of a matrix $\bA$ corresponds to its number of positive singular values. 
For any $s\geq 1$, the $s$-Schatten norm of $\bA$ is defined as the $l_s$ norm of its sequence of singular values, that is 
\beq\label{eq:definition_schatten_norm}
\|\bA\|^{s}_s = \sum_{i=1}^{q} \sigma^{s}_i(\bA)\ . 
\eeq
For $s=\infty$, we define $\|\bA\|_{\infty}=\sigma_1(\bA)$ as the operator norm of $\bA$. 
More generally, given a function $f:\mathbb{R}^+\mapsto \mathbb{R}^+$, we define the functional $f_{\sigma}(A)$ by $f_{\sigma}(\bA)= \sum_{i=1}^{q}f(\sigma_i(\bA))$. For instance, $\|\bA\|_s^s= f_{\sigma}(\bA)$ when one takes $f:x\mapsto x^s$. This manuscript is dedicated to the problem of estimating the Schatten norms of $\bA$ from a single noisy  observation $\bY$. Before describing our contribution, we explain how Schatten norms are related to effective rank measures.

\subsection{Effective Rank of a matrix}\label{sec:intro_effective_rank}

Since the rank of a matrix  is very sensitive to small perturbations, it is  difficult to estimate it from $\bY$. Furthermore,  a large rank matrix $\bA$ may have only few large singular values together with many small singular values. For such a matrix, the rank is poorly informative on the structure of $\bA$. 
As an alternative, various notions of effective ranks have been introduced. In particular, some of these notions of effective rank are at the heart of  high-dimensional or infinite-dimensional probabilistic results. For instance,  in~\cite{koltchinskii2014concentration}, Koltchinskii and Lounici consider, for a non-negative symmetric matrix $\bSigma$, the indicator $\tr[\bSigma]/\|\bSigma\|_{\infty}= \|\bSigma\|_1/\|\bSigma\|_{\infty}$. For a  rectangular matrix $\bA$, one can think of two extensions of this index, depending on whether we work directly with the singular values of $\bA$ or with the singular values of the square matrix $\bA^T \bA$. 
\beq\label{eq:definition_effective_rank}
\mathrm{ER}_{1,\infty}(\bA)= \frac{\|\bA\|_1 }{\|\bA\|_{\infty}} \, \quad  ;\quad \quad \quad     \mathrm{ER}_{2,\infty}(\bA)= \mathrm{ER}_{1,\infty}(\bA^T \bA)= \frac{\|\bA\|_2^2 }{\|\bA\|^2_{\infty}}\ . 
\eeq
Other notions of effective ranks are based on entropy measures. For instance Roy and Veterlli~\cite{roy2007effective} introduce the Shannon effective rank:
\beq\label{eq:definition_shannon_rank}
\mathrm{ER}_{1,1}(\bA)= \exp\left[-\sum_{i=1}^q \frac{\sigma_i(\bA)}{\|\bA\|_1 }\log\left(\frac{\sigma_i(\bA)}{\|\bA\|_{1}}\right)\right]\ , 
\eeq
which corresponds to the Shannon effective number of the distribution induced by the probability vector ($\sigma_i(\bA)/\|\bA\|_1$) $i=1,\ldots, q$. More generally, one  extends any effective number index based either on probability vector associated to  the singular values of $\bA$, that is the vector ($\sigma_i(\bA)/\|\bA\|_1$), $i=1,\ldots, q$ or to  the  vector ($\sigma_i^2(\bA)/\|\bA\|^2_2$), $i=1,\ldots, q$. For the purpose of biodiversity analysis, a large class of diversity measures are considered in~\cite{jost2006entropy}. In particular, for any positive $s>0$ and different from 1, the Hill's effective number~\cite{jost2006entropy} which derived from Renyi entropy straightforwardly straightforwardly extends to the following effective rank indices
\beq\label{eq:definition_effective_hill}
\mathrm{ER}_{1,s}(\bA)= \left(\frac{\|\bA\|_s }{\|\bA\|_1}\right)^{s/(1-s)}\, ;  \quad \quad \mathrm{ER}_{2,s}(\bA)= \mathrm{ER}_{1,s}(\bA^T\bA)=  \left(\frac{\|\bA \|_{2s} }{\|\bA \|_2}\right)^{2s/(1-s)}\ . 
\eeq
When all the non-zero singular values of $\bA$ are equal,  all these effective rank indices are equal to the true rank of $\bA$. However, these measures differ in the way they treat heterogeneous values for the singular values. In short, smaller values of $s$  in  $\mathrm{ER}_{1,s}(\bA)$ and $\mathrm{ER}_{2,s}(\bA)$ are more prone to take into account smaller singular values in the effective rank. See~\cite{jost2006entropy} and references therein for further discussions.

\medskip 

Most work around effective ranks consider noiseless setting \cite{roy2007effective}, but the matrix $\bA$ is sometimes allowed to be only partially observed (e.g~\cite{balcan2019testing}). In this work, we tackle the general problem of estimating the effective rank of $\bA$ relying on a single noisy observation $\bY$ of $\bA$. To the exception of the Shannon effective rank~\eqref{eq:definition_shannon_rank}, all the other effective rank measures are ratio of Schatten norms of $\bA$, so that evaluating the former amounts to estimating well the latter. Besides, assessing the intrinsic difficulty of estimating Schatten norms $\|\bA\|_s$ can lead the statistician to  favor some specific effective rank measures, that are based on Schatten norm $\|\bA\|_s$ that are easier to estimate.

\subsection{Our Contribution}

In most of the manuscript, we focus our attention to the Gaussian signal plus noise matrix model~\eqref{eq:noise}. As a warm-up, we consider the Frobenius norm $\|\bA\|_2$ and we prove that the simple quadratic estimator $(\|\bY\|_2^2 - pq)^{1/2}_+$, where $x_+= \max(x,0)$ achieves the optimal risk $(pq)^{1/4}$. Interestingly, this risk $(pq)^{1/4}$ cannot be improved by any estimator even if the matrix $\bA$ is additionally known to be of rank at most one. Then, we establish that a non-linear transformation of $\sigma_1(\bY)$ estimates $\|\bA\|_{\infty}=\sigma_1(\bA)$ with the same optimal error rates $(pq)^{1/4}$.

Regarding general even norms $\|\bA\|_{2k}$ where $k$ is any integer, we first remark that $\|\bA\|_{2k}^{2k}=tr[(\bA^T\bA)^{2k}]$ is a polynomial with respect to the entries of $\bA$. This allows us to  build an unbiased estimator $U_k$ of $\|\bA\|_{2k}^{2k}$  based on Hermite polynomials. Relying on the invariance of $\|\bA\|_{2k}^{2k}$ by left and right orthogonal transformations, we establish that this estimator has a simple expression as an algebraic combination of monomials of the form $tr[(\bY\bY)^{l}]$ so that the estimator $U_k$ can be efficiently computed. One of our main result is general variance upper bound for $U_k$, which allows us to prove that the estimator $(U_k)_+^{1/(2k)}$ achieves the optimal risk $(pq)^{1/4}$ uniformly over all matrices $\bA$.

Regarding general norms $\|\bA\|_{s}$ where $s\geq 1$ is not an even integer, we first exhibit a simple plug-in estimator based on a linear transformation of the empirical singular values ($\sigma_i(\bY)$) that achieves a much higher error of the order $q^{1/s}(pq)^{1/4}$ (compare with $(pq)^{1/4}$ for even Schatten norms). Our second main result is a minimax lower bound of order $\frac{q^{1/s}}{\log^{s}(q)}(pq)^{1/4}$ stating that it is impossible to estimate non-even Schatten norms at a much faster rate. Using polynomial approximation techniques that approximate $\|\bA\|_s$ by a linear combination of even Schatten norms $\|\bA\|_{2k}^{2k}$, we are able to partially close this logarithmic gap between our upper and lower minimax bounds. 

As a first application of our estimator $U_{2k}$, we apply a moment matching method introduced by~\cite{kong2017spectrum} in a covariance estimation model to estimate the whole sequence of singular values $\bA$ at the minimax rate $\tfrac{q}{\log(q)}(pq)^{1/4}$. As a second application of our findings, we build an estimator of the effective rank $\mathrm{ER}_{2,\infty}$ and control its error in probability. Besides, we argue why Effective ranks measures $\mathrm{ER}_{2,s}$ where $s\geq 2$ is an integer are easier to estimate than other definitions~(\ref{eq:definition_effective_rank},\ref{eq:definition_shannon_rank},\ref{eq:definition_effective_hill}).

Finally, we extend our analysis to a signal plus noise model with a general subGaussian noise distribution. Our third main result is an error control of the estimator $(U_k)_+^{1/(2k)}$ for even Schatten norm $\|\bA\|_{2k}$. 
Quite surprisingly, we establish that, while the definition of $U_k$ heavily depends on the sequence of moments of the normal distribution, $(U_k)_+^{1/(2k)}$ achieves the same convergence rate $(qp)^{1/4}$ as for Gaussian noise. However, the analysis turns out to be much more involved in comparison to Gaussian case.

\subsection{Related Literature}

\paragraph{Low rank Matrix Estimation and Detection.} There is a long line of work regarding the estimation of the matrix $\bA$ when $\bA$ is assumed to be low-rank or approximately low-rank. Procedures based on singular value thresholding procedures~\cite{gavish2014optimal,chatterjee2015matrix}, that is estimating $\bA$ by setting small singular values of $\bY$ to 0,  have been proved to achieve near optimal performances. Donoho and Gavish~\cite{donoho2014minimax} have assessed the asymptotic risk of singular value soft-thresholding procedures in an asymptotic framework where the rank is proportional to $q$ and $p$ is proportional to $q$.  \cite{shabalin2013reconstruction,nadakuditi2014optshrink} have introduced other non-linear singular values shrinkage estimators that turn out to achieve smaller risks (the constant in front of the rates is slightly better), at least in an  asymptotic framework where $p/q \rightarrow c\in (0,\infty)$ and $\bA$ has finite rank low-rank. As explained in Section~\ref{sec:operator},  some of our Schatten norm estimators are reminiscent of those non linear shrinkage functions.

\paragraph{Asymptotic estimation of the singular values in the asymptotic $p/q\rightarrow c$ .} In asymptotic matrix analysis, most work have been devoted to the setting where $p/q\rightarrow c\in (0,\infty)$. In that literature, the model~\eqref{eq:noise} is coined as the information plus noise model~\cite{couillet2011random}.  Given the symmetric matrix $\frac{1}{p}\bY^T\bY$ with ordered eigenvalue $\lambda_1,\ldots, \lambda_q$, define the spectral measure as $\mu_{\bY^T\bY}:= \frac{1}{q}\sum_{i=1}^{q}\delta_{\lambda_i}$. In the specific case where $\bA=0$, it has been established~\cite{anderson2010introduction} that $\mu_{\bY^T\bY}$ converges weakly towards the Marchenko-Pastur distribution. More generally, when the spectral measure of $\frac{1}{p}\bA^T \bA$ converges to a probability measure, \cite{vallet2012improved} have characterized the limit of  $\mu_{\bY^T\bY/p}$ through its Stieljes transform. Given that characterization, one may try to invert the mapping from the limit of $\mu_{\bA^T \bA/p}$ to the limit of $\mu_{\bY^T\bY/p}$ to estimate some functional of the spectrum of $\bA^T\bA$ from $\bY$. In some specific case, where $\bA$ has asymptotically a finite number $k$ of distinct singular values, consistency of these singular values has been derived. See \cite[][Ch.8]{couillet2011random} and references therein for more details. Given a consistent estimation of the spectral measure, on may plug it to estimate the Schatten norms. However, results in this line of work are not comparable to ours as they are restricted to very specific settings.

\paragraph{Matrix Sketching.} In computer science, there is an active line of research around the problem of matrix sketching which can be defined as follows: the practitioner has access to some entries of the matrix $\bA$ or to some linear combination of the entries of the matrix $\bA$. In this noiseless problem, one aims at recovering functionals of $\bA$ such as a Schatten norm or the rank using the smallest possible budget. See~\cite{li2014sketching,khetan2019spectrum,balcan2019testing}. As in our noisy framework,~\cite{li2014sketching} emphasizes that estimating even Schatten norms using sketches seems easier than estimating non-even Schatten norms. Apart from this, the techniques for both the lower and upper bounds highly differ from our settings.

\paragraph{Estimation of non linear functionals.} Schatten norms are non-linear functionals of the entries of $\bA$. Vectorizing the matrices $\bY$, $\bA$, and $\bE$, we can rewrite our problem in the Gaussian sequence model. In particular, estimating the 2-Schatten norm is equivalent to estimating the $l_2$ norm of a vector in the Gaussian sequence model. See~\cite{donoho1990minimax,collier2015minimax} and references therein for an account of work of quadratic functional estimation. More generally, the estimation of the $l_r$ norm of a vector has been addressed in~\cite{lepski1999estimation,cailow2011,mukherjee2017}. However, apart from the specific case $s=2$, $s$-Schatten norms do no correspond to $l_r$ norm and the methodology developed in the Gaussian sequence model does not extend to our setting. 
In discrete distribution estimation, optimal rates of convergence for Shannon entropy, Renyi entropy have been respectively established in \cite{wu2016minimax} and \cite{jiao2015minimax} (see also~\cite{han2018local}) and the polynomial approximation techniques in Section~\ref{sec:general} are inspired from this literature.

Closer to our settings, Kong and Valiant~\cite{kong2017spectrum} have considered the problem of estimating even Schatten norms of a covariance matrix $\bSigma$ given a $n$-sample of a mean zero random vector with covariance matrix $\bSigma$. 
Interestingly, their estimator exhibit a low computational complexity and is unbiased for general noise distributions. 
Also, they rely on these Schatten norms to estimate the whole sequence of eigenvalues of $\bSigma$. Our application to estimating the sequence of singular values is directly inspired from their work.  
However, their setting much differs from ours as they have access to $n$ independent observations so that they can easily build an unbiased estimator of any polynomials in $\bSigma$. In contrast, we only have a single observation of the matrix $\bY$ so that higher-degree polynomials in $\bA$ are more difficult to estimate for non-Gaussian noise distribution. 
Also, Kong and Valiant~\cite{kong2017spectrum} do not assess at all the optimality of their rates. In contrast, we provide matching minimax lower bound in our information plus noise model. As an aside, our proof approach could perhaps extend to their covariance framework.

\subsection{Notation}
In this work, $c$, $c_1$, $c'$ denote numerical positive constants that may vary from line to line. We write that two quantities $u$ and $v$ satisfy $u\lesssim v$ if there exists a numerical constant $c$ such that $u\leq c v$. 
We recall that norm of the form $\|\bA\|_s$ for $(s\in [1, \infty]$) correspond to the $s$-Schatten norms. 
In contrast, for $s\in [1,\infty]$, we write $|\bA|_s$ and $|u|_s$ for entry-wise $l_s$ norms of the matrix $\bA$  and the vector $u$.

Section \ref{sec:warm_up} is devoted to the simple case of Frobenius norm estimation and provides both minimax upper and lower bounds. In Section~\ref{sec:operator}, we consider the problem of estimating the operator norm $\|\bA\|_{\infty}$. General even norms are addressed in Section~\ref{sec:even}, whereas   non-even norms are studied in Section~\ref{sec:general}. In Section~\ref{sec:non_gaussian_noise}, we extend our results to general noise distributions.
Finally, we come back to the problem of effective estimation together with some open questions in the discussion Section. Proofs are postponed to the appendix.

\section{Frobenius and Operator Norm}\label{sec:warm_up}

\subsection{Frobenius Norms}
We first estimate $\|\bA\|_2^2= \sum_{i,j}\bA_{ij}^2$ and then we deduce an estimator of $\|\bA\|_2$. Remark that $\|\bA\|_2^2$ can be interpreted as a square $l_2$ norm in a Gaussian sequence model. The latter problem is classical -- see e.g.~\cite{collier2015minimax} and our simple estimator $U_1$ for $\|\bA\|_2^2$ has been already analyzed in the Gaussian sequence model. Still, we provide here a self-contained analysis as a warming-up for general even Schatten norms. 
Since $\|\bY\|_2^2$ follows a non-central $\chi^2$ distribution, we derive from  standard computations for the normal random variables that $\E[\|\bY\|_2^2]= \|\bA\|_2^2 +pq$. This leads us to the following unbiased  estimator 
\beq\label{eq:definition_U_1}
U_1 = \|\mathbf{\bY}\|_{2}^2 - pq\ .
\eeq
\begin{prp}\label{prp:risk_frobenius}
For any $p\times q$ matrix $\bA$ it holds that
\beqn 
\mathbb{E} \cro{|U_1 - \|\bA\|_2^2|} &\leq & \sqrt{2pq} + 2\|\bA\|_2\\
\mathbb{E} \cro{|(U_1)_+^{1/2} - \|\bA\|_2|} &\leq& 3(pq)^{\frac{1}{4}}\enspace .
\eeqn 
\end{prp}

\noindent 
 {\bf Remark}: Since estimating $\|\bA\|_2^2$  amounts to  estimating the square $l_2$ norm (resp. $l_2$ norm) of a noisy vector of size $pq$, this problem has been thoroughly  studied in the literature. See~\cite{collier2015minimax} and references therein. In particular, the counterparts of $U_1$ and $(U_1)_+^{1/2}$ in the vector setting were  already analyzed in~\cite{collier2015minimax}. Still, we provide proofs of the above proposition for the sake of completeness. Next, we assess the optimality of Proposition~\ref{prp:risk_frobenius} by providing a minimax lower bound. 

\begin{prp}\label{prp:lower_frob}
There exists a numerical constant $c>0$ such that 
\begin{equation*}
\inf_{\widehat{T}}\sup_{\bA:\  \rank(\bA)\leq 1}\mathbb{E}\left[|\widehat{T} - \|\bA\|_2|\right]\geq c(pq)^{1/4}\ . 
\end{equation*}
 \end{prp}
One cannot estimate  $\|\bA\|_2$ uniformly over all matrices $\bA$ at a rate faster than $(pq)^{1/4}$. More importantly, even when one restricts itself to the much simpler class of matrices of rank at most one, it is impossible to achieve a faster rate. 
This theorem is proved using Le Cam's approach by constructing a discrete prior distribution on the space of rank-one matrix -- see the proof for more details.

Since, for a rank-one matrix $\bA$, all Schatten norms are equal to $\sigma_1(\bA)$, we straightforwardly deduce a minimax lower bound for all Schatten norms. 
\begin{cor}\label{cor:general_lower}
For any real number $s\geq 1$, we have 
\[
 \inf_{\widehat{T}}\sup_{\bA\in \mathbb{R}^{p\times q}}\mathbb{E}\left[|\widehat{T} - \|\bA\|_s|\right]\geq \inf_{\widehat{T}}\sup_{\bA:\  \rank(\bA)\leq 1}\mathbb{E}\left[|\widehat{T} - \|\bA\|_s|\right]\geq c(pq)^{1/4}\ ,
\]
where $c$ is as in Proposition~\ref{prp:lower_frob}.
\end{cor}
As proved later in the manuscript, this minimax lower bound turns out to be tight when $s$ is an even integer.

\subsection{Operator Norm $\|\bA\|_{\infty}$}\label{sec:operator}

In this subsection, we address the case $s=\infty$ which amounts to  estimating $\|\bA\|_{\infty}=\sigma_1(\bA)$.  In view of Corollary~\ref{cor:general_lower}, the optimal estimation risk is at least of the order of $(pq)^{1/4}$ for general matrices $\bA$. 
From the decomposition $\bY= \bA+\bE$, we may be tempted to use  plug-in estimators of the form  $ \|\bY\|_{\infty}$.  
\begin{lem}\label{prp:nogestim}
For any matrix $\bA$, we have 
\[
\E[|\|\bY\|_{\infty}-\|\bA\|_{\infty}|]\leq \sqrt{p}+\sqrt{q} \ . 
\]
Conversely, there exist two positive numerical constants $c$ and $c'$ such that, for $p$ large enough, 
\[
 \sup_{\bA,\, \rank(\bA)\leq 1}\ \E[|\|\bY\|_{\infty}-\|\bA\|_{\infty}|]\geq c'\sqrt{p}\  . 
\]
\end{lem}

When $p$ is of the same order as $q$, that is $\bA$ is almost a square matrix, $\sqrt{p}$ is of the same order as $(pq)^{1/4}$ and these simple plug-in estimators turn out to be optimal. However, for highly rectangular matrices, these estimators achieve a much slower rate than the minimax lower bound.  To improve the rate from $\sqrt{p}$ to $(pq)^{1/4}$, we start by estimating the operator norm $\|\bA^T\bA\|_{\infty}= \sigma_1^2(\bA)$. Define the $q\times q$ matrix $\bW=\bY^T\bY - p \bI_q$ where $\bI_q$ is the identity matrix. Simple calculations lead to $\E[\bW]= \bA^T \bA$. This motivates us to to estimate $\sigma_1^2(\bA)$ by  $\sigma_1(\bW)= \sigma_1^2(\bY)-p$.

\begin{prp}\label{prp:control_operator}
There exists a numerical constant $c>0$ such that, for any matrix $\bA$, one has
\begin{eqnarray}\nonumber
\mathbb{E}\cro{|\sigma_1^2(\bY)-p - \sigma_1^2(\bA)|} &\leq& 3\sqrt{pq} + 4\|\bA\|_{\infty}\sqrt{q} + 8e^{-p/4} + \sqrt{64\pi p} \ ;\\
\label{eq:conclusion_operator_norm}
\mathbb{E}\cro{|\left(\sigma_1^2(\bY)-p\right)_+^{1/2} - \|\bA\|_{\infty}|} &\lesssim &   \pa{pq}^{\frac{1}{4}}\ .
\end{eqnarray}

\end{prp}
This new estimator $\left(\sigma_1^2(\bY)-p\right)_+^{1/2}$ achieves the optimal rate $(pq)^{1/4}$ uniformly over all possible matrices $\bA$. As for the Frobenius norm, a low-rank assumption on $\bA$ does not ease the operator norm estimation problem. 

From asymptotic considerations where $p/n\rightarrow c'\in (0,\infty)$ and $\bA$ has finite rank $r$, Shabalin and Nobel~\cite{shabalin2013reconstruction} have introduced an estimator $\widehat{\bA}$ of $\bA$ satisfying 
\[
  \sigma^2_i(\widehat{\bA})= 
  \frac{1}{2}\left[\sigma_i^2(\bY) - p- q + \sqrt{[\sigma^2_i(\bY)-p-q]^2- 4pq }\right] \  ,
\]
when $\sigma_i(\bY)$ is large enough. When $q$ is small compared to $p$, then the quantity $\sigma_1(\widehat{\bA})$ does not much differ from our estimator $\sqrt{(\sigma_1(\bY)-p)_+}$. In fact, it is not hard to prove  that  $\sigma_1(\widehat{\bA})$ also estimates $\|\bA\|_{\infty}$ at the optimal rate $(pq)^{1/4}$ in analogy to~\eqref{eq:conclusion_operator_norm}.

\section{Estimation of General even norms}\label{sec:even}

We now turn to the case of general even Schatten norms. In this section, $k$ stands for a positive integer and we are interested in estimating $\|\bA\|_{2k}= [\sum_{i=1}^q \sigma^{2k}_i(\bA)]^{2k}$. We will show in this section that, by carefully studying an unbiased estimator of $\|\bA\|_{2k}^{2k}$, it is possible to estimate $\|\bA\|_{2k}$ at the desired rate $(pq)^{1/4}$.  Recall that, when $k$ is an integer, we have the following identity
\[
 \|\bA\|_{2k}^{2k}= \tr[(\bA^T\bA)^{k}]= \sum_{i_1,\ldots, i_k=1}^{p}\sum_{j_1,\ldots, j_k=1}^{q} \prod_{t=1}^{k} \bA_{i_tj_t}\bA_{i_{t}j_{t+1}}\ ,
\]
with the convention $j_{k+1}=j_1$. Obviously, $\|\bA\|_{2k}^{2k}$ is a polynomial of order $2k$ with respect to the entries of $\bA$. Since $\bY_{ij}\sim \cN(\bA_{ij},1)$, one can unbiasedly estimates each monomial using polynomials in the entries of $\bY$.

Write $\phi(y)= e^{-y^2/2}(2\pi)^{-1/2}$ for the density of the standard normal random variables.  For a positive integer $r$, we define the Hermite polynomial of degree $r$ by the equation
\beq\label{eq:definition_hermite}
 \frac{d^r}{dy^{r}}\phi(y) = (-1)^{r}H_r(y) \phi(y)
\eeq
It is well known (see e.g.~\cite{cailow2011}) that, for  $Z\sim \cN(x,1)$,  $\E[H_r(Z)]= x^r$. Given two sequences  $i=(i_1,\ldots, i_k)$ and $j=(j_1,\ldots,j_k)$,  we denote $N_{rs}(ij)$ the number of occurrences of the $2$-tuple $(r,s)$ in the sequences $(i_t,j_t)$ and $(i_{t}, j_{t+1})$ with $t=1,\ldots, k$  (recall that, by convention $i_{k+1}=i_1$). Then, the estimator
\beq\label{eq:definition_u_q}
U_{k}= \sum_{i_1,\ldots, i_k=1}^{p}\sum_{j_1,\ldots, j_k=1}^{q} \prod_{r=1}^p \prod_{s=1}^q H_{N_{rs}(i,j)}(\bY_{rs})
\eeq
obviously satisfies $\mathbb{E}[U_k]= \tr[(\bA^T\bA)^k]= \|\bA\|_{2k}^{2k}$.

Although it is easy to deduce the expectation $\E[U_k]$ from~\eqref{eq:definition_u_q}, this definition is not convenient from a practical perspective as  it may suggest that $O(p^{k}q^{k})$ operations are needed to evaluate $U_k$. In fact, it turns out that $U_k$ simplifies as an algebraic combination of Schatten norms  of $\bY^T \bY$ -- see the next proposition.

Given a positive integer $l$, we write  $\cS[l]$  for the collection of  integer-valued sequences $s$ whose entries are  nondecreasing and that sum to $l$. In other words, any vector $s\in \cS[l]$ satisfies $1\leq s_1\leq s_2\leq \ldots \leq s_{|s|}$ and $\sum_{i}s_i=l$. 

\begin{prp}\label{prp:representation_somme_newton}
 There exist coefficients $\alpha_s$ (depending only on $p$ and $q$) such that
 \beq\label{eq:definition:representation_somme_newton}
 U_k = \tr[(\bY^T\bY)^k]+ \alpha_0 + \sum_{l=1}^{k-1} \sum_{s\in \cS[l]} \alpha_s \prod_{i=1}^{|s|} \tr[(\bY^T \bY)^{s_i}]\ .
 \eeq
\end{prp}

\noindent 
The proof of this proposition relies on the orthogonal invariance of the Gaussian distribution and on the representation of symmetric polynomials by newton sums.  As a consequence of this result, the estimator $U_k$ can  be computed in  $O((p+k)q^2+e^{c\sqrt{k}})$ operations. Indeed, we first evaluate and store all the $2l$ Schatten norms of $Y$ for $l=1,\ldots, k$ and then we compute $\sum_{l=1}^{k-1}|\cS[l]|$ combinations of these norms. Since $|\cS[l]|$ is the number of partition of the integer $l$, we know from Hardy and Ramanujan's asymptotic formula that $\log(|S[l]|)\sim \pi\sqrt{2l/3}$.

\medskip 

\noindent 
{\bf Remark}: In the specific case $k=1$, we have already established in the previous section that $U_1= \tr[\bY^T\bY]-pq$. In general, the proof of Proposition~\eqref{prp:representation_somme_newton} only ensures the existence of the coefficients in~\eqref{eq:definition:representation_somme_newton}.  Still, by computing the expectation of 
$ \tr[(\bY^T \bY)^{s}]$, one can recursively debias $\tr[(\bY^T \bY)^{k}]$ to work out the expression of $U_k$. As an example, the following proposition provides an explicit expression for $U_2$ and $U_3$.

\begin{prp}\label{prp:U2}
For any $q\leq p$, we have 
\beqn 
U_2& =& \tr[(\bY^T\bY)^2]- 2(p+q+1) \tr[\bY^T\bY]+ pq(1+p+q)\ ; \\
U_3& = &\tr[(\bY^T\bY)^3] -3(p+q+1)\tr[(\bY^T\bY)^2]+ 3[p^2+q^2 +pq+p+q-2]\tr[\bY^T \bY]\\ && + pq\left[-p^2-q^2+5\right] \ .
\eeqn 
\end{prp}

As another consequence of Proposition~\ref{prp:representation_somme_newton}, the distribution of $U_k$ is invariant by left and right  orthogonal transformations of $\bY$. Hence, we can assume henceforth that $\bA$ is null outside its diagonal and that $\bA_{ii}=\sigma_i(\bA)$ for $i=1,\ldots, q$. This observation allows us to work out the variance of $U_k$ in the next theorem.

\begin{thm}\label{prp:risk_U_k}
For any integer $k\geq 2$, and any $p\geq q$, we have 
 \[
 \Var{(U_k)}\leq   c^{k} \left[   (pq)^{k} +  p \|\bA\|_{4k-4}^{4k-4}+ \|\bA\|_{4k-2}^{4k-2} + p^kk^{3k}+  k^{6k}\right]\ ,
 \]
 where $c$ is a numerical constant. 
\end{thm} 

The proof is based on some combinatorial arguments: starting from the Hermite definition~\eqref{eq:definition_u_q} of $U_k$, we first establish necessary conditions on sequences $(i_1,j_1)$,\ldots, $(i_k,j_k)$  and  $(i'_1,j'_1)$,\ldots, $(i'_k,j'_k)$ so that the corresponding covariance is non-zero. Then, we count the remaining sequences building on earlier ideas on random matrix analysis~\cite{anderson2010introduction}. 
From this theorem, we deduce a risk bound for the estimator $(U_k)_+^{1/2k}$ of $\|\bA\|_{2k}$.

\begin{cor}\label{cor:schatten}
For any positive integer $k$ and any $p\times q$ matrix $\bA$ one has 
\beq
\E\left[|(U_k)_+^{1/(2k)}- \|\bA\|_{2k}|\right]\leq c^{k}\left[(pq)^{1/4}+ p^{1/4}k^{3/4}+ k^{3/2}\right]\ ,
\eeq
 where $c$ is a positive numerical constant. 
\end{cor}

Interestingly, the estimator $(U_k)_+^{1/(2k)}$ achieves the optimal rate $(pq)^{1/4}$ for any even norm. This implies that, as for Frobenius norm estimation, estimating the $2k$-Schatten norm of a small rank matrix $\bA$ is as difficult as estimating a full rank matrix.

\section{Estimation of general non-even norms}\label{sec:general}

\subsection{Simple estimators}

In contrast to even Schatten norms, one first observes that it is impossible to devise unbiased estimators of non-even norms because such functionals are not smooth with respect to $\bA$.

\begin{lem}\label{lem:no_unbiased_estimator}
For all square integrable  estimators $\varphi(\bY)$ and all $r>0$ and, there exists a matrix $\bA$ such that $\|\bA\|_1\leq r$ and $\E[\varphi(\bY)]\neq \|\bA\|_1$. 
\end{lem}

\begin{proof}[Proof of Lemma~\ref{lem:no_unbiased_estimator}]
The proof arguments are classical. For any $t\in \mathbb{R}$, consider the matrix $\bA_t$ such that $(\bA_t)_{1,1}=t$ and $(\bA_t)_{i,j}=0$ otherwise. Obviously $\|\bA_t\|_1= |t|$. The function $t\mapsto \|\bA_t\|_1$ is not differentiable at $0$. In contrast, the function $t\mapsto \mathbb{E}[\varphi(\bY)]= (2\pi)^{-pq/2}\int \varphi(\bY) e^{-\|\bY\|_2^2/2-t^2/2+ t\bY_{11}} d\bY$ is differentiable by Lebesgue's dominated convergence theorem. As a consequence, both  functions cannot match in any neighborhood of $0$. 
\end{proof}

Consider any real number $s\geq 1$. As a starting point and to assess the quality of more refined procedure, we consider the plug-in estimator $\|\bY\|_s$ that simply evaluates the $s$-Schatten norm  of the observed matrix $\bY$. 
\begin{lem}\label{lem:plug_in_nuclear_norm}
 For any matrix $\bA\in \mathbb{R}^{p\times q}$ and any $s\geq 1$, we have   $\mathbb{E}\left[|\|\bY\|_s- \|\bA\|_s |\right]\leq  2 q^{1/s}\sqrt{p}$. 
\end{lem}
\begin{proof}[Proof of Lemma \ref{lem:plug_in_nuclear_norm}]
From the triangular inequality, we deduce that $\mathbb{E}\left[|\|\bY\|_s- \|\bA\|_s |\right]\leq \mathbb{E}\left[\|\bE\|_s\right]$. Since $\|\bE\|_s\leq q^{1/s}\|\bE\|_{\infty}$, we simply need to bound the expected operator norm of $\bE$. By Lemma~\ref{lem:spectrum} in the appendix, we have $\E[\|\bE\|_{\infty}]\leq \sqrt{p}+\sqrt{q}$, which concludes the proof. 
\end{proof}

\noindent
{\bf Remark}: The risk bound achieved by $\|\bY\|_s$ is much larger than what has been obtained by unbiased estimators of the even norm. Even for square matrices $(p=q)$, Lemma~\ref{lem:plug_in_nuclear_norm} enforces a rate of the order $p^{1/s}\sqrt{p}$ whereas even Schatten norms can be estimated at the much faster rate $\sqrt{p}$.

\medskip 

\noindent 
{\bf Remark}: As in Section~\ref{sec:operator}, it turns out that the risk  upper bound $q^{1/s}\sqrt{p}$  cannot be significantly improved for this specific estimator $\|\bY\|_s$.  Indeed, consider the specific case  $\bA=0$. For any $p\geq 4q$, we have $\E[|\|\bY\|_s-\|\bA\|_s|]\geq c q^{1/s} \sqrt{p}$. Indeed, $\|\bY\|_s\geq q^{1/s} \sigma_q(\bY)$, and we deduce from Lemma~\ref{lem:DavSza}, that with probability higher than $1-e^{-q/8}$, $\sigma_q(\bY)\geq \sqrt{p} - 3/2\sqrt{q}\geq \sqrt{p}/4$. The result follows.

\medskip 

Next,  we  improve over the rate $p^{1/s}\sqrt{p}$ by relying on a correction of the empirical singular values in the same spirit as what we did for estimating the operator norm $\|\bA\|_{\infty}$. Consider the centered Gram matrix $\bW= \bY^T \bY- p \bI_q$ and observe that  
\[
\E[\bW]= \bA^T \bA + \E[\bE^T \bA + \bA^T \bE] + \E[\bE^T\bE-p\bI_q]= \bA^T \bA\ . 
\]
As a consequence, $\bW$ is an unbiased estimator $\bA^T\bA$, which could suggest that the eigenvalues of $\bW$ are close to that of $\bA^T\bA$, which in turn are the square of the singular values of $\bA$. Since the $i$-th eigenvalues of $\bW$ equals $\sigma_i^2(\bY)-p$, this leads us to considering the following estimator of $\|\bA\|_s$
\beq\label{eq:definition_T_S}
T_s= \left[\sum_{i=1}^q [(\sigma_i^2(\bY)-p)_+]^{s/2}\right]^{1/s}\ .
\eeq	
\begin{prp}\label{prp:upper_risk_nuclear_covariance}
For  any matrix $\bA\in \mathbb{R}^{p\times q}$ and any $s\geq 1$, we have
\begin{eqnarray}
 \E[|T^s_s - \|\bA\|^s_s|]&\leq& c^{s} \left[ q[(p\vee s)(q\vee s)]^{s/4}+ \sqrt{q}\|\bA\|_{s-1}^{s-1}+\1_{s\geq 2}\sqrt{pq}  \|\bA\|_{s-2}^{s-2}\right]\ ;  \nonumber  \\
\E[|T_s - \|\bA\|_s|]&\leq& c^{s} q^{1/s}[(p\vee s)(q\vee s)]^{1/4}\ ,
\label{eq:upper_risk_nuclear_covariance}
\end{eqnarray}
for some positive numerical constant $c>0$. 
\end{prp}

The proof of this result relies on interlacing inequalities for eigenvalues (Corollary III. 1.5 in~\cite{bhatia2013matrix}).
In comparison to the naive plug-in estimator $\|\bY\|_s$, the $q^{1/s}\sqrt{p}$ factor has been replaced by $q^{1/s}(pq)^{1/4}$. For highly rectangular matrices, there is therefore  a significant benefit to correcting the singular values.  Still, 
this risk bound is still $q^{1/s}$ times higher than the optimal risk bound for even Schatten norms. 

In the remainder of this section, we establish that, at least when $\|\bA\|_{\infty}$ is not too high this rate can be improved by a logarithmic factor by approximating $\|\bA\|^{s}_s$ by a suitable linear combination of even Schatten norms $\|\bA\|_{2k}^{2k}$. Then, we show that this new rate turns out to be minimax optimal. 

\subsection{Improved risk bounds by polynomial approximation}

Introducing the function $\psi_s: x \mapsto  x^{s/2}$ for $x>0$, we observe that $\|\bA\|^{s}_{s}= \sum_{i=1}^q \sigma^{s}_i(\bA)= \sum_{i=1}^q \psi_s(\sigma_i(\bA^T\bA))$. Following ideas pioneered  in Lepski et al.~\cite{lepski1999estimation,cailow2011,mukherjee2017} for estimating the $l_q$ norm of a vector in the Gaussian sequence model, we suggest to approximate the function $\psi_s$ by some polynomial function $P=\sum_{k=0}^{K}a_k x^k$ and then to build an unbiased estimator of $\sum_{i=1}^q P(\sigma^2_i(\bA))$ as a proxy for $\|\bA\|_s^{s}$. 
Indeed, $\sum_{i=1}^q P(\sigma^2_i(\bA))= a_0q+ \sum_{k=1}^{\deg(P)}a_k\|\bA\|_{2k}^{2k}$ is a weighted sum of even Schatten norms $\|\bA\|^{2k}_{2k}$ with $k$ ranging from one to the degree of $P$ and we can rely on the results of Section~\ref{sec:even} to build an unbiased estimator of  this quantity.

Given a polynomial real polynomial $P$ of degree $K$ and a tuning parameter $M>1$,  we therefore define the estimator
\beq\label{eq:definition_polynomial_approxiation_nuclear_norm}
G_s[P; \bY]=  M^{s}(pq)^{s/4}\left[a_0+ \sum_{k=1}^{K} a_k \frac{U_k}{(M(pq)^{1/4})^{2k}}\right]\ ,
\eeq
where $U_k$ is the unbiased estimator of $\|\bA\|_{2k}^{2k}$  defined in~\eqref{eq:definition_u_q}.

For a continuous function $g$ defined on an interval $I$, we denote  $|g|_{\infty,I}=\sup_{x\in I}|g(x)|$ its supremum norm. The following proposition bounds the risk of $G_s[P; \bY]$ in terms of its approximation error to the function $\psi_s$  and of a variance-like expression growing super-exponentially with the degree of the polynomial.

\begin{prp}\label{lem:risk_G_P}
Consider any polynomial $P$ of degree $K\leq q^{1/3}$.  Assuming that $\sigma_1(\bA)\leq M(pq)^{1/4}$, we have 
\beq\label{eq:upper_risk_G_P}
 \mathbb{E}\left[|G_s[P; \bY]- \|\bA\|^{s}_s|\right]\leq M^s(pq)^{s/4} \left[q |\psi_s-P|_{\infty,[0,1]}+ \frac{c^{K}}{M^2\wedge [M(p/q)^{1/4}]} |a|_{\infty}\right]\ .
\eeq
\end{prp}

In order to optimize the bound in~\eqref{eq:upper_risk_G_P}, we want to choose $P$ to minimize the approximation error $|\psi_s-P|_{\infty,[0,1]}$ while having coefficients $|a|_{\infty}$ not too large.

\subsubsection{Nuclear Norm $(s=1)$}

To ease the presentation, we first focus on the nuclear norm, the general case being handled below. 
In general, devising the best polynomial approximation of a function on a compact is a difficult problem. Fortunately, studying the best polynomial approximation of the square root function over $[0,1]$ boils down to studying the best polynomial approximation of the absolute function on $[-1,1]$, the latter problem being well studied in the approximation theory literature~\cite{rivlin1990chebyshev}. Indeed, given a polynomial $Q$, and any $x\in [-1,1]$, we have $||x|-Q(x^2)|= |\sqrt{x^2}- Q(x^2)|$ so that  $\sup_{x\in [-1,1]}||x|- Q(x^2) |= \sup_{x\in [0,1]}|\sqrt{x}-Q(x)|$.

In particular, an explicit and nearly optimal approximation of the absolute function by a degree $K$ Polynomial is given by its Chebychev's expansion. 
The  Chebychev polynomial (of the first kind) of degree $k$ is defined by
\[
 L_k(x)=\sum_{j=0}^{\lfloor k/2\rfloor}(-1)^{j}\frac{k}{k-j}\binom{j}{k-j}2^{k-2j-1}x^{k-2j}\ .
\]
The degree $2K$ expansion of the absolute function in Chebychev basis is  $B_K(x)= \frac{2}{\pi}L_0(x) + \frac{4}{\pi}\sum_{k=1}^{K}(-1)^{k+1}\frac{L_{2k}(x)}{4k^2-1}$ which is a symmetric polynomial of degree $2K$. 
Rewriting this polynomial in the canonical form as $B_{K}(x)= \sum_{k=0}^{K}a^*_k x^{2k}$, we build the polynomial $P^*_K(x)= \sum_{k=0}^{K}a^*_k x^{k}$ so that $P^*_K(x^2)= B_K(x)$. Relying on Bernstein's~\cite{bernstein1913valeur} approximation results, we then obtain the following.

In the sequel, we write $\cP_n$ for the space of polynomials of degree at most $n$. Given a collection $\mathcal{F}$ of functions, an interval $I$ and a function $g:I\mapsto \mathbb{R}$, we denote
\beq \label{eq:definition_best_approximation}
E_{\cF}[g,I]= \inf_{f\in \cF}|f-g|_{\infty,I}\ ,
\eeq
for the best uniform approximation of $g$ by $\cF$. 

\begin{lem}\label{lem:best_approximation_psi}
 Consider any positive integer $K$. The modified Chebychev polynomial $P^*_K$ satisfies
 \[
 |\psi_1 -P^*_K|_{\infty, [0,1]}\leq \frac{2}{\pi(2K+1)}\ ,
 \]
 and $|a^*|_{\infty}\leq 2^{3K}$.  Besides, no degree $K$ polynomial is able to achieve a much lower error. 
 \[
  \lim_{n\rightarrow \infty}2n E_{\cP_n}\left(\psi_1;[0,1]\right)= \beta^*\ ,
 \]
where $\beta^*\approx 0.28$ is the Bernstein constant.  
\end{lem}

Plugging this bound in Proposition~\ref{lem:risk_G_P} and choosing well the degree $K$ leads to the following result.

\begin{prp}\label{prp:risk_approximation}
Fix any $M>1$. Choose $K^*= \lfloor c \log(q)\rfloor$  for a suitable numerical constant $c$. If $\sigma_1(\bA)\leq M(pq)^{1/4}$, then the estimator  $G_1[P^*_{K^*}; \bY]$ satisfies
\beq\label{eq:upper_risk_oracle}
 \mathbb{E}\left[|G_1[P^*_{K^*}; \bY]- \|\bA\|_1|\right]\lesssim  M  \frac{q}{\log(q)}  (pq)^{1/4} \ . 
\eeq
\end{prp}

Let us compare this risk bound with that of   the plug-in estimator $T_1$ (Proposition~\ref{prp:upper_risk_nuclear_covariance}). When $M$ is considered as a constant, the risk in~\eqref{eq:upper_risk_oracle} improves by a logarithmic factor $\log(q)$ the bound for $T_1$. As will be established in Section~\ref{sec:lower_general} below, this rate $q (pq)^{1/4}/\log(q)$ turns out to be optimal.

Besides, it is possible to empirically check the hypothesis $\sigma_1(\bA)\leq M(pq)^{1/4}$ using the estimator $U_{\infty}=[\sigma^2_1(\bY)-p]^{1/2}_+$ of $\|\bA\|_{\infty}$ since $\E[|U_{\infty}-\sigma_1(\bA)|]\lesssim (pq)^{1/4}$. Hence, if $U_{\infty}$ is too large, then the modified plug-in estimator $T_1$ is perhaps preferable to  $G_1[P^*_{K^*}; \bY]$.

\subsubsection{General case $(s\geq 1)$}

We turn to the general case $s\geq 1$. In view of the discussion for the nuclear norm, we first have to quantify the best polynomial approximation on $[-1,1]$ of the functions $x\mapsto x^{s}$.  The following result dates from Bernstein~\cite{bernstein1938meilleure} and can also be found in \cite[][Sec.7.2.1]{timan2014theory}.
 \begin{lem}\label{lem:approximation_theory_polynom}
  There exists a function $\mu:\mathbb{R}^+\setminus \{0\}\mapsto \mathbb{R}^+$ such that, for any $s>0$, which is not an even integer, we have 
  \[
  \lim_{n\rightarrow \infty} n^{s}E_{\cP_n}\left(|x|^{s};[-1,1]\right)= \mu(s) > 0\ .
  \]
 \end{lem}
 Given an integer $K$, write $P^*_{K;s}$ a polynomial such that $P^*_{K;s}(x^2)$ achieves $E_{\cP_{2K}}\left(|x|^{s};[-1,1]\right)$. As argued above in the specific case $s=1$, $P^{*}_{K,s}$ approximates well the function $x^{s/2}$ on $[0,1]$. This leads us to the following result.

\begin{prp}\label{prp:risk_approximation_general_s}
Fix any $M>1$ and any $s\geq 1$. Choose $K^*= \lfloor c \log(q)\rfloor$  for a suitable numerical constant $c$. If $\sigma_1(\bA)\leq M(pq)^{1/4}$, then the polynomial $G_s[P^*_{K^*;s}; \bY]$ satisfies
\begin{eqnarray}\label{eq:upper_risk_oracle_s}
 \mathbb{E}\left[|G_s[P^*_{K^*;s}; \bY]- \|\bA\|^s_s|\right]&\lesssim & M^{s}  \frac{q}{\log^{s}(q)}  (pq)^{s/4} \ ; \\
 \mathbb{E}\left[|G_s[P^*_{K^*;s}; \bY]^{1/s}- \|\bA\|_s|\right]&\lesssim & M  \frac{q^{1/s}}{\log(q)}  (pq)^{1/4} \ . 
\end{eqnarray}
\end{prp}
As for $s=1$, the estimator $G_s[P^*_{K^*;s}; \bY]$ improves the rate by a factor $1/\log^{s}(q)$ when the $\sigma_1(\bA)$ is not too large.
Building the estimator $G_s[P^*_{K^*;s}; \bY]$ requires to compute the best approximation polynomial for the function $x\mapsto |x|^{s}$. Still, this can be computed offline and once and for all using the algorithm of Remez~\cite{remez1934determination}. See also the discussion in~\cite{pachon2009barycentric} and in~\cite{jiao2015minimax}.

\subsection{Minimax lower bound}\label{sec:lower_general}

We now prove that the rate $\frac{q}{\log^s(q)\vee 1} (pq)^{s/4}$ is optimal for estimating $\|\bA\|_{s}^{s}$ as long as $s\geq 1$ is not an even integer. Since the pioneering work of Lepski et al.~\cite{lepski1999estimation} and Cai and Low~\cite{cailow2011}, Le Cam's method with fuzzy hypotheses has turned out to be fruitful for establishing tight minimax lower bounds for non-smooth functional estimation problems. See for instance  the discussion in~\cite{polyanskiy2019dualizing} and the survey~\cite{wu2021polynomial}.
Although we follow the same general approach for Schatten norm estimation, the computation of the $\chi^2$ distance between the corresponding distributions turns out to be quite challenging. Indeed, we consider the singular value decomposition $\bA= \bU^T \bD \bV$ of $\bA$ and fix uniform prior distribution on $\bU$ and $\bV$ together with some suitable prior distributions on $\bD$. The resulting $\chi^2$ distance involves complex spherical integrals that we are unable to work out. For this reason, we  introduce a sequence of priors and use some data-processing inequalities to come back to a functional estimation problem for a rank-one matrix which turns out to be tractable.

In the sequel, $\cP^{sym}_{k}$ denote the space of the symmetric polynomials of degree less or equal to $k$.

\begin{thm}\label{thm:lower_general}
There exists a numerical constant $c>0$ such that the following holds for all positive integers $4\leq q\leq p$ and for all continuous functions $f:\mathbb{R}_+\mapsto \mathbb{R}$.  Writing  $I_0=[0,0.125(pq)^{1/4}]$ and $k^*= \lceil \log(4\lceil q/2\rceil )\rceil$, we have 
\beq\label{eq:lower_general_risk}
 \inf_{\widehat{T}}\sup_{\bA: \ \|\bA\|_{\infty}\leq 0.125(pq)^{1/4}}\E\left[\big|\widehat{T}- f_\sigma(\bA)\big|\right]\geq c\left[q E_{\cP^{sym}_{2k^*}}[f;I_0]]- q^{1/2}\|f\|_{\infty,I_0}\right]_+\ ,
\eeq
where we recall that $f_{\sigma}(\bA)= \sum_{i=1}^{q}f(\sigma_i(\bA))$.
\end{thm}

The above minimax lower bound is valid for all functionals of the form $\sum_{i=1}^q f(\sigma_i(\bA))$, although the bound uninformative if $f$ is a symmetric polynomial of degree less than $\log(q)$. Next, we apply this result to the functional $\psi_s: x\mapsto |x|^s$.

\begin{cor}\label{prp:lower_general_norm}
 For any $s\geq 1$, which is not an even integer, there exist constants $c_s>0$ and $c'_s>0$ such that for all positive integers $q\leq p$, we have 
\begin{eqnarray}
 \inf_{\widehat{T}}\sup_{\bA: \ \|\bA\|_{\infty}\leq 0.125(pq)^{1/4}}\E\left[\big|\widehat{T}- \|\bA\|^s_s\big|\right]&\geq& c_s\frac{q}{\log^s(q)\vee 1} (pq)^{s/4}\ ; \nonumber \\
 \inf_{\widehat{T}}\sup_{\bA: \ \|\bA\|_{\infty}\leq 0.125(pq)^{1/4}}\E\left[\big|\widehat{T}- \|\bA\|_s\big|\right]&\geq& c'_s\frac{q^{1/s}}{\log^{s}(q)\vee 1} (pq)^{1/4}\ ;\label{eq:lower_nuclear_norm} 
 \end{eqnarray}
\end{cor}

In view of Propositions~\ref{prp:risk_approximation_general_s} and Corollary~\ref{prp:lower_general_norm}, the optimal rates for estimating $\|\bA\|^s_s$ is of the order of $q(pq)^{s/4}/\log^{s}(q)$. Regarding the $s$-Schatten norm $\|\bA\|_s$, there is a small logarithmic gap between the upper and lower bound except in the case where $s=1$.

\section{Estimating the whole sequence of singular values}\label{sec:wasserstein}

In this section, we are interested in reconstructing the sequence of singular values $\sigma_1(\bA)$, \ldots, $\sigma_q(\bA)$ in  $l_1$ distance. In other words, we aims at building an estimator $\widehat{\sigma}_1$,\ldots, $\widehat{\sigma}_q$ such that $\E[\sum_{i=1}^q |\sigma_i(\bA) - \widehat{\sigma}_i(\bA)|]$ is as small as possible.

First, we study the simple corrected plug-in estimator $\widehat{\sigma}_i= [(\sigma_i^2(\bY)-p)_+]^{1/2}$ introduced in Section~\ref{sec:general}. Arguing as in the proof of Proposition~\ref{prp:upper_risk_nuclear_covariance}, we show the following result.

\begin{prp}\label{prp:upper_l1_spectrum}
For any integers $p\geq q$ and any $p\times q$ matrix $\bA$, the estimator  $\widehat{\sigma}_i= [(\sigma_i^2(\bY)-p)_+]^{1/2}$ satisfies 
\beq\label{eq:rate_l1}
  \E\left[\sum_{i=1}^{q}|\widehat{\sigma}_i-\sigma_i(\bA)|\right]\lesssim q(pq)^{1/4}\enspace .  
\eeq
\end{prp}

Conversely, by triangular inequality, any estimator ($\tilde{\sigma}_i$) satisfies
\[
  \E\left[\sum_{i=1}^{q}|\tilde{\sigma}_i-\sigma_i(\bA)|\right]\geq   \E\left[|\sum_{i=1}^{q}\widehat{\sigma}_i-\|\bA\|_1|\right]\ , 
\]
which, in light of Corollary~\ref{prp:lower_general_norm} is higher than $c q(pq)^{1/4}/\log(q)$, at least for some matrices $\bA$ with $\sigma_1(\bA)\in [0, 0.125(pq)^{1/4}]$. Hence, one can improve the rate in~\eqref{eq:rate_l1} at best by a factor $1/\log(q)$. We explain how to do this in the remainder of this section by extending the moment matching method of Kong and Valiant~\cite{kong2017spectrum} to our setting. As in the previous section, we shall assume henceforth that the singular values of $\bA$ are bounded by $M(pq)^{1/4}$ for some known $M>1$.

Write $\delta_x$ for  the Dirac measure at $x$. 
Given a vector $\sigma= (\sigma_1, \ldots, \sigma_q)$, let $\mu_{\sigma}=   q^{-1}\sum_{i=1}^{q}\delta_{\sigma_i}$ where $\delta_x$ denote the associated probability measure. For two measures $\mu_1$ and $\mu_2$ on $\mathbb{R}$ with cumulative distribution functions $F_{\mu_1}$ and $F_{\mu_2}$,  the Wasserstein distance $W(\mu_1,\mu_2)$ is defined as 
\[
W(\mu_1,\mu_2)= \int_{\mathbb{R}}|F_{\mu_1}(t)- F_{\mu_2}(t)|dt \ . 
\]
It was previously observed (e.g. in~\cite{kong2017spectrum}) that the $l_1$ distance between two ordered vectors is proportional to  the Wasserstein distance between the two corresponding measure. 
\beq\label{eq:connection_Wasserstein_l1}
\sum_{i=1}^q |\sigma_i- \sigma'_i| = q W(\mu_{\sigma};\mu_{\sigma'}) \ . 
\eeq
Indeed, $ q W(\mu_{\sigma};\mu_{\sigma'})= q\int |F_{\mu_{\sigma}}(t) - F_{\mu_{\sigma'}}(t)|dt= \int \big|\sum_{i=1}^q (\1_{\sigma_i\leq  t} - \1_{\sigma'_i\leq t})\big|dt$. For a fixed $t$, all $q$ differences $(\1_{\sigma_i\leq  t} - \1_{\sigma'_i\leq t})$ are either $0$ or share the same sign. As a consequence  $ q W(\mu_{\sigma};\mu_{\sigma'})=\int \sum_{i=1}^q \big|\1_{\sigma_i\leq  t} - \1_{\sigma'_i\leq t}\big|dt= \sum_{i=1}^q |\sigma_i- \sigma'_i|$ and our problem is equivalent to estimating the measure $\mu_{\sigma(\bA)}$ in Wasserstein distance.

In~\cite{kong2017spectrum}, Kong and Valiant bound the Wasserstein distance between two compactly distributed probability measures in terms of their sequence of $k$ first moments (see Lemma~\ref{lem:kong_valiant_original} below).  This led them to first estimate the $k$-first moment of a measure and then pick a measure whose moments are  closest to the estimated moments. See Algorithm 1  in~\cite{kong2017spectrum}. In our setting, we are only able to estimate efficiently the even moments. Still, we can adapt their approach by nearly matching the first even moments instead of all the first moments.

\medskip 
Let $K$ be a tuning parameter. Denote the size $K$ vector $\widehat{m}$ by $\widehat{m}_k=q^{-1}[M(pq)^{1/4}]^{-2k}U_k$ where $U_k$ defined in~(\ref{eq:definition_u_q}) is our estimator of the $\|\bA\|_{2k}^{2k}$. Fix $\zeta= 1/(M^2q)$ and consider a fine regular grid $x= (x_1,\ldots,x_d)$ on $[0,b(pq)^{1/4}]$ with $x_i=(i-1)\zeta$ and $d= \lceil M(pq)^{1/4}/\zeta \rceil$. 
Define also the  $K\times d$ matrix $\bV$ such that $\bV_{ij}= x_j^{2i}$. We consider the following two step estimator of the singular values vector. 
\begin{enumerate}
 \item Let $p^+\in \mathbb{R}^{d}$ be the solution to the following linear program, which we will regard
as a distribution consisting of point masses at values $x$
\beq\label{eq:definition_p+}
\mathrm{minimize}_{p}|\bV p - \widehat{U}|_1 \text{ subject to } 1^T p =1\text{ and  }p > 0
\eeq
\item Return the singular values $\widehat{\sigma}_i$ by taking the rescaled $i$-th $(q+1)$st quantile of the distribution corresponding to $p^+$, that is $\widehat{\sigma}_i= M(qp)^{1/4}\min\{x_j: \sum_{l\leq j }p_l^+\geq \frac{i}{q+1}\}$.

\end{enumerate}

\begin{thm}\label{thm:consistent_singular_values}
Fix any $M>1$ and choose $K= \lfloor c\log(q)\rfloor$ for a suitable numerical constant $M>1$. If $\sigma_1(\bA) \leq M(pq)^{1/4}$, then the estimator $\widehat{\sigma}_i$ defined above satisfies
\beq\label{eq:upper_risk_wasserstein}
W\left(\mu_{\widehat{\sigma}} ,\mu_{\sigma(\bA)}\right)\lesssim  \frac{M}{\log(q)}(pq)^{1/4}\ .
\eeq
In particular, this implies that 
\[
\E\left[ \sum_{i=1}^q |\widehat{\sigma}_i-\sigma_i(\bA)|\right]\lesssim  \frac{M}{\log(q)}q(pq)^{1/4}\ . 
\]
\end{thm}

As a simple corollary, we observe that the plug-in estimator $\sum_{i=1}^{q}\widehat{\sigma}_i$ estimates the nuclear norm $\|\bA\|_1$ at a rate similar to that of \eqref{eq:upper_risk_oracle}. 
More generally, the dual representation of Wasserstein distance entails that for any Lipschitz function with Lipschitz constant $L$ on $[0,M(pq)^{1/4}]$, we have 
\[
 \Big|\sum_{i=1}^q f(\widehat{\sigma}_i)-\sum_{i=1}^q f(\sigma_i(\bA))\Big|= q  \big|\inf f (d\mu_{\widehat{\sigma}} -d\mu_{\sigma(\bA)}) \big|\lesssim L M\frac{q}{\log(q)}(pq)^{1/4}\ . 
\]
Hence, the plug-in estimator achieves the rate $Mq(pq)^{1/4}/\log(q)$ uniformly over all Lipschitz functional of the singular values. 
See~\cite{han2018local} for related discussions in discrete settings.   Regarding the nuclear norm, the slight downside of this plug-in approach  is that it requires to solve the linear program~\eqref{eq:definition_p+}, while Bernstein polynomials are more explicit.

\section{Extension to general noise distributions}\label{sec:non_gaussian_noise}

Until now, we assumed that  the noise components $(\bE_{ij})$ follow a  standard normal distribution.  In this section, we extend our results to a non-Gaussian model $\bY= \bA + \bE$ where the entries $(\bE_{i,j})$ of $\bE$ are independent, and identically distributed as a sub-Gaussian random variable $E$ satisfying $\E[E]=0$ and $\Var{(E)}=1$. In the sequel, we write $\|E\|_{\psi_2}$ for its sub-Gaussian norm -- see e.g.~\cite{vershynin2018high} for a definition. As in the first part of the manuscript, we start with the Frobenius and operator norms before turning to general even norms and to non-even norms.   The estimator $U_1=tr[(\bY^T\bY)]-qp$ of $\|\bA\|_F^2$ defined in~\eqref{eq:definition_U_1} is a second degree polynomial and we easily work out its variance. 
\[
  \mathbb{E} \cro{|U_1 - \|\bA\|_2^2|^2} = \var{E^2}pq + 4\|\bA\|^2_2+ 4\E[E^3]\sum_{i,j}\bA_{ij}\leq  3\E[E^4][pq + 2\|\bA\|_2^2] \  ,
\]
where we used Young's and Cauchy Schwarz inequality. Arguing as in the proof of Proposition~\ref{prp:risk_frobenius}, we arrive at
\[
  \mathbb{E} \cro{|(U_1)_+^{1/2} - \|\bA\|_2^2|} \leq 3\sqrt{\E[E^4]}(pq)^{1/4}\  .
\]
The proof is omitted. Similarly, we can extend the analysis of the estimator $U_{\infty}= (\sigma_1^2(\bY)-p)_+^{1/2}$ of $\|\bA\|_{\infty}$ or the estimator $T_{s}= [\sum_{i=1}^q [(\sigma_i^2(\bY)-p)_+]^{s/2}]^{1/s}$ of the $s$-th Schatten norm $\|\bA\|_s$ to obtain similar bounds to the Gaussian case.

  \begin{prp}\label{prp:control_operator_subgaussian}
    For any matrix $\bA\in \mathbb{R}^{p\times q}$, we have
    \begin{equation}\label{eq:conclusion_operator_norm_subgaussian}
    \mathbb{E}\cro{|\left(\sigma_1^2(\bY)-p\right)_+^{1/2} - \|\bA\|_{\infty}|} \leq  c_{\|E\|_{\psi_2}} \pa{pq}^{\frac{1}{4}}\ ,
    \end{equation}
    for a positive numerical quantity $c_{\|E\|_{\psi_2}}$  that only depends on $\|E\|_{\psi_2}$. 
    \end{prp}

  \begin{prp}\label{prp:upper_risk_nuclear_covariance_subgaussian}
    For  any matrix $\bA\in \mathbb{R}^{p\times q}$ and any $s\geq 1$, we have
    \begin{eqnarray}
    \E[|T_s - \|\bA\|_s|]&\leq& c^{s}_{\|E\|_{\psi_2}} q^{1/s}[(p\vee s)(q\vee s)]^{1/4}\ ,
    \label{eq:upper_risk_nuclear_covariance_sugaussian}
    \end{eqnarray}
    for some positive quantity $c_{\|E\|_{\psi_2}}$ that only depends on $\|E\|_{\psi_2}$. 
    \end{prp}

In contrast, it is much less clear to what extent the $2k$-Schatten norms $\|\bA\|_{2k}$ can be estimated at the rate $(pq)^{1/4}$ for general sub-Gaussian noise distribution. Indeed, the definition our estimator $U_k$ introduced in Section~\ref{sec:even} delicately depends on the sequence of moments of the standard normal distribution through the use of Hermite polynomials. Since 
only the two first moments of $E$ match those of the standard normal distribution, the estimator $U_k$ is therefore biased for all $k>1$.  Still, we are able to control its risk in the next theorem.  The arguments significantly differ from the Gaussian case. Indeed, we could assume in the proof of Theorem~\ref{prp:risk_U_k} that $\bA$ is diagonal by distribution invariance under orthogonal transformations. This is not the case here and we need to carefully enumerate all the terms in $U_k$ using generalized M\"obius formula for partitions and   graph enumeration arguments.

\begin{thm}\label{thm:non_gaussian}
For any positive integer $k\geq 1$ and any $p\geq q$, we have 
     \beq\label{eq:risk_upper_bound_non_gaussian}
      \E\left[\left(U_k - \|\bA\|_{2k}^{2k}\right)^2\right]\leq   (c\|E\|_{\psi_2}k)^{c'k} \left[   (pq)^{k} +  pq \|\bA\|_{\infty}^{4k-4}+ q\|\bA\|_{\infty}^{4k-2}\right]\ ,
     \eeq
     where $c$ and $c'$ are numerical constants.
\end{thm}

In contrast to Theorem~\ref{prp:risk_U_k} for Gaussian noise, the Schatten norms $\|\bA\|_{4k}^{4k-4}$ and $\|\bA\|_{4k}^{4k-2}$ in the risk upper bounds have been respectively replaced by the looser bounds $q \|\bA\|_{\infty}^{4k-4}$ and $q \|\bA\|_{\infty}^{4k-2}$. Still, this mild  modifications does not affect the rates for estimating $\|\bA\|_{2k}$ as stated  in the next corollary.
\begin{cor}\label{cor:schatten_non_gaussian}
  For any positive integer $k$, and any $p\times q$ matrix $\bA$ one has 
  \beq
  \E\left[|(U_k)_+^{1/(2k)}- \|\bA\|_{2k}|\right]\leq (c\|E\|_{\psi_2}k )^{c'k} (pq)^{1/4}\ ,
  \eeq
   where $c$ and $c'$ are positive numerical constants. 
  \end{cor}
As in the  Gaussian case, we achieve a risk bound of the order of $(pq)^{1/4}$  for any even Schatten norm. Looking more closely at the risk bound~\eqref{eq:risk_upper_bound_non_gaussian}, we observe that the risk bound is super-exponential with respect to $k$, whereas the dependency of the risk bound in the Gaussian case is only exponential with respect to $k$.

Now we can rely on Theorem~\ref{thm:non_gaussian} to analyze the  polynomial approximation estimators of general Schatten norms $\|\bA\|_{s}$ as in Section~\ref{sec:general}, the main difference  being that the degree $K^*$ of the polynomial should be now set to be of the order of $\frac{\log(q)}{\log\log(q)}$ to account for the super exponential dependency with respect to $k$ in~\eqref{eq:risk_upper_bound_non_gaussian}. This results in additional $\log\log(q)$ terms in the rates. For instance, the nuclear norm can now be estimator up to the rate $Mq\frac{\log\log(q)}{\log(q)}(pq)^{1/4}$ instead of the rate $M\frac{q}{\log(q)}(pq)^{1/4}$. In summary, up to these $\log\log(q)$ possible loss, all the results extend to a subGaussian noise distribution.

\section{Discussion}\label{sec:discussion}

\subsection{Application to Effective rank estimation }

Throughout this subsection, we assume that $q\geq 2$, otherwise the problem of effective rank estimation simply corresponds to that of testing the nullity of a vector.  
Coming back to the problem of estimating the  Effective rank of a noisy matrix, we deduce from the previous sections  that one should favor ratio of even Schatten norms as these are much easier to estimate. Let us further focus on the specific choice
\beq\label{eq:effective_rank}
\mathrm{ER}_{2,\infty}(\bA)= \frac{\|\bA\|_2^2 }{\|\bA\|^2_{\infty}}\ . 
\eeq
In light of our analysis in Section~\ref{sec:warm_up} of Frobenius and operator norms, we consider the estimator
\beq\label{eq:effective_rank_estimation}
\widehat{\mathrm{ER}}_{2,\infty}(\bA)= \max\left[\frac{\tr[\bY^T \bY -p \bI_q]}{\|\bY^T \bY -p \bI_q\|_{\infty} },1\right]\ . 
\eeq
We cannot simply plug Propositions~\ref{prp:risk_frobenius} and \ref{prp:control_operator} to control this estimator as we only proved expectation bounds. The following result pushes the analysis slightly further to derive high probability deviation for the effective rank. 

\begin{prp}\label{prp:effective_rank_estimation}
There exist numerical constants $c$--$c_3$ such that the following holds for any matrix $\bA$ and for any $t>0$. Provided that 
\beq\label{eq:condition_rank_estimation}
 \|\bA\|_2^2\geq c (\sqrt{pqt}+t)\ , \text{ and } \|\bA\|_{\infty}\geq c \left[\sqrt{t}+ (pq)^{1/4}+ (pt)^{1/4}\right]\ ,
\eeq
then, with probability higher than $1-c_3 e^{-t}$, it holds that 
\beq\label{eq:prp:rank_estimation}
 \frac{\big|\widehat{\mathrm{ER}}_{2,\infty}(\bA)  - \mathrm{ER}_{2,\infty}(\bA)\big|}{\mathrm{ER}_{2,\infty}(\bA)}\leq c' \left[\frac{\sqrt{pqt}}{\|\bA\|_2^2}+ \frac{\sqrt{pt}+ \sqrt{pq}}{\|\bA\|^2_{\infty}}+ \frac{\sqrt{q}+\sqrt{t}}{\|\bA\|_{\infty}}\right]\ . 
\eeq
\end{prp}

If  $\sigma_1(\bA)= \|\bA\|_{\infty}$ is large compared to $(pq)^{1/4}$, the above proposition implies that the effective rank is well estimated. 
Condition~\eqref{eq:condition_rank_estimation} is unavoidable for consistent effective risk estimation. Indeed, we established in the proof of Proposition~\ref{prp:lower_frob} that it is impossible to distinguish $\bA$ from the null matrix when $\sigma_1(\bA)$ is less than $(pq)^{1/4}$. It turns out that the plug-in estimator~\eqref{eq:effective_rank_estimation} is optimal  as stated by the next proposition.

\begin{prp}\label{prp:rank_estimatino_optimalite}
  Fix $\alpha\in (0,1/2)$. If $\sigma_1(\bA)\geq c_{\alpha} (pq)^{1/4}$, then, with probability higher than $1-\alpha$, we have 
\beq\label{eq:prp:rank_estimation2}
 \frac{\big|\widehat{\mathrm{ER}}_{2,\infty}(\bA)  - \mathrm{ER}_{2,\infty}(\bA)\big|}{\mathrm{ER}_{2,\infty}(\bA)}\leq c'_{\alpha}\left[\frac{\sqrt{pq}}{\sigma_1^2(\bA)}+ \frac{\sqrt{q}}{\sigma_1(\bA)}\right] \ . 
\eeq
Conversely, for any $r\geq (pq)^{1/4}$, we have the following lower bound on the error
\beq \label{eq:lower}
\inf_{\tilde{R}}\sup_{\bA:\,  \sigma_1(\bA)\geq r }\P\left[\frac{|\widehat{R} -\mathrm{ER}_{2,\infty}(\bA)\big|}{\mathrm{ER}_{2,\infty}(\bA)}\geq c \left(\frac{\sqrt{pq}}{r^2}+ \frac{\sqrt{q}}{r}\right)\right] \geq 0.2\ .
\eeq
\end{prp}
The optimal convergence rate of the effective rank is of the order of $\frac{\sqrt{pq}}{\sigma^2_1(\bA)}$ when $\sigma_1(\bA)$ is in $[(pq)^{1/4}; \sqrt{p}]$ and of the order of $\frac{\sqrt{q}}{\sigma_1(\bA)}$ for $r\geq \sqrt{p}$.

\subsection{Low-Rank matrices}

The optimal risk of estimating even Schatten norms does not depend on the rank of $\bA$. In other words, estimating $\|\bA\|_{2k}$ for a rank-one matrix is as difficult as for a full rank matrix. The situation seems quite different for general Schatten norms as we have obtained the much  slower convergence rates $q^{1/s}(pq)^{1/4}$ for $\|\bA\|_s$ when $s$ is not even. In the extreme case of rank-one matrix, this rates can be drastically improved to $(pq)^{1/4}$ as $\|\bA\|_s=\|\bA\|_2$ so that we can simply use $(U_1)_+^{1/2}$. More generally, we could hope to derive faster rates for $\|\bA\|_s$ when $\|\bA\|$ is a small rank (or approximately a small rank matrix). This is an interesting direction for future research.

\subsection{Estimating general $\|\bA\|_s$ without any constraint on $\|\bA\|_{\infty}$}

In Section~\ref{sec:general}, we have introduced a simple modified plug-in estimator $T_1$ for $\|\bA\|_1$ up to an an error $q(pq)^{1/4}$. Then, assuming that $\|\bA\|_{\infty}\leq M(pq)^{1/4}$, for some $M>0$, we have analyzed a moment-based estimator that achieves the rate $q(pq)^{1/4}/\log(q)$. This rate turns out to be minimax optimal, at least for bounded $\|\bA\|_{\infty}$. For other functional estimation problems in the Gaussian vector model~\cite{cailow2011} or in discrete distribution~\cite{jiao2015minimax,han2018local}, the authors introduce an unified estimator by considering separately observations corresponding to parameters in the smooth part of the functional and observations corresponding to the non-smooth part of the functional. In our setting, this could amount to estimating separately the large singular values  $\sum_{i=1}^{s}\sigma_i(\bA)$ and then use a polynomial approximation method to estimate the part $\sum_{i=s+1}^{q}\sigma_i(\bA)$ of $\|\bA\|_1$ arising from the bulk of the singular values. Unfortunately, this cannot be easily done in our matrix setting because of the intertwined dependencies between the singular values of ${\bf Y}$.

\subsection*{Acknowledgements}
We are grateful to Christophe Giraud for many stimulating discussions.

\section{Proofs}

\subsection{Frobenius and operator norm }

 \begin{proof}[Proof of Proposition \ref{prp:risk_frobenius}]
The difference $U_1 - \|\bA\|_2^2$ decomposes as 
\begin{equation*}
U_1 - \|\bA\|_2^2 = \|\mathbf{\bY}\|_{2}^2 - pq - \|\bA\|_2^2 = \pa{\|\bE\|_{2}^2 - pq} + 2 \langle\bA, \bE \rangle ,
\end{equation*}
where $\langle \bB,\bC\rangle= \tr(\bB^T \bC)$. First, $\|\bE\|_{2}^2$ follows a $\chi^2$ distribution with $pq$ degrees of freedom. Hence, we derive from Cauchy-Schwarz inequality that  $\mathbb{E}\cro{|\|\bE\|_{2}^2 - pq|}\leq  \mathrm{var}^{1/2}(\|\bE\|_{2}^2)=\sqrt{2pq}$. Since $\bA$ is deterministic, $\langle\bA, \bE \rangle/\|\bA\|_2$  follows a standard normal distribution and $\mathbb{E} \cro{|\langle\bA,\bE \rangle|} \leq \|\bA\|_2$. The first result follows. 
Since $|\sqrt{a}-\sqrt{b}|\leq \sqrt{|a-b|}$, we derive from Cauchy-Schwarz inequality and the first result that 
\beqn 
\mathbb{E}\cro{|(U_1)_+^{1/2} - \|\bA\|_2|} &\leq & \left[\mathbb{E}\left(|(U_1)_+ - \|\bA\|_2 |\right)\right]^{1/2}\\
&\leq & \left[\mathbb{E}\left(|U_1 - \|\bA\|_2 |\right)\right]^{1/2}
\leq  \sqrt{(2pq)^{1/2}+ 2\|\bA\|_{2}}\ \ , 
\eeqn 
which is smaller than $2(pq)^{1/4}$ as long as $\|\bA\|_2\leq \sqrt{pq}$.  Now assume that $\|\bA\|_2\geq\sqrt{pq}$. Since $(x-y)= (x^2-y^2)/(x+y)$, it follows again from the first result  that 
\beqn 
\mathbb{E} \cro{|(U_1)_+^{1/2} - \|\bA\|_2| }  \leq \mathbb{E}\cro{\frac{|U_1 - \|\bA\|_2^2|}{ \|\bA\|_2} } 
\leq \frac{\sqrt{2pq} + 2\|\bA\|_2}{\|\bA\|_2}\leq 2+ \sqrt{2}\leq 3(pq)^{1/4}\ . 
\eeqn 
This concludes the proof.
\end{proof}

\begin{proof}[Proof of Lemma \ref{prp:nogestim}]
By triangular inequality, we have $\E[|\|\bY\|_{\infty}-\|\bA\|_{\infty}|]\leq \E[\|\bE\|_{\infty}]$ and $\E[|\|\bY\|_{\infty}-\E[\|\bE\|_{\infty}]-\|\bA\|_{\infty}|]\leq 2\E[\|\bE\|_{\infty}]$. By Lemma~\ref{lem:spectrum}, the expectation of the operator norm of $\bE$ is less or equal to $\sqrt{p}+\sqrt{q}$. This concludes the first part of the proof. 
Let us now lower bound the risk of $\|\bY\|_{\infty}$ and $\|\bY\|_{\infty}-\mathbb{E}\cro{\|\bE\|_{\infty}}$.
Choose $\bA$ to be the null matrix. The risk of $\|\bY\|_{\infty}$ is then $\E[\|\bE\|_{\infty}]$. Consider the size $q$ vector $e_1=(1,0,\ldots,0)$. Then, 
\[
 \E[|\|\bY\|_{\infty}-\|\bA\|_{\infty}|]\geq \E[|\bE e_1|_2]\ . 
\]
Since $|\bE e_1|_2$ is the norm of a standard Gaussian vector of dimension $p$, the expectation  $\E[|\bE e_1|_2]$ is that of the norm of a standard Gaussian vector. Relying on deviation inequalities for $\chi^2$ random variables, we deduce that 
\[
\E[|\bE e_1|_2]\geq \sqrt{\frac{p}{3}}\P[|\bE e_1|^2_2\geq p/3]\geq  c\sqrt{p}(1- e^{-p/32})\geq c' \sqrt{p} \ , 
\]
for a constant $c>0$. 
\end{proof}

\begin{proof}[Proof of Proposition~\ref{prp:control_operator}]
It follows from the triangular inequality that 
\[
\|\bY^T\bY\|_{\infty} - p - \|\bA^T\bA\|_{\infty} \leq 2\|\bA^T\bE\|_{\infty} + \|\bE^T\bE\|_{\infty} - p\ .
\]
Conversely, we have 
\beqn 
\|\bY^T\bY\|_{\infty}&=& \sup_{u, |u|_2=1} \left(u^T \bA^T\bA u+ 2u^T \bA^T \bE u +  u^T\bE^T\bE u \right)\\ & \geq&  \sup_{u, |u|_2=1} \left(u^T \bA^T\bA u\right) - 2\|\bA^T \bE\|_{\infty} + \sigma_q(\bE^T \bE)\\
&\geq & \|\bA^T \bA \|_{\infty}- 2\|\bA^T \bE\|_{\infty} + \sigma_q(\bE^T \bE)\ . 
\eeqn 
This allows us to bound the error
\begin{eqnarray} \nonumber
 \big|\|\bY^T\bY\|_{\infty} - p - \|\bA^T\bA\|_{\infty}\big| &\leq& 2\|\bA^T\bE\|_{\infty} + [\sigma_1(\bE^T\bE) - p]\vee [\sigma_q(\bE^T\bE) - p] 
\\
&\leq &  2\|\bA^T\bE\|_{\infty}+ \|\bE^T\bE - p\bI\|_{\infty}\ .  \label{eq:upper_error_operator}
\end{eqnarray}
Then, the first risk bound is a straightforward consequence of the two following lemmas whose proof is given below.
\begin{lem}\label{croise}
For all matrices $\bA$, we have  $\mathbb{E} \cro{\|\bA^T\bE\|_{\infty}} \leq 2\|\bA\|_{\infty}\sqrt{q}$. 
\end{lem}
\begin{lem}\label{lem:noiseop}
$\mathbb{E} \cro{\|\bE^T\bE - p\bI\|_{\infty}} \leq 3\sqrt{pq} + 8e^{-p/4} + \sqrt{64\pi p}.$
\end{lem}

As in the proof of Proposition \ref{prp:risk_frobenius}, we deal differently with small and large values of $\|\bA\|_{\infty}$. Since $\sqrt{x}-\sqrt{y}\leq \sqrt{|x-y|}$, we obtain by Cauchy-Schwarz inequality that 
\beqn 
\mathbb{E}\cro{|(\|\bY^T\bY\|_{\infty}-p)_+^{1/2} - \|\bA\|_{\infty}|} & = &\mathbb{E}\cro{|\sqrt{(\|\bY^T \bY\|_{\infty}-p)_+} - \sqrt{\|\bA\|_{\infty}^2}|} \\
 &\leq& \mathbb{E}\cro{\sqrt{|(\|\bY^T \bY\|_{\infty}-p)_+ - \|\bA^T\bA\|_{\infty}|}} \\
&\leq& \pa{\mathbb{E}\cro{|\|\bY^T\bY\|_{\infty} -p  -\|\bA^T\bA\|_{\infty}|}}^{\frac{1}{2}}\\
&\leq &\pa{3\sqrt{pq} + 4\|\bA\|_{\infty}\sqrt{q} + c'(\sqrt{p}+ 1)}^{\frac{1}{2}} 
\eeqn 
If we assume that $\|\bA\|_{\infty}\leq (pq)^{1/4}$, this leads us to 
\[
\mathbb{E}\cro{|(\|\bY^T\bY\|_{\infty}-p)_+^{1/2} - \|\bA\|_{\infty}|}\leq  c(pq)^{1/4} \ .
\]
Now assume that  $\|\bA\|_{\infty}\geq (pq)^{1/4}$, which is equivalent to $\|\bA^T\bA\|_{\infty}\geq \sqrt{pq}$. Since $|x-y|\leq |x^2-y^2|/|y|$, it follows that 
\beqn 
\mathbb{E}\cro{|(\|\bY^T \bY\|_{\infty}-p)_+^{1/2} - \|\bA\|_{\infty}|}&\leq &\frac{1}{\|\bA\|_{\infty}}\mathbb{E}\cro{|(\|\bY^T \bY\|_{\infty}-p)_+ - \|\bA\|_{\infty}^2|}\\
&\leq & 3(pq)^{1/4}+ 4\sqrt{q}+ c'(p/q)^{1/4}\lesssim (pq)^{1/4}\ , 
\eeqn 
which concludes the proof. 
\end{proof}

\begin{proof}[Proof of lemma \ref{croise}]
Let $\bA= \bU^T \bD \bV$ denote a singular value decomposition of $\bA$. Since the distribution of $\bE$ is invariant by left and right orthogonal transformation, it follows that $\|\bA\bE\|_{\infty}$ follows the same distribution as $\|\bV\bD^T  \bE\|_{\infty}= \|\bD^T\bE\|_{\infty}$. We shall bound the expectation of this last random variable. Since $\bD^T$ is a $q\times p$ diagonal matrix, $\bD^T\bE$ does not depend on the entries $\bE_{ij}$ with $i\geq q$. Write $\overline{\bD}$ and $\overline{\bE}$ for the restriction of $\bD$ and $\bE$ to their $q$ first rows. We obtain 
\beqn
\E\cro{\|\bA^T\bE\|_{\infty}} & =  & \E\cro{\|\overline{\bD}\overline{\bE}\|_{\infty}}\leq \|\overline{\bD}\|_{\infty}\E\cro{\|\overline{\bE}\|_{\infty}}\\
&\leq & 2\|\bA\|_{\infty}\sqrt{q}\ , 
\eeqn 
by Lemma \ref{lem:spectrum} and since  $\|\overline{\bD}\|_{\infty}= \|\bD\|_{\infty}= \|\bA\|_{\infty}$. 

\end{proof}

\begin{proof}[Proof of lemma \ref{lem:noiseop}]
Observe that $\|\bE^T\bE-p\bI\|_{\infty}= \max(\sigma^2_1(\bE)- p,p - \sigma_q^2(\bE))$.  We deduce from  Lemma~\ref{lem:DavSza} (taken from~\cite{Davidson2001}) that, for any $t>0$, with probability higher than $1-2e^{-t}$ we have
\beqn 
\sigma_1(\bE) \leq \sqrt{p}+\sqrt{q}+ \sqrt{2t} \ ;  \sigma_q(\bE) \leq \sqrt{p}-\sqrt{q}-  \sqrt{2t}\ . 
\eeqn 
Hence, with probability higher than $1-2e^{-t}$, we have 
$ \|\bE^T\bE-p\bI\|_{\infty}\leq q +  2\sqrt{pq}+ 2t + 4\sqrt{2pt}  $, which implies that 
\[
 \P\left[\|\bE^T\bE-p\bI\|_{\infty}-q - 2\sqrt{pq}\geq x \right]\leq 2\exp\left[-\frac{1}{4}\left(x \wedge \frac{x^2}{16p}\right)\right]\ . 
\]
for any $x>0$. Integrating this inequality we conclude that 
\[
 \E\left[\|\bE^T\bE-p\bI\|_{\infty}\right]\leq 3\sqrt{pq} + \int_{\R} e^{-x^2/(64p)}+  2\int_{16p }^{\infty}e^{-x/4}= 3\sqrt{pq} + 8e^{-4p}+ \sqrt{64\pi p }\ . 
\]

\end{proof}

\subsection{Minimax lower bound}

\begin{proof}[Proof of Proposition~\ref{prp:lower_frob}]
We follow the general Le Cam's approach of fuzzy hypotheses as may be found in~\cite[][Sec.2.7.4]{tsybakov_book}. Given a matrix $\bA$, we denote in the section of $\P_{\bA}$ for the distribution of $\bY$. Consider a prior distribution $\mu$ on $\bA$, such that, for some $s>0$, $\|\bA\|_2\geq s$, $\mu$ almost surely. Define the set $\Theta$ as the union of the support of $\mu$ and the null matrix $0$.  Denoting $\mathbf{P}= \int_{\bA}\P_{\bA}\mu(d\bA)$ for the integrated distribution, it follows from Theorem 2.15 in~\cite{tsybakov_book} that
\[
 \inf_{\widehat{T}}\sup_{\bA\in \Theta}\mathbb{P}_{\bA}[|\widehat{T}-\|\bA\|_2|\geq s/2]\geq \frac{1-\sqrt{\chi^2(\P_{0},\mathbf{P})/2}}{2}\ .
\]
If we further assume that $\Theta$ only contains matrices of $\rank$ less or equal to one, this implies that 
\beq\label{eq:lower_general}
 \inf_{\widehat{T}}\sup_{\bA, \ \rank(\bA)\leq 1}\E_{\bA}[|\widehat{T}-\|\bA\|_2|]\geq \frac{s}{4}\left[1-\sqrt{\chi^2(\P_{0},\mathbf{P})/2}\right]\ . 
 \eeq
It remains to choose $\mu$ and bound the corresponding $\chi^2$ distance. 

Let $s>0$ be a positive quantity that will be fixed later.  Let $\nu_1$ be the uniform distribution over the vectors of the form $p^{-1/2}(\eta_1,\ldots, \eta_p)$ where the $\eta_i$'s belong to $\{-1,1\}$. Let $\nu_2$ be corresponding distribution where $q$ is replaced by $p$. Finally, let $\mu$ be the distribution of $\bA= s u v^T$ where $u$ and $v$ are sampled independently from $\nu_1$ and $\nu_2$. By construction, $\mu$ almost surely, $\bA$ is a rank-one matrix and $\|\bA\|_2= s$.

\medskip 

The main part of the proof amounts to upper bounding the $\chi^2$ discrepancy between $\P_{0}$ and $\mathbf{P}$. Writing $L= d\mathbf{P}/\P_0$ the likelihood ratio between $\mathbf{P}$ and $\P_0$, we have by definition~\cite{tsybakov_book}  that
\beq\label{eq:chi^2}
\chi^2(\P_{0},\mathbf{P})= \E_0[(L-1)^2]= \E_0[L^2]-1\ . 
\eeq

\begin{lem}\label{lem:controlL2}
Taking $s=(pq/4)^{1/4}$, we have $\mathbb{E}_0\cro{L^2}\leq \frac{5}{3}$. 
\end{lem} 

Equipped with this choice of $s$, we have $\chi^2(\P_{0},\mathbf{P})\leq 2/3$ and we conclude thanks to~\eqref{eq:lower_general} that 
\[
  \inf_{\widehat{T}}\sup_{\bA, \ \rank(\bA)\leq 1}\E_{\bA}[|\widehat{T}-\|\bA\|_2|]\geq \frac{(pq/4)^{1/4}}{4}(1-\sqrt{1/3})\geq \frac{(pq)^{1/4}}{20}\ . 
\]

\end{proof}

\begin{proof}[Proof of Lemma \ref{lem:controlL2}]
In this proof, we write, $\langle , \rangle$ for the inner product associated to the the Frobenius norm. 
First, we work out the likelihood ratio $L$. Since the density of $\P_{\bA}$ with respect to  the Lebesgue measure is $(2\pi)^{-(pq)/2}e^{-\|\bY-\bA\|_2^2/2}$, it follows that 
\beqn 
 L &=& \int \exp\left[-\frac{1}{2}\|\mathbf{\bY} - \bA\|_2^2 + \|\mathbf{\bY}\|_2^2\right] \mu(d\bA)\\
 & =&  \int \exp\left[-\frac{1}{2}\|\bA\|^2_2 + \langle \bY,\bA \rangle \right]\mu(d\bA)\ .
\eeqn
As a consequence, the second moment of the likelihood writes as 
\beqn
\E_0[L^2] &= &\E_0\left[\int\int \exp\left[-\frac{1}{2}\|\bA_1\|^2_2- \frac{1}{2}\|\bA_2\|^2_2 + \langle \bY,(\bA_1+\bA_2) \rangle \right]\right]\\
&= &\int\int\E_0\left[\exp\left[-\frac{1}{2}\|\bA_1\|^2_2- \frac{1}{2}\|\bA_2\|_2^2 + \langle \bY,(\bA_1+\bA_2) \rangle \right]\right] \mu(d\bA_1)\mu(d\bA_2)\\
&= & \int\int \exp\left[-\frac{1}{2}\|\bA_1\|_2^2- \frac{1}{2}\|\bA_2\|_2^2+ \|\bA_1+\bA_2\|^2/2\right] \mu(d\bA_1)\mu(d\bA_2)\\
& =& \int\int \exp\left[\langle \bA_1 ,\bA_2\rangle \right] \mu(d\bA_1)\mu(d\bA_2)\ , 
\eeqn 
where we used the Laplace transform of the normal random distribution in the third line. In view of the definition of $\mu$, this integral further decomposes as 
\beqn
 \E_0[L^2]&=& \int\int \exp\left[s^2 \tr[v_1 u_1^T u_2 v_2^T \right] \nu_1(du_1)\nu_1(d u_2)\nu_2(dv_1)\nu_1(d v_2)\\
 &=& \int\int \exp\left[s^2 (v_1^T v_2) (u_1^T u_2) \right] \nu_1(du_1)\nu_1(d u_2)\nu_2(dv_1)\nu_1(d v_2)\ . 
\eeqn 
Since $Z_1= pu_1^Tu_2$ and $Z_2= qv_1^T v_2$ are respectively distributed as sums of $p$ and $q$ independent Rademacher random variables, arrive at 
\[
 \E_0[L^2]= \E\left[\exp\left(\frac{s^2Z_1Z_2}{pq} \right)\right]
\]
For $a\in \mathbb{R}$ and $X$ a Rademacher random variable, we have $\E[e^{aX}]= \cosh(ax)\leq e^{a^2x^2/2}$ (compare the power series). We obtain by integration with respect to $Z_1$, 
\[
 \E_0[L^2]= \E\left[\cosh\left(\frac{s^2 Z_2}{pq}\right)^p \right]\leq \E\left[\exp\left(\frac{s^4Z_2^2}{2pq^2}\right)\right]= \E\left[\exp\left(\frac{1}{8q}Z_2^2\right)\right]\ ,
\]
since we fixed $s=(pq/4)^{1/4}$. Since $Z_2$ is a sum of Rademacher random variables, we can apply Hoeffding's inequality~\cite{book_concentration}, which leads us to $\mathbb{P}\pa{|Z_2| \geq u} \leq 2e^{\frac{-u^2}{2q}}$. As a consequence, 
\beqn 
\E_0[L^2] &\leq & \E\left[\exp\left(\frac{1}{8q}Z_2^2\right)\right] = \int_{0}^{\infty} \mathbb{P}\left[\exp\left(\frac{Z_2^2}{8q}\right)\geq t\right] dt\\
& \leq& 1 + \int_{1}^{\infty} \mathbb{P}\left[\exp\left(\frac{Z_2^2}{8q}\right)\geq t\right] dt\\
&\leq & 1+ 2\int_{1}^{\infty}\frac{1}{t^4}dt = 5/3\ . 
\eeqn
The result follows.
\end{proof}

\subsection{Proof for general even norms $\|\bA\|_{2k}$}
We start with a lemma summarizing important properties of the Hermite polynomials~\eqref{eq:definition_hermite}. It is a slight variation of Lemma 3 in~\cite{cailow2011}. 
\begin{lem} \label{lem:hermite}
Let $Z\sim \cN(x,1)$. Then,  $\E[H^2_r(Z)]\leq [e(x^2\vee r)]^r$ and $\var{H_r(Z)}\leq er[e(x^2\vee r)]^{r-1}$. Besides, 
\beq\label{eq:orthogonal_hermite}
 \E[H^2_r(Z)]= r!\ , \text{ and } \E[H_r(Z)H_{l}(Z)]=0\ , \text{ for }r\neq l\ . 
\eeq

\end{lem}
\begin{proof}[Proof of Lemma~\ref{lem:hermite}]
The identity~\eqref{eq:orthogonal_hermite} is classical. The bound $\E[H^2_r(Z)]\leq [e(x^2\vee r)]^r$ is taken from Lemma 3 in~\cite{cailow2011}. Since $\var{H_r(Z)}\leq \E[H^2_r(Z)]$, we only have to prove that, 
for $x^2> r$, $\E[H_r^2(Z)]\leq x^{2r}+ r(ex^2)^{r-1}$. We know from the proof of Lemma 3 in~\cite{cailow2011} that 
\beqn
 \E[H_r^2(Z)]-x^{2k}&=& r!\sum_{j=0}^{r-1} \binom{r}{j} \frac{x^{2j} }{j!}
 \leq r!\sum_{j=0}^{r-1} \binom{r}{j} \frac{x^{2(r-1)} }{(r-1)!}\leq rx^{2r-2}2^{r}\ . 
\eeqn
The result follows.
\end{proof}

\begin{proof}[Proof of Theorem \ref{prp:risk_U_k}]

Let us upper bound $\Var({U_k})= \E[\left(U_k- \sum_{i=1}^q\sigma_i^{2k}(\bA)\right)^2]$. Recall that we assume without loss of generality that $\bA$ is of singular value form, that is
$\bA_{rs}= \sigma_{r}(\bA)\1_{r=s}$. Given two sequences $i=(i_1,\ldots,i_k)\in [p]^k$ and $j=(j_1,\ldots j_k)\in [q]^k$ we write $W_{ij}= \prod_{r,s} H_{N_{rs}(i,j)}(\bY_{rs})$. 
Relying on the Hermite decomposition of $U_k$, we observe that 
 \begin{equation}\label{eq:definition_T_ij}
 \var{U_{k}}= \sum_{i,j}\sum_{i',j'}\underset{= T_{iji'j'}}{\underbrace{\E\left[\left[ W_{ij}- \E\left(W_{ij}\right)\right]\left[ W_{i'j'}- \E\left(W_{i'j'}\right)\right]\right]}}\ . 
 \end{equation}
 In the remainder of this proof, we use combinatorial arguments to count the number of such non-zero $T_{iji'j'}$ and to bound each of these non-zero terms.  
 The expectation  $\E\left(W_{ij}\right)$ is non-zero if and only if $N_{rs}(i,j)= 0$ for all $r\neq s$.  This is possible only if $i_1=i_2=\ldots=i_k=j_1=\ldots = j_k$. As a consequence,  only terms of the form $H_{2k}(Y_{rr})$ have a non-zero expectation.  To control $T_{iji'j'}$, we consider two cases depending on the nullity of  $\E[W_{ij}]$.
  
 \noindent 
 {\bf Case 1}: $\E[W_{ij}]\neq 0$. It follows from the above discussion that, for some $r\in [q]$, $\E[W_{ij}]= \mathbb{E}[H_{2k}(Y_{rr})] = \sigma_{r}^{2k}(\bA)$. Then, $T_{iji'j'}$ is non-zero only if $r$ occurs in both sequences $i'$ and $j'$ . If $i'\neq i$ or if $j'\neq j$, this implies that $(i',j')$ is not constant so that $\E[W_{i'j'}]=0$. Besides, if $i'\neq i$ or if $j'\neq j$,   there exist $s_1\neq s_2$ such that $N_{s_1s_2}(i',j')>0$. Since $\bA_{s_1s_2}=0$ and  $N_{s_1s_2}(i,j)=0$ this implies  that $T_{iji'j'}=0$ by independence of the noise components. If $(i'=i)$ and $(j'=j)$, then Lemma~\ref{lem:hermite} ensures that   
 \beq\label{eq:upper_T0}
 T_{iji'j'}= \Var{H_{2k}(\bY_{rr})}\leq  (2ke)[e(\sigma^2_{r}(\bA)\vee 2k)]^{2k-1}\ . 
 \eeq
In summary, if $\E[W_{ij}]\neq 0$ or $\E[W_{i'j'}]\neq 0$, we have $T_{iji'j'}=0$ unless $(i=i')$ and $(j'=j)$. In the latter case, $T_{iji'j'}$ satisfies~\eqref{eq:upper_T0}.

 \noindent 
 {\bf Case 2}: $\E\left(W_{ij}\right)=\E\left(W_{i'j'}\right)= 0$. Then, $T_{iji'j'}$ simplifies as 
 \[
  T_{iji'j'}= \E[W_{ij}W_{i'j'}]  = \prod_{r=1}^{p}\prod_{s=1}^q \E[H_{N_{rr}(i,j)}(\bY_{rs})H_{N_{rr}(i,j)}(\bY_{rs})]  
 \]
Since $\E[\bY_{rs}]=\bA_{rs}=0$ for any $r\neq s$, it follows from Lemma~\ref{lem:hermite} that $\E[H_{l}(\bY_{rs})H_{l'}(\bY_{rs})]= l! \1_{l=l'}$ for $r\neq s$. We deduce that $T_{iji'j'}\neq 0$ only if $N_{rs}(ij)= N_{rs}(i'j')$ for all $r\neq s$. By definition of $N_{rs}$, we have $\sum_{r=1}^p\sum_{s=1}^qN_{rs}(ij)= \sum_{r=1}^p\sum_{s=1}^qN_{rs}(i'j')=2k$. Together with the previous property, we deduce $T_{iji'j'}\neq 0$ only if $\sum_{r}N_{rr}(ij)= \sum_{r}N_{rr}(i'j')$.
As explained in the proof of Proposition~\ref{prp:representation_somme_newton} below, we also have $\sum_{s=1}^qN_{rs}(ij)\equiv 0[2]$ for each $r\in[p]$, because each entry in the sequence $i$  appears twice in the monomial $W_{ij}$. Hence, $\E[T_{iji'j'}]\neq  0$ implies that, for each $r\in [q]$ 
\[
  N_{rr}(ij)+ N_{rr}(i'j') = \underset{\equiv 0[2]}{\underbrace{\sum_{s=1}^q N_{rs}(ij)+ N_{rs}(i'j')}} - \sum_{s=1, \ s\neq r}^q \underset{\equiv 0[2]}{\underbrace{N_{rs}(ij)+ N_{rs}(i'j')}}\enspace ,
\]
which implies that $N_{rr}(ij)+ N_{rr}(i'j')\equiv 0[2]$. 
Furthermore, if for some $r\in [q]$,  $N_{rr}(i'j')>0$, then there exists either $s\neq r\in [q]$ such that $N_{rs}(i'j')>0$ or $s\neq r\in [p]$ such that $N_{rs}(i'j')>0$ --otherwise $\E[W_{i'j'}]\neq 0$. Hence, we have $T_{iji'j'}\neq 0$ unless either $N_{rs}(ij)>0$ for some $s\in [q]$  or $N_{sr}(ij)>0$ for some $s\in [p]$. Let us summarize our findings. When  $\E\left(W_{ij}\right)=\E\left(W_{i'j'}\right)= 0$, we have $T_{iji'j'}\neq  0$ only if the four following properties are satisfied
%
%
%
%
 \begin{itemize} \label{prop:a_b_c_d}
  \item[(a)] $N_{rs}(ij)=  N_{rs}(i'j')$ for all $(r,s)\in [p]\times [q]$ such that $r\neq s$.
  \item[(b)] $\sum_{r=1}^qN_{rr}(ij)= \sum_{r=1}^qN_{rr}(i'j')$. 
  \item[(c)] $N_{rr}(ij)+N_{rr}(i'j')\equiv 0  [2]$ for all  $r\in [q]$.
  \item[(d)] For any $l=1,\ldots, k$, there exists $l'$ such that either  $i_{l}=i'_{l'}$ or $i_{l}=j'_{l'}$. Similarly, for any $l=1,\ldots, k$, there exists $l'$ such that either  $j_{l}=i'_{l'}$ or $j_{l}=j'_{l'}$.
 \end{itemize}
The last property does not ensure that $i$ and $i'$ take same value, but if $i$ takes some value that does not appear in $i'$, then this value must appear in $j'$. When $T_{iji'j'}\neq 0$, then Lemma~\ref{lem:hermite} yields 
 \begin{eqnarray}
T_{iji'j'}&=& \prod_{rs} \E\left[H_{N_{rs}(ij)}(\bY_{rs})H_{N_{rs}(i'j')}(\bY_{rs})\right] \nonumber \\ 
&\leq &\prod_{r\neq  s}[N_{rs}(ij)]! \prod_{r=1}^q \left[e^{1/2}\left(\sigma_{r}(\bA)\vee \sqrt{N_{rr}(ij)\vee N_{rr}(i'j')}\right)\right]^{N_{rr}(ij)+N_{rr}(i'j') } \nonumber \\
&\leq & \prod_{r\neq  s}[N_{rs}(ij)]! \prod_{r=1}^q \left[e^{1/2}\left(\sigma_{r}(\bA)\vee \sqrt{2k}\right)\right]^{N_{rr}(ij)+N_{rr}(i'j') }
\label{eq:upper_T}\ , 
   \end{eqnarray}
   by Cauchy-Schwarz inequality.

   \medskip 
  Based on the above observations, we shall now bound $\var{U_k}$ by suitably partitioning the sequences $(i,j)$ and $(i',j')$. Let us introduce some notation. 
  Given a $k$-tuple  $i$, we introduce the $2k$-tuples $i_{ev}= (i_1,i_1,i_2,i_2,\ldots, i_k,i_k)$ and $i_{od}=(i_1,i_2,i_2,\ldots, i_k,i_k,i_1)$. Given two $k$-tuples $i$ and $j$, we write $i\oplus j\in (\mathbb{N}^2)^{2k}$ by 
 \beq\label{eq:definition_i+j}
   i\oplus j= ((i_1,j_1), (i_1,j_2),(i_2,j_2),\ldots, (i_k,j_k), (i_k,j_1))\enspace .
\eeq
Finally, we define $S^{(id)}[i, j]=\cup_{r=1}^q\{(i\oplus j)^{-1}[\{(r,r)\}]\}$ as the collection of indices $t$ such that $(i\oplus j)_t= (r,r)$ for some $r\in [q]$. 
For any two subsets $S$ and $S'$ of $[2k]$, we consider $R_{S,S'}$ as the sum of $T_{iji'j'}$ over all sequences $(i,j)$ and $(i',j')$ satisfying $S^{(id)}[i, j]=S$ and $S^{(id)}[i', j']=S'$. Since, for such sequences $(i,j)$, we have $\sum_{r=1}^qN_{rr}(ij)=|S|$ and $\sum_{r=1}^qN_{rr}(i'j')=|S'|$, it follows from Property (b) above that 
$R_{S,S'}=0$ if $|S|\neq |S'|$. This leads us to the following decomposition
\[
 \var{U_k}= \sum_{s=0}^{k}\quad \sum_{S,S'\text{ s.t.} |S|=|S'|=s} R_{S,S'}
\]
Observe that no sequence $(i,j)$ satisfies $|S^{(id)}[i, j]|=k-1$. Besides, in view of Case 1 above and \eqref{eq:upper_T0}, we have 
\[
 R_{[k],[k]}\leq 2ke\sum_{u=1}^q[e(\sigma^2_{r}(\bA)\vee 2k)]^{2k-1}\ . 
\]
Denoting $\bA^+$ the modification of $\bA$ where each singular value has been replaced by $\sigma_r(\bA)\vee \sqrt{2k}$, we deduce that 
\beq \label{eq:upper_var_Uk2}
\var{U_k}\leq  \sum_{s=0}^{k-2}\quad \sum_{S,S'\text{ s.t.} |S|=|S'|=s} R_{S,S'}+ 2ke^{2k}\|\bA^{+}\|_{2k-1}^{2k-1}\ . 
\eeq

\begin{figure}[h]
  \centering 
  \subfloat[Sequence $(i,j)$]{{\includegraphics[width=6.5cm]{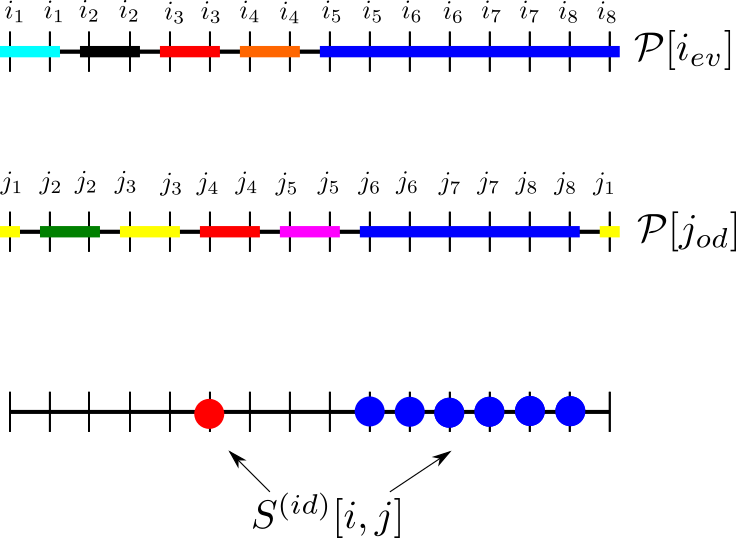}}}\hspace{2cm}
  \subfloat[Sequence $(i',j')$]{{\includegraphics[width=6.5cm]{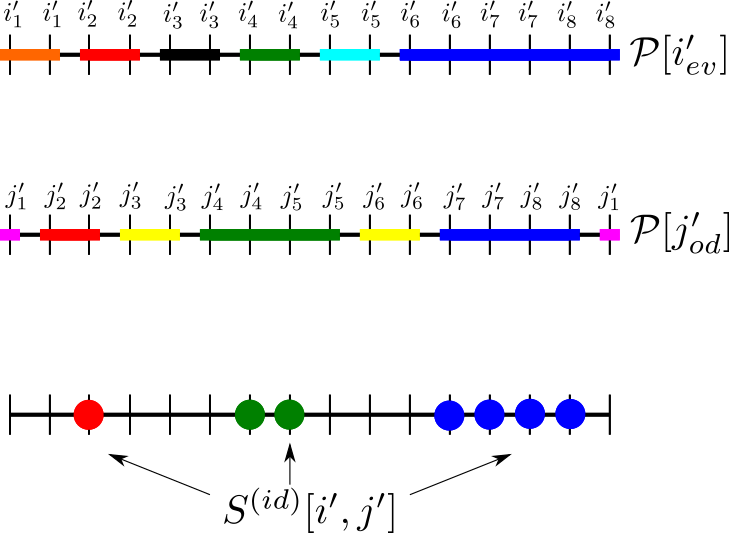}}}
  \caption{Examples of sequences two sequences $(i,j)$ and $(i',j')$ such that $T_{iji'j'}\neq 0$. Identical values for $i$, $i'$, $j$, and $j'$ are depicted using the same colour. 
  The first rows stands for $i_{ev}$ (resp. $i'_{ev}$) and the corresponding partition $\cP[i_{ev}]$ (resp. $\cP[i'_{ev}]$). The second rows stands for $j_{od}$ (resp. $j'_{od}$) and the corresponding partition $\cP[j_{od}]$ (resp. $\cP[j'_{od}]$). The last row depicts the sets $S^{(id)}[i,j]$ and $S^{(id)}[i',j']$.
  \label{fig:proof_partition_graph} }    

\end{figure}

Consider two sequences $(i,j)$ and $(i',j')$ such that $S^{(id)}[i, j]=S$, $S^{(id)}[i', j']=S'$, and $T_{iji'j'}\neq 0$. Since $N_{rs}(ij)= N_{rs}(i',j')$ for any $r\neq s$, it follows that the restriction of $(i'\oplus j')$ to $(S^{'})^c$ is a permutation of the restriction of $(i\oplus j)$ to $S^c$.  Define $\cP[i_{ev}]$ (resp. $\cP[j_{od}]$) the partition of $[2k]$ corresponding to identical values of $i_{ev}$ (resp. $j_{od}$).  There exist unique maps $F_1:([2k]\setminus S')\mapsto \cP[i_{ev}]$ and $F_2:([2k]\setminus S')\mapsto \cP[j_{od}]$ such that, for any $s\notin S'$, $(i'_{ev})_s= (i_{ev})_t$ for all $t$ in $F_1(s)$ and $(j'_{od})_s= (j_{od})_t$ for all $t$ in $F_2(s)$. As a consequence of these definitions, given $(S,S',i,j,\cP[i_{ev}],\cP[j_{od}],F_1,F_2)$, the  values of $(i'\oplus j')$ outside $S'$ are uniquely defined. Furthermore, it follows from the definition~\eqref{eq:definition_i+j} of $(i'\oplus j')$ and of $S^{(id)}[i',j']$ that $(i'\oplus j')$ is uniquely defined by its values outside $S'=S^{(id)}[i',j']$, unless $S'=[2k]$. Indeed, if $t\in [2k]\in S'$ is even, then one has 
\[
  (i'\oplus j')_t=(i'_{t/2},i'_{t/2})= \big(((i'\oplus j')_{t-1})_1,((i'\oplus j')_{t-1})_1\big)\enspace  .
\]
Besides, if $t\in [2k]\in S'$ is odd, then one easily checks
\[
  (i'\oplus j')_t=(i'_{(j'+1)/2},j'_{(t+1)/2}) = \big(((i'\oplus j')_{t-1})_2,((i'\oplus j')_{t-1})_2\big)\enspace .  
\] 
with the convention that $t-1=2k$ when $t=0$. As a consequence, for $t\in S'$, $(i'\oplus j')_t$ is uniquely defined by $(i'\oplus j')_{t-1}$ and, unless $(S')^c=\emptyset$, the vector $(i'\oplus j')$ is uniquely defined by its restriction to $S'^c$. In summary, given $(S,S',i,j,\cP[i_{ev}],\cP[j_{od}],F_1,F_2)$
such that $S= S^{(id)}[i,j]$, $\cP_1= \cP[i_{ev}]$ and $\cP_2=\cP[i_{od}]$. 
There exists at most one $2$-tuple of sequences $(i',j')$ satisfying $S'= S^{(id)}[i',j']$, and
\begin{eqnarray}\label{eq:prop_1_i'}
  (i'_{ev})_s= (i_{ev})_t \text{ for any } s\in S'^c\text{ and } t\in F_1(s)\enspace ;\\
  (j'_{od})_s= (j_{od})_t \text{ for any } s\in S'^c\text{ and } t\in F_2(s)\enspace  .\label{eq:prop_1_j'}
\end{eqnarray}
Since all terms $ T_{iji'j'}$ are non-negative, this allows us to further decompose the expression $R_{S,S'}$
\begin{eqnarray}\label{eq:definition_RSS'} 
R_{S,S'}\leq  \sum_{\cP_1,\cP_2,F_1,F_2} R_{S,S';\cP_1,\cP_2;F_1,F_2}\ ,\quad \text{ where }   \quad R_{S,S';\cP_1,\cP_2;F_1,F_2} = 
 \sum_{\substack{(i,j),\, S^{id}[i,j]=S \\   \cP[i_{ev}]= \cP_1\ ,\, \cP[j_{od}]= \cP_2}} T_{iji'j'}\ , 
\end{eqnarray}
 where, in the rhs,  $(i',j')$ is the unique (if it exists) $2$-tuple of sequences satisfying $S'= S^{(id)}[i',j']$, \eqref{eq:prop_1_i'}, and \eqref{eq:prop_1_j'}.

 Fix any $S$, $S'$,  $\cP_1$, $\cP_2$, $F_1$, and $F_2$ and let us upper bound the expression $T_{iji'j'}$  for sequences $(i,j,i',j')$ satisfying the constraints in~\eqref{eq:definition_RSS'}. Define 
 \beq\label{eq:definition_gamma}
 \gamma= \sum_{P\in \cP_1} \1_{P\cap S\neq \emptyset}= \sum_{P\in \cP_2} \1_{P\cap S\neq \emptyset}\ ,
 \eeq
 as the number of groups of $\cP_1$ (or equivalently $\cP_2$) that intersect $S$. 
 For any any sequence $(i,j)$ satisfying $S= S^{(id)}[i,j]$, $\cP_1= \cP[i_{ev}]$, and $\cP_2= \cP[j_{od}]$, one easily checks that $\gamma= \sum_{r=1}^q \1_{N_{rr}(ij)>0}$ and therefore also corresponds to the number of distinct values of $(i\oplus j)$  on $S^{(id)}[i,j]$. For instance, in Figure~\ref{fig:proof_partition_graph}, we have $\gamma=2$ for this specific value $\cP[i_{ev}]$  . 

\begin{lem}\label{lem:number_degrees}
Consider any $S$, $\cP_1$ and $\cP_2$. For any sequences $(i,j)$  such that $\cP[i_{ev}]=\cP_1$ and $\cP[j_{od}]=\cP_2$, we have
\[
 \prod_{r\neq s}[N_{rs}(ij)]!\leq (2k)^{2k-|S|-|\cP_1|-|\cP_2|+\gamma+1}\ ,
\]
 \end{lem}

From \eqref{eq:upper_T} and Lemma~\ref{lem:number_degrees} we deduce that 
\beq\label{eq:upper_T_lem}
T_{iji'j'}
\leq (2k)^{2k-|S|-|P_1|-|P_2|+\gamma+1} \prod_{r=1}^q\left[e^{1/2}\sigma_r(\bA^+)\right]^{N_{rr}(ij)+ N_{rr}(i'j')}\enspace ,
\eeq
where we recall that  $\sigma_r(\bA^+)= \sigma_r(\bA)\vee \sqrt{2k}$. To work out the quantity $R_{S,S';\cP_1,\cP_2;F_1,F_2}$, we need to introduce a few additional notation. Since  $\cP_1\neq \{[2k]\}$ and $\cP_2\neq \{[2k]\}$, one easily checks that no element of $\cP_1$ and $\cP_2$ are equal. Indeed, any $P$ in $\cP_1$ is an union of sets of the form $\{2t, 2t+1\}$ whereas any $P\in \cP_2$ in an union of sets of the form $\{2t+1,2t+2\}$. Hence $P\in \cP_1\cap \cP_2$ implies that $P=[2k]$. As a consequence, one can unambiguously define the collection $\cP_1\cup \cP_2$ of subsets of $[2k]$. Then, we can extend the maps $F_1: [2k]\setminus S' \rightarrow \cP_1$ and $F_2: [2k]\setminus S' \rightarrow \cP_2$ as maps $F_{1,ext}: [2k]\rightarrow \cP_1\cup \cP_2$ and $F_{2,ext}: [2k]\rightarrow \cP_1\cup \cP_2$ as follows. Recall that the values of $(i'\oplus j')$ on $S'$ are uniquely defined by its values outside $S'$. Besides, for $t\in S'$, $(i'\oplus j')_{t}$ is of the form $(r,r)$ where $r$ is either equal to $i_{t'}$ for some $t'$ or $j_{t'}$ for some $t'$ -- see Property (d) in Page~\pageref{prop:a_b_c_d}. Hence, we define $F_{1,ext}(t)$ as the unique $P\in \cP_1$ (if there exists one) such that for $t'\in P$, $(i'_{ev})_{t}=(i_ev)_{t'}$. In case there is no such $P\in \cP_1$, we define $F_{1,ext}(t)$ as the unique $P\in \cP_2$  such that for $t'\in P$, $(i'_{ev})_{t}=(j_{od})_{t'}$. In Figure~\ref{fig:proof_partition_graph}, we observe for instance that $F_{1,ext}(7)= F_{1,ext}(8)$ is a group of $\cP[j_{od}]$. 
We define similarly $F_{2,ext}$ be reversing the role of $\cP_2$ and $\cP_1$. One easily checks that the definitions of $F_{1,ext}$ and $F_{2,ext}$ only depend on $S$, $S'$, $\cP_1$, $\cP_2$, $F_1$ and $F_2$ but not on the specific values of $i$ and $j$.

Finally, we introduce $\gamma'$, $\gamma'_1$, and $\gamma'_2$. First, $\gamma'$ corresponds to the number values of the form $(r,r)$ that occur in $(i'\oplus j')$ but not in $(i\oplus j)$. This quantity does not depend on the specific values of $(i,j)$ and can equivalently be defined by 
\[
  \gamma' = \sum_{P\in F_{1,ext}(S')} \1_{P\cap S=\emptyset}= \sum_{P\in F_{2,ext}(S')}\1_{P\cap S=\emptyset}\enspace .  
\]
Then, $\gamma'_1$ (resp. $\gamma'_2$) is the number of groups $P\in \cP_1$ that do not intersect $S$ but whose inverse image through $F_{1,ext}$ (resp. $F_{2,ext}$) intersects $S'$. Equivalently, $\gamma'_1$ is the number of distinct values $r$ of $i'_{ev}$ such that $N_{rr}(i'j')>0$, $N_{rr}(ij)=0$, and  $N_{rs}(ij)>0$ for some $s\neq r$. 
\[
  \gamma'_1 = \sum_{P\in \cP_1}\1_{P\cap S=\emptyset}\1_{(F_{1,ext}^{-1}(P))\cap S'\neq \emptyset}\ ; \quad   \gamma'_2 = \sum_{P\in \cP_2}\1_{P\cap S=\emptyset}\1_{(F_{2,ext}^{-1}(P))\cap S'\neq \emptyset}\enspace .
\]
One easily checks that $\gamma'\leq \gamma'_1+\gamma'_2$. In fact, we even have $\gamma'=\gamma'_1+\gamma'_2$. For instance, in Figure~\ref{fig:proof_partition_graph}, we have $\gamma'=1$, $\gamma'_2=1$, and $\gamma'_1=1$.


Among the $|\cP_1|$ groups of the partition $\cP_1$ of $[2k]$ of distinct values of $i_{ev}$, $\gamma$ groups intersect $S$ and correspond to values $r$ such that $N_{rr}(ij)>0$, $\gamma'_1$ groups correspond to values $r$ such that $N_{rr}(i'j')>0$ and $N_{rr}(ij)=0$, and $|\cP_1|-\gamma-\gamma'_1$ groups  correspond to values $r$ such that $N_{rr}(ij)+ N_{rr}(i'j')=0$. Among the $|\cP_2|$ groups of the partition $\cP_2$ of $[2k]$ of distinct values of $j_{od}$, $\gamma$ groups intersect $S$ and correspond to values $r$ such that $N_{rr}(ij)>0$, $\gamma'_2$ groups correspond to values $r$ such that $N_{rr}(i'j')>0$ and $N_{rr}(ij)=0$, and $|\cP_2|-\gamma-\gamma'_2$ groups  correspond to values $r$ such that $N_{rr}(ij)+ N_{rr}(i'j')=0$.

The upper bound~\eqref{eq:upper_T_lem} of $T_{iji'j'}$ does not depend on the specific values of $(i,j)$ in these $|\cP_1|-\gamma-\gamma'_1$ and $|\cP_2|-\gamma-\gamma'_2$  groups. 
Recall that there are $\gamma+\gamma'$ distinct values $r$ of such that  $N_{rr}(ij)+ N_{rr}(i'j')>0$. In fact, there exists a vector of positive integers $\alpha_{(1)}\geq \alpha_{(2)}\geq \ldots  \geq \alpha_{(\gamma+\gamma')}>0 $ satisfying $\sum_{l=1}^{\gamma+\gamma'} \alpha_{(l)}= |S|+|S'|=2|S|$ and a $\gamma+\gamma'$-tuple $l_1,\ldots,l_{\gamma+\gamma'}$ of distinct values (that only depend on the values of $i_{ev}$ in the $\gamma+\gamma'_1$ groups of $\cP_1$ and on the values of $j_{od}$ in the $\gamma+\gamma'_2$ groups in $\cP_2$) such that 
\[
T_{iji'j'}  \leq (2k)^{2k-|S|-|P_1|-|P_2|+\gamma+1} \prod_{s=1}^{\gamma+\gamma}\left[e^{1/2}\sigma_{l_s}(\bA^+)\right]^{\alpha_{(l)}}\enspace .
\]
 Summing over all possible sequences $(i,j)$, we conclude that 
\begin{eqnarray}
R_{S,S';\cP_1,\cP_2; F_1;F_2} &\leq&e^{k}(2k)^{2k-|S|-|\cP_1|-|\cP_2|+\gamma+1} p^{|\cP_1|-\gamma -\gamma'_1}q^{|\cP_2|-\gamma-\gamma'_2} \sum_{l_1,\ldots, l_{\gamma+\gamma'}=1}^q \prod_{s=1}^{\gamma+\gamma'}\sigma_{l_s}(\bA^+)^{\alpha_{(s)}}\nonumber \\
&\leq & e^{k}(2k)^{2k-|S|-|\cP_1|-|\cP_2|+\gamma+1} p^{|\cP_1|-\gamma -\gamma'_1}q^{|\cP_2|-\gamma-\gamma'_2} \prod_{s=1}^{\gamma+\gamma'} \|\bA^+\|_{\alpha_{(s)}}^{\alpha_{(s)}}
\nonumber \\
&\leq & e^{k}(2k)^{2k-|S|-|\cP_1|-|\cP_2|+\gamma+1} p^{|\cP_1|-\gamma -\gamma'_1}q^{|\cP_2|+\gamma'-\gamma'_2-1} \|\bA^+\|_{2|S|}^{2|S|} \nonumber \\
&\leq & e^{k}(2k)^{2k-|S|-|\cP_1|-|\cP_2|+\gamma+1} p^{|\cP_1|-\gamma -\gamma'_1}q^{|\cP_2|+\gamma'_1-1} \|\bA^+\|_{2|S|}^{2|S|} \nonumber \\
&\leq & e^{k}(2k)^{2k-|S|-|\cP_1|-|\cP_2|+\gamma+1}  p^{|\cP_1|- \gamma}q^{|\cP_2|-1} \|\bA^+\|_{2|S|}^{2|S|}\enspace ,
\label{eq:R_S_Q_P1P2}
\end{eqnarray}
where we used  H\"older's inequality in the third line, $\gamma'\leq \gamma'_1+\gamma'_2$ in the fourth line and $p\geq q$ in the last line. This upper bound does not depend on the values of $F_1$ and $F_2$. Given $S$, $S'$, $\cP_1$, and $\cP_2$, let us count the number of possible maps $F_1$ and $F_2$ such that $R_{S,S';\cP_1,\cP_2; F_1;F_2} >0$.


\begin{lem}\label{lem:number_F}
Given $S$, $S'$ with $|S'|=|S|$ and partitions  $\cP_1$ and $\cP_2$ of $[2k]$, there are at most 
\beq\label{eq:upper_P1}
 (2e)^{2k}k^{2k-|S|+1 -|\cP_1|-|\cP_2|+\gamma}
\eeq
possible maps $F_1$ and $F_2$ such that  $R_{S,S';\cP_1,\cP_2; F_1;F_2} >0$.
\end{lem}

Upon defining $R_{S,S';\cP_1,\cP_2}= \sum_{F_1,F_2}R_{S,S',\cP_1,\cP_2,F_1,F_2}$, we derive that 
\beq\label{eq:upper_S_S'_P_1_P_2}
R_{S,S';\cP_1,\cP_2} \leq 2^{4k}e^{3k}k^{	2(2k-|S|-|\cP_1|-|\cP_2|+\gamma+1)}p^{|\cP_1|- \gamma}q^{|\cP_2|-1} \|\bA^+\|_{2|S|}^{2|S|}\enspace .
\eeq
Now, we shall sum these expressions $R_{S,S';\cP_1,\cP_2}$ over the partitions $\cP_1$ and $\cP_2$.  Fix $l_1$, $l_2$, and $r \leq l_1\wedge l_2$ and define $R_{S,S';r,l_1,l_2}= \sum_{\cP_1,\cP_2\ |\cP_1|=l_1,\ |\cP_2|=l_2,\  \gamma=r} R_{S,S',\cP_1,\cP_2}$ the sum over partitions $\cP_1$ and $\cP_2$ with sizes $l_1$ and $l_2$ and where $\gamma$ is equal to $r$.

\begin{lem}~\label{lem:number_P1_P2_r}
The number of partitions $\cP_1$, $\cP_2$ with respective sizes $l_1$ and $l_2$ and common number of groups $r$ is at most
\[
e^{3k}k^{2k-|S|-l_1-l_2+r}\enspace .
\]
\end{lem}

We deduce that 
\begin{eqnarray}\nonumber
R_{S,S';r,l_1,l_2}& \leq& k^2(2^{4k}e^{6k})k^{3[2k-|S|-l_1-l_2+r]}p^{l_1- r}q^{l_2-1} \|\bA^+\|_{2|S|}^{2|S|}\\
&\leq &  k^2(2^{4k}e^{6k}) k^{3(2k-|S|)}\left(\frac{\ol{p}}{k^3}\right)^{l_1- r}\left(\frac{\ol{q}}{k^3}\right)^{l_2-1} \|\bA^+\|_{2|S|}^{2|S|}\ ,\label{eq:upper_R_SS'rl}
\end{eqnarray}
where we denote $\ol{p}= p\vee k^3$ and $\ol{q}= q\vee k^3$.  
Since $\ol{p}\geq \ol{q}\geq k^3$, the above rhs is maximized for $l_1-r$ and $l_2$ the largest possible. For $|S|=0$, we simply use $l_1\leq k$ and $l_2\leq k$ and $r=0$.

For $|S|=s>0$, we first observe that each of the $l_1-r$ groups of the partition $\cP_1$ outside $S$ must contain at least two points. This implies that 
$2l_1-2r\leq 2k-s$. Since $l_1-r$ is an integer, we deduce that $l_1-r\leq k- \lceil s/2\rceil$. 

Furthermore, we denote $\Delta_{S,(-)}=\{t\in S^c: t+1\in S\}$ and $\Delta_{S,(+)}=\{t\in S^c: t-1\in S\}$ the left and right boundaries of $S$. Since the subsets of the form $\{2t+1,2t+2\}$ are in the same group in $\cP_1$ and the subsets of the form $\{2t,2t+1\}$ are in the same group in $\cP_2$, it follows that each $t\in \Delta_{S,(-)}\cup\Delta_{S,(+)}$  either belongs to a group of $\cP_1$ that intersects $S$ or a group in $\cP_2$ that intersects $S$. Besides, each  $t\in \Delta_{S,(-)}\cap\Delta_{S,(+)}$ belongs to a group of $\cP_1$ that intersects $S$ and a group in $\cP_2$ that intersects $S$.
Since each of the $l_1-r$ groups of $\cP_1$ and $l_2-r$ groups of $\cP_2$ that do not intersect $S$ are at least of size $2$, we deduce that $2(l_1-r+l_2-r)\leq 4k -2s -|\Delta_{S,(-)}|-|\Delta_{S,(+)}|$ which implies that
\[
l_1-r+l_2-1\leq 2k - s -\frac{|\Delta_{S,(-)}|+|\Delta_{S,(+)}| }{2} + r-1 \leq 2k-s-1  \enspace ,
\]
since both $|\Delta_{S,(-)}|$ and $|\Delta_{S,(+)}|$ equal the number of connected components of $S$ which is larger or equal to $r$. Together with the bound $l_1-r\leq k- \lceil s/2\rceil$ and the assumption $k^3\leq \overline{q}\leq \overline{p}$, we deduce from~\eqref{eq:upper_R_SS'rl} that 
\[
R_{S,S';r,l_1,l_2} \leq  k^5(2^{4k}e^{6k})   \ol{p}^{k- \lceil |S|/2\rceil}\ol{q}^{k- \lfloor |S|/2\rfloor- 1} \|\bA^+\|_{2|S|}^{2|S|}\enspace . 
\]
There are at most $2^{4k}$ possible sets $(S,S')$ and at most $k^3$ possible values for $r$, $l_1$, and $l_2$. Hence, together with~\eqref{eq:upper_var_Uk2} we conclude that 
\[
 \var{U_k}\leq k^8c_0^k \left[(\ol{p}\ol{q})^{k}+  \sum_{s=1}^{2k-1}\ol{p}^{k - \lceil s/2\rceil  } \ol{q}^{k-\lfloor s/2\rfloor-1} \|\bA^+\|_{2s}^{2s} \right]\enspace ,
\]
for $c_0=(2^{8}e^{6})$. 
Applying again H\"older's inequality, we get $\|\bA^+\|_{2s}^{2s}\leq [\|\bA^+\|^{2s}_{4k-4}q^{1-s/(2k-2)}]\wedge [\|\bA^+\|^{2s}_{4k-2}q^{1-s/(2k-1)}]$ for $s\leq 2k-2$. Hence,
\beqn 
\var{U_{k}}&\leq &  k^8c_0^k\left[  \sum_{t=0}^{k-1}\ol{p}^{k-t}\ol{q}^{k-t\frac{k}{k-1}}\|\bA^+\|^{4t}_{4k-4}+ \ol{p}^{k-t-1}\ol{q}^{k-t-\frac{2t+1}{2k-1} }\|\bA^+\|^{4t+2}_{4k-2} \right]\\
&\leq & k^9c_0^k\left[ (\ol{p}\ol{q})^k+ \ol{p}^{k-1}\ol{q}^{k-\frac{1}{2k-1}}\|\bA^+\|^{2}_{4k-2}+  \ol{p}\|\bA^+\|^{4k-4}_{4k-4}+ \|\bA^+\|^{4k-2}_{4k-2} \right]\\
&\leq & k^9c_0^k\left[ (\ol{p}\ol{q})^k+ [(\ol{p}\ol{q})^{k}]^{1-\frac{1}{2k-1}}\left[\|\bA^+\|^{4k-2}_{4k-2}\right]^{\frac{1}{2k-1}}+  \ol{p}\|\bA^+\|^{4k-4}_{4k-4}+ \|\bA^+\|^{4k-2}_{4k-2} \right]\\
&\leq & 2k^9c_0^k\left[ (\ol{p}\ol{q})^k+   \ol{p}\|\bA^+\|^{4k-4}_{4k-4}+ \|\bA^+\|^{4k-2}_{4k-2} \right]\ ,
\eeqn 
where we used that the sequences are monotonic with respect to $t$ in the second line, that $\ol{p}\geq \ol{q}$ in the third line, and that $x^{a}y^{1-a}\leq x\vee y$ for any $a\in [0,1]$.  Since $\|\bA^+\|^{2s}_{2s}\leq \|\bA\|_{2s}^{2s}+ q(2k)^{s}$, we conclude that 
\[
 \var{U_k}\leq  k^{c_1}(c'_1)^k\left[   (pq)^{k} +  p \|\bA\|_{4k-4}^{4k-4}+ \|\bA\|_{4k-2}^{4k-2} + p^kk^{3k}+  k^{6k}\right]\enspace \ , 
\]
for some $c_1$ and $c'_1>0$, which concludes the proof.

\end{proof}

\begin{proof}[Proof of Lemma~\ref{lem:number_degrees}]

  Since $S^{id}[i,j]=S$, we have $\sum_{r\neq s}N_{rs}(ij)= 2k -|S|$. Write $d= \sum_{r\neq s}\1_{N_{rs}(ij)>0}$. We have therefore
  \[
   \prod_{r\neq s}[N_{rs}(ij)]!\leq (2k)^{2k-|S|-d}\ ,
  \]
  and it suffices to prove that $d\geq |\cP_1|+|\cP_2|-\gamma -1$. Write $d'=\sum_{r,s} \1_{N_{rs}(ij)>0}$. By definition of $\gamma$, we have $d+\gamma\geq d'$ and we shall prove that $d'\geq |\cP_1|+|\cP_2|-1$. Exploring iteratively the sequence $(i\oplus j)$ in order, we have $N_{i_1,j_1}(ij)>0$. Then, each time we visit the first element of one group of $\cP_1$ or one group of $\cP_2$, we discover a new  $(r,s)$ such $N_{rs}(ij)>0$. Since, except at the first coordinate of $(i\oplus j)$,   we cannot simultaneously visit for the first time an element of $\cP_1$ and element of $\cP_2$, we deduce that  $d'\geq |P_1|+|P_2|-1$ and the result follows.
  \end{proof}

  \begin{proof}[Proof of Lemma~\ref{lem:number_F}]
    The two partitions $\cP_1$ and $\cP_2$ induce a partition of $[2k]\setminus S$ into groups of the form $(r,s)$ where $r\in [|\cP_1|]$ and $s\in [|\cP_2|]$.
    Note that the number of groups of this partition is exactly  $d= \sum_{r\neq s}\1_{N_{rs}(ij)>0}$, where $i$ and $j$ are any sequence that are compatible with $\cP_1$ and $\cP_2$.
    The maps $F_1$ and $F_2$ uniquely induce  a  partition of $[2k]\setminus S'$. By property (a) in Page~\pageref{prop:a_b_c_d}, $R_{S,S';\cP_1,\cP_2; F_1;F_2} >0$ only if this partition of $[2k]\setminus S'$ is, up a permutation, the same as one induced  $\cP_1$ and $\cP_2$ on $[2k]\setminus S'$. In particular, the number of possible $(F_1,F_2)$ is less than the Stirling number of the second kind with parameters $n=(2k-|S'|)$ and $d$, which in turn is smaller than 
    \[
    \binom{2k-|S'|}{d}d^{2k-|S'|-d}\leq e^{2k}(2k)^{2k-|S'|-d}\leq (2e)^{2k}k^{2k-|S'|-d}
     \ . 
    \]
    We have shown in the proof of Lemma~\ref{lem:number_degrees} that $d\geq |\cP_1|+|\cP_2|-\gamma-1 $. Since $|S'|=|S|$, the result follows. 
    \end{proof}

\begin{proof}[Proof of Lemma~\ref{lem:number_P1_P2_r}]
  Let us count the number of possible partitions $\cP_1$ and $\cP_2$ of $[2k]$ with respective size $l_1$ and $l_2$. 
  Write $K$ the number of connected components of $S$ when one consider $[2k]$ as a torus. Among those $K$ connected components, denote $K_1$ (resp. $K_2$) the number of connected components whose most left entry lies at an even (resp. odd) position. As explained earlier, if the  most left entry $t$ of a connected component $T$ of $S$ is an even integer than $t-1$ belongs to the same group as $t$ in $\cP_1$. Besides, all the points of a connected component $T$ of $S$  belong to the same group. 
  Since there are $r$ groups of $\cP_1$ and $\cP_2$ that intersect $S$, we start by considering the total number of induced partitions on $S$ with $r$ groups which is at most $r^K/r!$. Given this induced partition it remains to count the number of possible partitions $\cP_1$ outside $S$. Since $\cP_1$ is completely defined by its restrictions to odd integers of $[2k]$, it remains to define the partition $\cP_1$ on $k-K_1-|S_1|$ points, where $S_1$ (resp. $S_2$) is the intersection of $S$ with the set of odd (resp. even) integers. Arguing similarly for $\cP_2$, we derive that that there are at most 
  \[
  \frac{r^K l_1^{k- |S_1|-K_1}l_2^{k-|S_2|-K_1}}{r!(l_1-r)!(l_2-r)!}\enspace ,
  \]
  possible partitions $\cP_1$ and $\cP_2$. Using Stirling's lower bound $r!\geq (r/e)^r$, that $K_1+K_2=K$, $|S_1|+|S_2|=|S|$ and that $l_1\vee l_2\leq k$, we conclude that the total number of partitions it at most
  \[
    e^{r}\frac{r^{K-r} l_1^{k+r- |S_{1}|-K_1}l_2^{k+r-|S_2|-K_2}}{l_1!l_2!}\leq e^{3k}k^{2k-|S|-l_1-l_2+r}  \enspace .
  \]
  \end{proof}

\subsection{Proof of Proposition~\ref{prp:representation_somme_newton}}

The result is a consequence of the following lemma.
 \begin{lem}\label{lem:representation_V}
 There exist coefficients $\alpha_s$ such  that 
 \beq\label{eq:definition:representation_somme_newton_bis}
V_k = \tr[(\bY^T\bY)^k]+ \alpha_0 + \sum_{l=1}^{k-1} \sum_{s=(s_1,\ldots, )\ , 1\leq  s_1\leq s_2 \leq s_3 ; \sum s_i=l} \alpha_s \prod \tr[(\bY^T \bY)]^{s_i}
 \eeq
 satisfies $\mathbb{E}[U_k]= \mathbb{E}[V_k]$ for all $\bA\in \R^{p\times q}$.
 \end{lem}

 Obviously, both  $\mathbb{E}[V_k]$  and $\E[U_k]$  are polynomials with respect to the entries of $\bA$. Since these two polynomials take the same value for all $\bA \in \R^{p\times q}$ and since the polynomial evaluation morphism is injective (see Corollary 1.6 in~\cite{lang2002graduate}), this implies that both $\E[U_k]$ and $\mathbb{E}[V_k]$ correspond to the same polynomial in $\R[X_{ij}, 1\leq i\leq p; 1\leq j\leq q]$. 
 
 We now deduce that $U_k=V_k$. Let $Y\sim \mathcal{N}(x,1)$.  For any non-negative integer $\beta$, there exist a polynomial $G_{\beta}$ of degree $\beta$ such that $\E[Y^\beta]= G_{\beta}(x)$. Besides, its term of degree $\beta$ is exactly $x^\beta$. To see this, it suffices to observe to use the binomial formula on $\E[Y^\beta]= \E[(x + (Y-x))^\beta]$.
 
 Consider the vector space morphism $\psi:  \R[(X_{ij})]\rightarrow \R[(X_{ij})]$   such that $\psi[\prod_{i=1,j=1}^{p,q} X_{ij}^{\beta_{ij}}]= \prod_{i=1,j=1}^{p,q} G_{\beta_{ij}}[X_{ij}]$. Then, $\mathbb{E}(V_k)= \psi[(V_k)](\bA)$ and $\mathbb{E}(U_k)= \psi[(U_k)](\bA)$. This implies that  $\psi(U_k)= \psi(V_k)$.  We claim that $\psi$ is a injective. As a consequence,  $U_k= V_k$ and the result follows.
 
 Let us show that $\psi$ is injective. Consider any monomial of the form $Z= \prod_{i=1,j=1}^{p,q} X_{ij}^{\beta_{ij}}$. Observe that $\psi(Z)-Z$ has a total degree less than $\sum_{ij}\beta_{ij}$. Given a polynomial $P$ and the sum $P_l$  of monomials of largest total degree, we derive $\psi(P)-P_l$ has a total degree less than that of $P$. Thus, $\psi(P)$ has the same total degree as $P$ and $\psi$ is therefore injective.

\begin{proof}[Proof of Lemma \ref{lem:representation_V}]

Given a sequence of integers $1\leq \alpha_1\leq \alpha_2\leq  \ldots \leq \alpha_r$, we claim that 
\beq\label{CLAIM_1representation}
\E\left[\prod_{t=1}^r \tr[(\bY^T\bY)^{\alpha_t}]\right]- \prod_{t=1}^r \tr[(\bA^T\bA)^{\alpha_t}]
\eeq 
is a symmetric polynomial with total degree less than $\sum_{i}\alpha_i$ with respect to ($\sigma^2_{i}(\bA)$, $i=1,\ldots, q$). The proof of this claim is postponed to the end of the proof. Since the space of symmetric polynomials is spanned by the Newton sums,
there exist coefficients $\beta_s$ such that 
\beqn 
\E\left[\prod_{t=1}^r \tr[(\bY^T\bY)^{\alpha_t}]\right]=  \prod_{t=1}^r \tr[(\bA^T\bA)^{\alpha_t}]+  \sum_{s: \sum s_t < \sum \alpha_t} \beta_s \prod_{t=1}^{l(s)}\tr[(\bA^T\bA)^{s_t}]\ . 
\eeqn 

Then, we prove the existence of $V_k$ by induction. First, $\E[\tr[(\bY^T\bY)^k]]-\|\bA\|_{2k}^{2k}= Z_1$ where $Z_1$ is of the form 
$Z_1= \sum_{s: \sum s_t < k }\beta^{(1)}_s \prod_{t=1}^{l(s)}\tr[(\bA^T\bA)^{s_t}]$. Then, considering each term of total degree $k-1$, we can estimate them by $\beta^{(1)}_s\prod_{t=1}^{l(s)}\tr[(\bY^T\bY)^{s_t}]$. The resulting bias has a total degree less or equal to $k-2$. At each step, one can correct the estimator to decrease the total degree of the bias. The result follows.

\medskip 

\noindent 
{\bf Proof of the claim in \eqref{CLAIM_1representation}}. 
Write $Z= \prod_{t=1}^r \tr[(\bY^T\bY)^{\alpha_t}]$. Since $Z$ is invariant by left and right orthogonal transformation of $\bY$, we  assume without loss of generality that $\bA$ has null non-diagonal entries. As a consequence, $\E[Z]$ is a polynomial with respect to $\sigma_i(\bA)$. Since $Z$ is invariant by permutation of the rows and columns of $\bY$, $\E[Z]$ is symmetric. It remains to prove that (i) $\E[Z]- \prod_{t=1}^r \prod \tr[(\bA^T\bA)^{\alpha_t}]$ has a total degree less than $2k$ and (ii) that, for all $i=1,\ldots, q$, the order of $\sigma_i(\bA)$ in each monomial is even:

For $X\sim \cN(x,1)$, $\E[X^r]-x^r$ is polynomial of degree less than $r$. Since the decomposition of $Z$ leads to a sum of monomials in $\bY_{ij}$ of total degree equals to  $\sum_{t=1}^r\alpha_t$, it follows that $\E[Z]$ contains the exact same monomial in $A_{ij}$ of total degree $\sum_{t=1}^r\alpha_t$ and the remaining terms have a total degree less than that. (i) follows. 

Turning to (ii), we decompose again $Z$ into monomials.  $Z$ writes as 
\[
 Z= \sum_{(i^{(1)},j^{(1)})\in p^{\alpha_1}\times q^{\alpha_1}}\ldots \sum_{(i^{(r)},j^{(r)})\in p^{\alpha_r}\times q^{\alpha_r}}\prod_{l=1}^{\alpha_1}\bY_{i^{(1)}_lj^{(1)}_l}\bY_{i^{(1)}_{l+1}j^{(1)}_{l}}\ldots \prod_{l=1}^{\alpha_r} \bY_{i^{(r)}_lj^{(r)}_l}\bY_{i^{(r)}_{l+1}j^{(r)}_{l}}
\]
As a consequence, for any index $r\in [p]$,  the set $\{\bY_{rs}; s=1,\ldots, q\}$ is visited an even number of times in each monomial. Besides, the expectation of one monomial is non-zero only if each non diagonal element $\bY_{rs}$ is visited an even number of times since $\bA_{rs}=0$ and the normal distribution is symmetric. As a consequence of these two observations, each  $\bY_{ss}$ must be visited an even number of times in the monomial so that its expectation is non-zero. Since, for $X\sim \cN(x,1)$, $\E[X^{2r}]$ is an even polynomial in $x$, this implies that the expectation of each monomial of $Z$ only involves terms of the form $\sigma^{2r}_i(\bA)$ and (ii) follows.

\end{proof}

\begin{proof}[Proof of Corollary \ref{cor:schatten}]

Since $l_r$ norms of vectors are non increasing with respect to $r$, we have  $\|\bA\|^{4k-2}_{4k-2}\leq \|\bA\|^{4k-2}_{2k}$ and $\|\bA\|^{4k-4}_{4k-4}\leq \|\bA\|^{4k-4}_{2k}$. Hence, we derive from Theorem~\ref{prp:risk_U_k} and Chebychev inequality that 
\beq\label{eq:upper_risk_U_k_1}
\var{U_k}^{1/2}\leq  c^{k}\left[(pq)^{k/2}+ p^{1/2} \|\bA\|^{2k-2}_{2k}+ \|\bA\|^{2k-1}_{2k}+(pk^3)^{k/2}+ k^{3k}\right]\ .
\eeq
We consider two cases depending on $\|\bA\|^{2k}_{2k}$. If  $\|\bA\|^{2k}_{2k} \leq   [(p\vee k^3)(q\vee k^3)]^{k/2}$, then the above risk bound simplifies in 
$\E\left[|(U_k)-\|\bA\|_{2k}^{2k}|\right]\leq  5c^{k} [(p\vee k^3)(q\vee k^3)]^{k/2}$. Then, together with the inequality $|x^{1/(2k)}-y^{1/(2k)}|\leq |x-y|^{1/(2k)}$ and H\"older's inequality, we have 
\beq\label{eq:upper_1_cor:schatten}
\E\left[|(U_k)_+^{1/(2k)}- \|\bA\|_{2k}|\right]\leq \E\left[|U_k- \|\bA\|^{2k}_{2k}|\right]^{1/2k}\lesssim  [(p\vee k^3)(q\vee k^3)]^{1/4}\ .
\eeq
Let us turn to the case where  $\|\bA\|^{2k}_{2k} \geq   [(p\vee k^3)(q\vee k^3)]^{k/2}$. Since 
$|(1+x)^{1/(2k)}-1|\leq  2|x|$ for any $x\geq - 1/2$, we have 
\begin{eqnarray}
  \lefteqn{
\E\left[|(U_k)_+^{1/(2k)}- \|\bA\|_{2k}|\right]} &&  \nonumber \\&=&\E\left[ \|\bA\|_{2k} \big|[U_k/\|\bA\|^{2k}_{2k}]^{\frac{1}{2k}} - 1\big|\right] \nonumber\\
&\leq & \|\bA\|_{2k}\P\left[U_k\leq \|\bA\|_{2k}^{2k}/2\right]+ 2\|\bA\|^{1-2k}_{2k}\E\left[|U_k- \|\bA\|_{2k}^{2k}|\right] \nonumber \\
&\leq & \|\bA\|_{2k}\left[4\frac{\var{U_k}}{\|\bA\|^{4k}_{2k}}+ 2\frac{\var{U_k}^{1/2}}{\|\bA\|^{2k}_{2k}} \right]\nonumber \\
&\leq &  c'c^{k}\left[\frac{(p\vee k^3)(q\vee k^3)]^k}{\|\bA\|^{4k-1}_{2k}}+ \frac{(p\vee k^3)(q\vee k^3)]^{k/2}}{\|\bA\|^{2k-1}_{2k}}+ \frac{p}{\|\bA\|_{2k}^{3}}+\frac{p^{1/2}}{\|\bA\|_{2k}}+  \frac{1}{\|\bA\|_{2k}}+1 \right] \nonumber
\\
&\lesssim & c^{k}[(p\vee k^3)(q\vee k^3)]^{1/4}
\label{eq:up_risk_2}\ ,
\end{eqnarray}
where we used \eqref{eq:upper_risk_U_k_1} in the fourth line  and the condition $\|\bA\|_{2k} \geq  (p\vee k^3)(q\vee k^3)^{1/4}$ in the last line. 
 
\end{proof}

\begin{proof}[Proof of Proposition \ref{prp:U2}]
The definition of the estimators $U_2$ and $U_3$ is a consequence of the following moment identities
\begin{eqnarray}
\E[\tr(\bY^T\bY)]&=& \|\bA\|_2^2 + pq \ ;  \label{eq:moment1} \\
\E[\tr((\bY^T\bY)^2)]&=  &\|\bA\|_4^4+ 2(p+q+1)\|\bA\|_2^2+pq(1+q+p)\ ;   \label{eq:moment2} \\
\E[\tr((\bY^T\bY)^3)] &= &\|\bA\|_6^6 + 3(p+q+1)\|\bA\|_4^4+ 3(p^2+ q^2 + 3pq+3p +3q +4)\|\bA\|_2^2 \nonumber \\ 
 && + pq [p^2 + q^2+ 3pq+ 3p+ 3q + 4]\ .\label{eq:moment3} 
\end{eqnarray}
After some tedious computations, we see that the expressions in the right-hand side in Proposition~\ref{prp:U2} are unbiased estimators of $\|\bA\|_2^2$ and $\|\bA\|_4^4$ respectively. 

In the remainder of the proof, we show (\ref{eq:moment1}--\ref{eq:moment3}).
The first identity has already been proved to analyze $U_1$. In order to derive the second and the third identity, we decompose $\bY\bY^T$ into a sum involving its expectation. 
\[
\bY^T\bY= \left(\bA^T\bA + p\bI_q\right) + \left(\bE^T\bE- p\bI_q+ \bE^T\bA + \bA^T \bE\right)=: \bS + \bN\ . 
\]
Since $\bS$ is deterministic and $\bN$ is centered, we have 
\[
\E[\tr[(\bY^T\bY)^2]]=  \tr(\bS^2) +\tr[\E(\bN^2)]\ . 
\]
Since the Gaussian distribution is symmetric, we obtain by straightforward computation that 
 \beq\label{eq:decomposition_2-0}
 \E[\bN^2] = \E[(\bE^T\bE-p\bI_q)^2]+ \E[(\bA^T \bE + \bE^T\bA)^2]= (q+1)\left[p\bI_q + 2\bA^T\bA\right]\ . 
\eeq 
Combining the two previous identities leads to \eqref{eq:moment2}. Turning to \eqref{eq:moment3}, we have 
\beq\label{eq:decomposition_3-1}
\E[\tr[(\bY^T\bY)^2]]=  \tr(\bS^3) +\tr[\E(\bN^3)]+ 3\tr[\bS\E(\bN^2)]\ ,
\eeq
since $\bN$ is centered. For $\E[\tr(\bN^3)]$, we use a again that odd moments of centered normal distributions are null to derive
\beq\label{eq:decomposition_3-2}
\E[\tr(\bN^3)]= \E[\tr[(\bE^T\bE-p\bI_q)^3]]+3\E\left[\tr[ (\bE^T\bE-p\bI_q)(\bA^T \bE + \bE^T\bA)^2]\right]\ . 
\eeq
Starting with the first expression, we derive from moments of standard Wishart distribution (see e.g. \cite{letac2004all}) that 
\[
\E[\tr[(\bE^T\bE-p\bI_q)^3]]= pq [q^2+3q+4] \ . 
\]
Tedious computations also lead us to 
\[
 \E\left[\tr[ (\bE^T\bE-p\bI_q)(\bA^T \bE + \bE^T\bA)^2]\right]= \|\bA\|_2^2 [(q-1)(q-2)+3(q-1)2+ 8]= \|\bA\|_2^2 [q^2+3q+4]\ . 
\]
Combining the two previous identities in \eqref{eq:decomposition_3-2}, we obtain 
\[
 \E[\tr(\bN^3)]= [q^2+3q+4] (pq + 3\|\bA\|_2^2)\ . 
\]
Then, coming back to \eqref{eq:decomposition_3-1}, and relying on \eqref{eq:decomposition_2-0}, we get 
\beqn 
 \E[\tr((\bY^T\bY)^3)]&=& \tr[(\bA^T\bA+p\bI_q)^3]+ [q^2+3q+4] (pq + 3\|\bA\|_2^2)\\ && + 3 (q+1)\tr[(\bA^T\bA+p\bI_q)( 2\bA^T\bA+ p\bI_q )]]\\
 & =& \|\bA\|_6^6 + 3(p+q+1)\|\bA\|_4^4+ 3(p^2+ q^2 + 3pq+3p +3q +4)\|\bA\|_2^2\\ 
 && + p^3q + q^3p+ 3p^2q^2+ 3p^2q+ 3q^2p + 4pq\ , 
\eeqn 
which is exactly \eqref{eq:moment3}.

\end{proof}

\subsection{Proofs for general norms $\|\bA\|_{s}$}

Given a symmetric matrix $\bB$, we denote $\lambda_1(\bB)\geq \lambda_2(\bB)\geq \ldots$ its sequence of eigenvalues.
We start with a technical lemma. 

\begin{lem}\label{lem:local_control_lambda_i}
 Define $\bW= \bY^T\bY-p\bI_q$. 
Write $I(\bA)$ for the image of $\bA$ and $\Pi_{I(\bA)}$ for  any orthogonal projection matrix in  $\mathbb{R}^p$ onto $I(\bA)$. For any $i=1,\ldots, q$, we have  
 \beq\label{eq:upper_loss_sigma_i_A}
 |\lambda_i(\bW) - \sigma_i^2(\bA)|\leq 2\sigma_i(\bA)\|\Pi_{I(\bA)}\bE\|_{\infty} + \|\Pi_{I(\bA)}\bE\|^2_{\infty}+ \|\bE^T\bE- p \bI_q\|_{\infty}\ .
 \eeq
\end{lem}

\begin{proof}[Proof of Proposition~\ref{prp:upper_risk_nuclear_covariance}]
Define $\bW= \bY^T\bY-p\bI_q$, then $\sigma_i^2(\bY)-p=\lambda_i(\bW)$ so that  $T_s= [\sum_{i=1}^q (\lambda_i(\bW))_+^{s/2}]^{1/s}$. We first consider the case $s/2\geq 1$.  For any $x\in \mathbb{R}_+$, we have $|(x+y)_+^{s/2}-x^{s/2}|\leq s|y|(2x)^{s/2-1}+ (2|y|)^{s/2}$. It then follows from the triangular inequality and Lemma~\ref{lem:local_control_lambda_i} that 
\begin{eqnarray}
|T^s_s - \|\bA\|^s_s|&\leq& \sum_{i=1}^q |[\lambda_i(\bW)]_{+}^{s/2}- \sigma^{s}_i(\bA)|  \nonumber \\
&\leq & 2^{s/2}s \sum_{i=1}^{q}\left[\|\Pi_{I(\bA)}\bE\|^2_{\infty}+ \|\bE^T\bE- p \bI_q\|_{\infty}+ 2 \sigma_i(\bA)\|\Pi_{I(\bA)}\bE\|_{\infty}\right]\sigma^{s-2}_i(\bA) \nonumber \\ & &\quad \quad\quad \quad  + 3^{s/2-1}\left[\|\Pi_{I(\bA)}\bE\|^s_{\infty}+ \|\bE^T\bE- p \bI_q\|^{s/2}_{\infty}+ 2^{s/2} \sigma^{s/2}_i(\bA)\|\Pi_{I(\bA)}\bE\|^{s/2}_{\infty}\right]\ .\nonumber \\
\label{eq:upper_T_S}
\end{eqnarray}
Since $\bE$ is distribution invariant by a left orthogonal transformation $\|\Pi_{I(\bA)}\bE\|_{\infty}$ is distributed as the operator norm of a $\mathrm{dim}(I(\bA))\times q$ noise matrix. By integration, we deduce  from Lemma~\ref{lem:DavSza} that, for any $s\geq 1$,  $\E[\|\Pi_{I(\bA)}\bE\|^{s}_{\infty}]\leq c^{s}(q\vee s)^{s/2}$. Similarly, arguing as in the proof of Lemma~\ref{lem:noiseop} and integrating the tail distribution of $\|\bE^T\bE- p \bI_q\|_{\infty}$, we derive that $\E[\|\bE^T\bE- p \bI_q\|^{s}_{\infty}]\leq c^{s}[(p\vee s)(q\vee s)]^{s/2}$ for any $s\geq 1$. This allows us to conclude that 
\begin{eqnarray}
 \E\left[|T^s_s - \|\bA\|^s_s|\right]&\leq& c^{s} \left[\sqrt{pq} \|\bA\|_{s-2}^{s-2}+ \sqrt{q}\|\bA\|_{s-1}^{s-1}+ q[(p\vee s)(q\vee s)]^{s/4}+ (q\vee s )^{s/4}\|\bA\|_{s/2}^{s/2}\right] \nonumber \\
&\leq & c^{s} \left[\sqrt{pq} \|\bA\|_{s-2}^{s-2}+ \sqrt{q}\|\bA\|_{s-1}^{s-1}+ q[(p\vee s)(q\vee s)]^{s/4}\right]\ , \label{eq:upper_risk_moment1}
\end{eqnarray}
where we used H\"older's inequality and $xy\leq x^\alpha + y^{\beta}$ for $1/\alpha + 1/\beta=1$. 

For $s\in [1,2)$, we use the inequality $|(x+y)_+^{s/2} - x^{s/2}|\leq |y|^{s/2} \wedge (2|y||x|^{s/2-1})$. We then deduce from Lemma~\ref{lem:local_control_lambda_i} that 
\beq
|T^s_s - \|\bA\|^s_s|= \sum_{i=1}^q |[\lambda_i(\bW)]_{+}^{s/2}- \sigma^{s}_i(\bA)| 
\leq  \sum_{i=1}^{q}\|\Pi_{I(\bA)}\bE\|^{s}_{\infty}+ \|\bE^T\bE- p \bI_q\|^{s/2}_{\infty}+ 4 \sigma^{s-1}_i(\bA)\|\Pi_{I(\bA)}\bE\|_{\infty}\ . \label{eq:upper_l_1}
\eeq
Bounding the moments as above, we conclude that 
\beq\label{eq:upper_risk_moment2}
 \E\left[|T^s_s - \|\bA\|^s_s|\right]\lesssim \left[q(pq)^{s/4}+ \sqrt{q}\|\bA\|^{s-1}_{s-1}\right]\ . 
\eeq
The first result follows from~\eqref{eq:upper_risk_moment1} and \eqref{eq:upper_risk_moment2}. 
Regarding the second result, we use again that, for $s>1$, 
 $|(x+y)_+^{1/s} - x^{1/s}|\leq |y|^{1/s} \wedge (2|y||x|^{1/s-1})$. First, assume that $\|\bA\|_{s}\leq q^{1/s}[(p\vee s)(q\vee s)]^{1/4}$. 
\beqn
\E\left[|T_s - \|\bA\|_s|\right]&\leq& c\left[q^{1/s}[(p\vee s)(q\vee s)^{1/4}+  q^{1/(2s)}\|\bA\|^{1-1/s}_{s-1}+  \1_{s\geq 2} (pq)^{1/(2s)}\|\bA\|^{1-2/s}_{s-2}\right]\\ &\leq & c q^{1/s}[(p\vee s)(q\vee s)]^{1/4}\ ,
\eeqn 
where we used H\"older's inequality.  If $\|\bA\|_{s}\geq q^{1/s}[(p\vee s)(q\vee s)]^{1/4}$, then we have 
\beqn 
\E\left[|T_s - \|\bA\|_s|\right]&\leq& c^{s}\|\bA\|_s^{1-s}\left[q[(p\vee s)(q\vee s)]^{s/4}+  \sqrt{q}\|\bA\|_{s-1}^{s-1} +  \1_{s\geq 2}\sqrt{pq} \|\bA\|_{s-2}^{s-2}\right] \\
&\leq &  c^{s} q^{1/s}(pq)^{1/4}\ , 
\eeqn 
where we used again H\"older's inequality. 

\end{proof}

\begin{proof}[Proof of Lemma~\ref{lem:local_control_lambda_i}]
 Without loss of generality, assume that the matrix  $\bA$  is in diagonal form, that is $\bA_{ii}=\sigma_i(\bA)$ for $i=1,\ldots, q$ and $\bA_{ij}$ is zero otherwise. Then, we deduce from the interlacing inequality (Corollary III. 1.5 in~\cite{bhatia2013matrix}) that the $i$-th largest  eigenvalues of $\bW= \bY^T\bY-p\bI_q$ is less of or equal to to the first eigenvalue of the restriction of $\bW$ to indices in $[i,\ldots,  q]\times [i,\ldots,  q]$. Writing $\bW_{[i:q]}$ for this restriction, we arrive at 
\beqn 
 \lambda_i(\bW)&\leq& \lambda_1(\bW_{[i:q]})\\&\leq& \sigma^2_i(\bA) +\|[\bE^T \bA+ \bA^T \bE+ \bE^T\bE- p \bI_q]_{[i:q]} \|_{\infty}\\
 &\leq  & \sigma^2_i(\bA) + \|\bE^T\bE- p \bI_q\|_{\infty}+ 2\|[\bE^T \bA]_{[i:q]}\|_{\infty} \\
&\leq &  \sigma^2_i(\bA) + \|\bE^T\bE- p \bI_q\|_{\infty}+ 2\sigma_i(\bA)\|[\bE_{[1:q]}\|_{\infty}\ .
\eeqn

 \medskip 
 
 Turning the lower bound of $\lambda_i(\bW)$, we define $V_i$ the subspace of dimension $i$ spanned by the $i$-th largest eigenvectors of $\bA^T\bA$ (pick any such subspace if it is not unique). It follows from Courant-Fischer min-max theorem that 
\begin{eqnarray}
 \lambda_i(\bW)&\geq& \inf_{x\in  V_i} (x^T\bW x)/|x|^2_2\nonumber \\ 
 &\geq& \inf_{x: |x|_2=1 }  \sum_{j=1}^i x_j^2 \sigma_j^2(\bA) +  x^T (\bE^T \bE -p \bI_q) x +2 \sum_{j_1,j_2=1}^i x_{j_1}x_{j_2} \sigma_{j_1}(\bA)\bE_{j_1j_2}  \nonumber\\
 &\geq & \inf_{x: |x|_2=1 }  \sum_{j=1}^i x_j^2 \sigma_j^2(\bA)- 2[\sum_{j=1}^i x_j^2 \sigma_j^2(\bA)]^{1/2}\|\bE_{[1:q]}\|_{\infty}  - 
 \|\bE^T \bE -p \bI_q\|_{\infty}  \ ,  \label{eq:lower_bound_lambda_i_W}
\end{eqnarray}
where we applied Cauchy-Schwarz inequality and used that $\inf f+ g \geq \inf f + \inf g$. The quantity $\sum_{j=1}^i x_j^2 \sigma_j^2(\bA)$ lies in $[\sigma_i^2(\bA); \sigma_1^2(\bA)]$. 
The function $x\mapsto x^2 -2xz$ is decreasing for $x\leq z$ and increasing for $x\geq z$ and its minimum equals $-z^2$. When $\sigma_i(\bA)\geq \|\bE_{[1:q]}\|_{\infty}$, the minimum of  the left-hand side  expression in \eqref{eq:lower_bound_lambda_i_W} is achieved at $\sigma_i(\bA)$ and we have 
\beqn 
\lambda_i(\bW)&\geq& \sigma_i^2(\bA) - 2\sigma_i(\bA) \|\bE_{[1:q]}\|_{\infty}-  \|\bE^T \bE -p \bI_q\|_{\infty}\\ &\geq&   \sigma_i^2(\bA) - 2\sigma_i(\bA) \|\bE_{[1:q]}\|_{\infty}-  \|\bE^T \bE -p \bI_q\|_{\infty} - \|\bE_{[1:q]}\|^2_{\infty} \  . 
\eeqn
When  $\sigma_i(\bA)<\|\bE_{[1:q]}\|_{\infty}$, this implies that $\sigma_i^2(\bA) - 2\sigma_i(\bA) \|\bE_{[1:q]}\|_{\infty} < 0$, and we also have 
\[
\lambda_i(\bW)\geq   \sigma_i^2(\bA) - 2\sigma_i(\bA) \|\bE_{[1:q]}\|_{\infty}-  \|\bE^T \bE -p \bI_q\|_{\infty} - \|\bE_{[1:q]}\|^2_{\infty} \  .  
\]
All in all, we have proved that 
\[
  \big|\lambda_i(\bW)-  \sigma^2_i(\bA)\big|\leq  \|\bE^T\bE- p \bI_q\|_{\infty}+ 2\sigma_i(\bA)\|[\bE_{[1:q]}\|_{\infty}+ \|\bE_{[1:q]}\|^2_{\infty} \ .
\]
Since $\bE_{[1:q]}$ is a submatrix of $\Pi_{I(\bA)}(\bE)$ (recall that $\bA$ is assumed to be diagonal in the proof), we get the desired result.  
\end{proof}

\begin{proof}[Proof of Proposition~\ref{prp:upper_l1_spectrum}]
We only have to gather some arguments of Proposition~\ref{prp:upper_risk_nuclear_covariance}. Indeed, we have established in~\eqref{eq:upper_l_1} that 
\[
  \sum_{i=1}^q |[\lambda_i(\bW)]_{+}^{1/2}- \sigma_i(\bA)| 
  \leq  q(1+2\sqrt{2})\|\Pi_{I(\bA)}\bE\|_{\infty}+ q \|\bE^T\bE- p \bI_q\|^{1/2}_{\infty} \ .
\]
Besides, we have shown that $\E[\|\Pi_{I(\bA)}\bE\|_{\infty}]\lesssim q^{1/4}$ and $\E[\|\bE^T\bE- p \bI_q\|^{1/2}_{\infty}]\lesssim (pq)^{1/4}$. The result follows.

\end{proof}

\begin{proof}[Proof of Proposition~\ref{lem:risk_G_P}]
Since $U_k$ is an unbiased estimator of $\|\bA\|_{2k}^{2k}$, it follows that 
\beqn 
 \E[G_s[P; \bY]]&=& M^s(pq)^{s/4}\left[a_0+ \sum_{k=1}^{\deg(P)} a_k \left\|\frac{\bA}{M(pq)^{1/4}} \right\|^{2k}_{2k}\right]\\ &=&  M^s(pq)^{s/4} \sum_{j=1}^q P\big[\lambda_i\big(M^{-2}(pq)^{-1/2}\bA^T\bA\big)\Big]\ .
\eeqn 
 Assuming that $\|\bA\|_{\infty}\leq M(pq)^{1/4}$, we deduce that the bias of $G_s[P; \bY]$ satisfies 
\begin{eqnarray}
|\E[G_s[P; \bY]] - \|\bA\|_s^{s}|&\leq &  M^s(pq)^{s/4} \sum_{i=1}^q \left|\psi_s\left[\lambda_i\left(\frac{\bA^T\bA}{M^2(pq)^{1/2}}\right)\right] - P\left[\lambda_i\left(\frac{\bA^T\bA}{M^{2}(pq)^{1/2}}\right)\right]\right|  \nonumber \\ 
&\leq& M^s(pq)^{s/4}  q |\psi-P|_{ [0, 1],\infty}\ . \label{eq:upper_bias_polynomial}
\end{eqnarray}
Regarding the variance of $G_s[P; \bY]$, we deduce from Theorem~\ref{prp:risk_U_k} that 
\beqn 
 \var{G_s[P; \bY]}&\leq&  M^{2s}(pq)^{s/2}K\sum_{k=1}^K a_k^2 \frac{\var{U_k}}{M^{4k}(pq)^k} \\ 
 & \leq & M^{2s}(pq)^{s/2}K\sum_{k=1}^K c^{k} \frac{a_k^2}{M^{4k}} [1+ p^{1-k}q^{1-k} \lambda_1^{2k-2}(\bA^T\bA)+ q(pq)^{-k} \lambda_1^{2k-1}(\bA^T\bA)]\ ,
\eeqn 
where we used that $k\leq K\leq q^{1/3}$. Assuming that $\|\bA\|_{\infty}\leq M (pq)^{1/4}$, this simplifies in 
\[
 \frac{\var{G_s[P; \bY]}}{M^{2s}(pq)^{s/2}}\leq K \sum_{k=1}^K c^{k} a_k^2 \left(\frac{1}{M^{4k}}+ \frac{1}{M^4}+ \frac{\sqrt{q}}{\sqrt{p}M^2}\right) \lesssim \frac{K}{M^4\wedge[(p/q)^{1/2}M^2]}\sum_{k=1}^K c^{k} a_k^2\ , 
\]
which, together with~\eqref{eq:upper_bias_polynomial}, leads to the desired result. 
\end{proof}

The following simple result is used several times throughout the proofs.
\begin{lem}\label{lem:symmetrie_approximation}
Let $f$ be a symmetric function defined on $[-1,1]$. For any positive integer $n$, we have
\[
 E_{\cP_n}\left(f;[-1;1]\right)= E_{\cP^{sym}_n}\left(f;[0;1]\right)
\]
\end{lem}
\begin{proof}[Proof of Lemma~\ref{lem:symmetrie_approximation}]
First, we prove that  $E_{\cP_n}\left(f;[-1;1]\right)$ is achieved by a symmetric polynomial. Consider any Polynomial $P$. By symmetry of $f$, 
$\sup_{x\in [-1,1]}|[P(x)+P(-x)]/2 - f(x)|\leq \sup_{x\in [-1;1]}|f(x)-P(x)|$ and the symmetric polynomial $[P(x)+P(-x)]/2$ achieves a no higher approximation error. This implies that $  E_{\cP_n}\left(f;[-1;1]\right)= E_{\cP^{sym}_n}\left(f;[-1;1]\right)=  E_{\cP^{sym}_n}\left(f;[0;1]\right)$, where we used again the symmetry of $f$. 

\end{proof}

\begin{proof}[Proof of Lemma~\ref{lem:best_approximation_psi}]
 The first result is equivalent to $||x|- B_K(x)|_{\infty,[-1,1]}\leq \frac{2}{\pi(2K+1)}$ which has been proved in Bernstein's original paper~\cite{bernstein1913valeur}. The bound $|a^*|_{\infty}\leq 2^{3K}$ is proved in Lemma 2 of~\cite{cailow2011}. 
 Regarding the last result, we first apply Lemma~\ref{lem:symmetrie_approximation} to deduce that 
 \[
  E_{\cP_{2K}}[|x|;[-1,1]]= E_{\cP^{sym}_{2K}}[|x|;[0,1]]
 \]
Since any polynomial $P\in \cP_k$ satisfies $\sup_{x\in[0,1]}|P(x^2)-|x||= \sup_{x\in[0,1]}|P(x)-\sqrt{x}|$ and since any symmetric polynomial of degree at most $2K$ can be represented as $P(x^2)$ where $P\in \cP_K$, we deduce that 
 \[
 E_{\cP^{sym}_{2K}}[|x|;[0,1]]= E_{\cP_K}[\sqrt{x};[0,1]]
\]
Then, the result follows again from Bernstein's paper~\cite{bernstein1913valeur} on the approximation of the absolute function.

\end{proof}

\begin{proof}[Proof of Proposition~\ref{prp:risk_approximation}] 
 Taking $K^*=\lfloor 0.5(\log^{-1}((8c)\vee 2)) \log(q)\rfloor$
where $c$ is the constant in~\eqref{eq:upper_risk_G_P}, we have $(c2^{3})^{ K^*}\leq \sqrt{q}$. Then, we deduce from Proposition~\ref{lem:risk_G_P} and Lemma~\ref{lem:best_approximation_psi} that 
 \[
   \mathbb{E}\left[|G_1[P^*_{K^*}; \bY]- \|\bA\|_1|\right]\leq M  (pq)^{1/4}\left[\frac{2q}{\pi(2K^*+1)}+ c_1\sqrt{q} \right]\lesssim M  \frac{q}{\log(q)}  (pq)^{1/4}  \ .
 \]

\end{proof}

\begin{proof}[Proof of Proposition~\ref{prp:risk_approximation_general_s}]
Arguing as in the proof of Lemma~\ref{lem:best_approximation_psi}, we deduce that $E_{\cP_K}[x^{s/2};[0,1]]= E_{\cP_{2K}}[|x|^{s};[-1,1]]$.
Then, combining Proposition~\ref{lem:risk_G_P} with Lemma~\ref{lem:approximation_theory_polynom}, we deduce that, for any positive integer $K$
 \[
 \mathbb{E}\left[|G_s[P^*_{K;s}; \bY]- \|\bA\|^{s}_s|\right]\lesssim  M^s(pq)^{s/4} \left[ \frac{q}{K^s}+ \frac{c^{K}}{M^2\wedge [M(p/q)^{1/4}]} |a|_{\infty}\right]\ ,
 \]
where $|a|_{\infty}$ refers to the largest coefficient of $P^*_{K;s}$. Then,  we rely on Theorem E of~\cite{qazi2007some} to bound $|a|_{\infty}$: 
\begin{lem}[\cite{qazi2007some}]
Let $P_n(x)= \sum_{k=0}^n a_k x^{k}$  be a polynomial of degree at most $n$ such that $|P_n(x)|\leq 1$ for all $x\in [-1,1]$. Then,
$|a_{n-2\nu}|$ is bounded above by the modulus of the corresponding
coefficient of $L_n$ for $\nu= 0, 1,\ldots, \lfloor n/2\rfloor$, and $|a_{n-1-2\nu}|$ is
bounded above by the modulus of the corresponding coefficient
of $L_{n-1}$ for $\nu=0,1,\ldots, \lfloor (n-1)/2\rfloor$.  Here, we recall that  $L_n(x)$ is the $n$-th Chebychev polynomial. 
\end{lem}
Since $\sup_{x\in [-1,1]}|P^*_{K;s}(x^2) - |x|^{s}|\leq 1$ for $K$ large enough and since the largest coefficient of the Chebychev polynomial is less or equal to $2^{3K}$ (Lemma~\ref{lem:best_approximation_psi}), it follows that $|a|_{\infty}\leq c'^K$ and 
\[
 \mathbb{E}\left[|G_s[P^*_{K;s}; \bY]- \|\bA\|^{s}_s|\right]\lesssim  M^s(pq)^{s/4} \left[ \frac{q}{K^s}+ \frac{c^{K}}{M^2\wedge [M(p/q)^{1/4}]} \right]\ . 
\]
Taking $K= \lfloor \log^{-1}(c\vee 2)\log(q)\rfloor$ leads to the first result. The second result is a straightforward consequence of Jensen inequality and the inequality $|x^{1/s}-y^{1/s}|\leq |x-y|^{1/s}$.

\end{proof}

\subsection{Proof of Theorem~\ref{thm:lower_general}}

In comparison to the proof of Proposition~\ref{prp:lower_frob}, we now follow a generalized Le Cam's approach, by building two prior distributions $\mu_0$ and $\mu_1$ on matrices $\bA$ such that the corresponding integrated distributions of the observations $\mathbf{P}_{\mu_i}=\int \P_{\bA}\mu_i(d\bA)$ are close in total variation distance, while the distributions of $f_{\sigma}(\bA)$ under $\mu_0$ and $\mu_1$ are highly different. The statement of the following lemma is provided for a general functional $T$.

\begin{lem}[\cite{tsybakov_book}, Thm. 2.15]\label{lem:lecam_general}
Consider any function $T:\ \mathbb{R}^{p\times q}\mapsto \mathbb{R}$. Suppose that there exist $s>0$, $\zeta\in \mathbb{R}$ and a collection $\cA$ of matrices 
such that $\mu_0$ and $\mu_1$ are supported in $\cA$ and
\[\mu_0 \big[\bA:\  T(\bA)\leq \zeta-s\big]\geq 1- \beta_0\ ; \quad \quad 
\mu_1\big[\bA: \ T(\bA)\geq  \zeta+ s\big] \geq 1- \beta_1\ . 
\]
If $\|\mathbf{P}_{\mu_1} - \mathbf{P}_{\mu_0}\|_{TV}\leq \eta < 1$, then 
\[
 \inf_{\widehat{T}}\sup_{\bA\in \mathcal{A}}\P_{\theta}\left[|\widehat{T}- T(\bA)|\geq s\right] \geq \frac{1-\eta-\beta_0 - \beta_1}{2}\ . 
\]
\end{lem}

For a positive integer $n$, denote $\cU(n)$ the Haar measure on orthogonal matrices of size $n$. Let $\nu_0$ and $\nu_1$ be two probability measures supported  on $[0,0.125(pq)^{1/4}]$ whose definition is postponed at the end of the proof. 
Given a vector $\theta$ of size $\lceil q/2\rceil$, Let $\bD(\theta)$ denote the $p\times q$ matrix such that $[\bD(\theta)]_{aa}=\theta_a$ for $a\leq \lceil q/2\rceil$ and is zero otherwise.

For $i=0,1$, consider the prior distributions  $\mu_0$ and $\mu_1$ such that, under $\mu_i$,  $\bA$ is distributed as $\bU\bD(\theta) \bV^T $ where $\bU$, $\bV$, and $\theta$ are independent and respectively sampled according to the Haar measures $\cU(p)$ and $\cU(q)$ and $\nu_i^{\otimes \lceil q/2\rceil}$. Note that, up to a permutation,  the non-zero singular values of $\bA$ equal $\theta_a$ for $a=1,\ldots, \lceil q/2\rceil$, which implies that $f_{\sigma}(\bA)= \sum_{a}f(|\theta_a|)$.

\medskip

The measure $\mathbf{P}_{\mu_0}$ and $\mathbf{P}_{\mu_1}$ are challenging to handle because the corresponding matrices $\bA$ have ranks $\lfloor q/2\rfloor$ (if $\nu_0(\{0\})=0$) and  involve intricate integrals with respect to the Haar measure. The following lemma reduces $\|\mathbf{P}_{\mu_0}- \mathbf{P}_{\mu_1}\|_{TV}$ to distance between distributions involving rank 1 matrices $\bA$. 

Define $q'= q+1 - \lceil q/2\rceil$ , $p'= p+1- \lceil q/2\rceil$. Let $\underline{\mu}_0$ (resp. $\underline{\mu}_1$) denote the distribution on $\bA\in \mathbb{R}^{p'\times q'}$ such that $\bA= \sigma u v^T$ where $\sigma\sim \nu_0$ (resp. $\nu_1$), $u$ and $v$ are sampled uniformly on the unit sphere of $\mathbb{R}^{p'}$ and $\mathbb{R}^{q'}$. Finally, we denote  $\underline{\mathbf{P}}_{i}= \int \P'_{\bA}\underline{\mu}_i(d\bA)$ for $i=0,1$ the marginal distributions of the noisy $p'\times q'$ observations $\bY$ when $\bA\sim \underline{\mu}_i$.
\begin{lem} \label{lem:reduction}
For  any $\nu_0$ and $\nu_1$, we have 
\[
\|\mathbf{P}_{\mu_0}- \mathbf{P}_{\mu_1}\|_{TV}\leq \lceil q/2 \rceil  \|\underline{\mathbf{P}}_{0}- \underline{\mathbf{P}}_{1}\|_{TV} \ . 
\]
\end{lem}
Intuitively, $\|\mathbf{P}_{\mu_0}- \mathbf{P}_{\mu_1}\|_{TV}$ quantifies the difficulty to decipher rank $\lceil q/2\rceil $ matrices whose singular values are either i.i.d. distributed according to $\nu_0$ or $\nu_1$, whereas $\|\underline{\mathbf{P}}_{0}- \underline{\mathbf{P}}_{1}\|_{TV}$ quantifies the difficulty to decipher rank $1$ matrices whose unique non-zero singular value is either distributed according to $\nu_0$ or to $\nu_1$. Fortunately, the distance  $\|\underline{\mathbf{P}}_{0}- \underline{\mathbf{P}}_{1}\|_{TV}$ is now tractable and we are now in position to apply moment matching techniques~\cite{cailow2011}.

Define $k_0=k_0(\nu_0,\nu_1)$ be the smallest integer such that, for all $0\leq k\leq k_0$, $\int x^{2k} \nu_0(dx)= \int x^{2k}\nu_1(dx)$. In other words, the $k_0$ first even moments of $\nu_0$ and $\nu_1$ are matching. Note that $k_0\geq 0$ as $\nu_0$ and $\nu_1$ are probability measures.
\begin{lem}\label{lem:controle_distance_variation_totale}
If the supports of $\nu_0$ and $\nu_1$ are included in $[0,0.125(pq)^{1/4}]$ and if the first $k_0$ even moments of $\nu_0$ and $\nu_1$ are matching, then we have 
\[
\|\underline{\mathbf{P}}_{0}- \underline{\mathbf{P}}_{1}\|^2_{TV}\leq  e^{-k_0}\ . 
\]
\end{lem}

Define $k_0^*:= \lceil \log(4\lceil q/2\rceil )\rceil$. Provided that  $k_0\geq k_0^*$, we then  derive from Lemma~\ref{lem:reduction} that 
\beq\label{eq:upper_TV}
\|\mathbf{P}_{\mu_0}- \mathbf{P}_{\mu_1}\|_{TV}\leq   e^{-k_0}\lceil q/2 \rceil \leq  \frac{1}{4} \ . 
\eeq

\medskip

Since the support of $\nu_0$ and $\nu_1$ are included in $[0,(pq)^{1/4}/8]$ we  deduce from Hoeffding inequality 
\[
\nu_0^{\otimes\lceil q/2\rceil}\left[\sum_{i=1}^{\lceil q/2\rceil} f(|\theta_i|)\geq \lceil q/2\rceil\int f(t) \nu_0(dx) + (\lceil q/2\rceil\log(2))^{1/2}\|f\|_{\infty,[0,(pq)^{1/4}/8]}\right]\leq \frac{1}{4} 
\]
A similar bound holds for the left deviations of $\sum_{i}f(\theta_i)$ under $\nu_1^{\otimes \lceil q/2\rceil}$. Using \eqref{eq:upper_TV}, we are now in position to apply Lemma~\ref{lem:lecam_general}. This leads us to 
\begin{eqnarray}\nonumber
 \lefteqn{\inf_{\widehat{T}}\sup_{\bA ;\, \|\bA\|_{\infty}\leq 0.125(pq)^{1/4}}\E\left[\big|\widehat{T}- f_{\sigma}(\bA)\big|\right]}&& \\ &\geq &\frac{1}{16} \left[\lceil q/2\rceil\left[\int f(t) \left(\nu_1(dt)-\nu_0(dt)\right)\right]  - 2(\lceil q/2\rceil\log(2))^{1/2}\|f\|_{\infty,[0,(pq)^{1/4}/8]} \right]_+\ , \label{eq:resultat_intermediaire}
 \end{eqnarray}
provided that the even moments up to  $k_0^*$  of $\nu_0$ and $\nu_1$ are matching. To conclude, it remains to consider $\nu_0$ and $\nu_1$ in such a way that
(a) $\nu_0$ and $\nu_1$ are supported in $[0,(pq)^{1/4}/8]$, (b) the first even moments up to $k_0^*$ of $\nu_0$ and $\nu_1$ are matching, and (c)  $\int f(t) (\nu_1(dt) -  \nu_0(dt))$ is the largest possible. Fortunately, such  extremum problem  is now well understood. 
The following lemma is a straightforward extensions of Lemma 1 in~\cite{cailow2011} (see also Lemma 5.4 in~\cite{mukherjee2017}). The proof being the same as in Cai and Low~\cite{cailow2011}, we omit it.
\begin{lem}\label{lem:existence_extremal_measure_moment_method}
For any bounded interval $I\subset \mathbb{R}$, any continuous function $f$ on $I$, any collection $(g_1,\ldots, g_q)$ continuous functions on $I$ with $g_1=1$,
there exist two probability measures $\nu_0$ and $\nu_1$ such that 
\begin{enumerate}
 \item $\int_{I} g_l(t)\nu_0(dt)= \int_{I} g_l(t)\nu_1(dt)\ ,\quad \text{ for all } l=1,\ldots, q $ 
 \item $\int_{I}f\nu_1(dt) - \int_{I}f\nu_0(dt) = 2E_{G}[f;I]$, 
\end{enumerate}
where $G= \mathrm{Vect}(g_1,\ldots, g_q)$.  Furthermore, it is not possible to built $\nu_0$ and $\nu_1$ satisfying (1) and $\int_{I}f\nu_1(dt) - \int_{I}f\nu_0(dt) > 2E_{G}[f;I]$. 
\end{lem}
We apply this lemma with $G= \cP^{sym}_{2k^*_0}$. Together with~\eqref{eq:resultat_intermediaire}, this concludes the proof.

\begin{proof}[Proof of Lemma~\ref{lem:controle_distance_variation_totale}]
  Let $\phi$ denote the density of the  standard normal normal distribution. With a slight abuse of notation, we define $\phi(\bB)= (2\pi)^{-pq/2}e^{-\|\bB\|_2^2/2}$ for a $p'\times q'$ matrix $\bB$. 
 Relying on Fubini Theorem and Cauchy-Schwarz inequality, we get 
\beqn 
\|\underline{\mathbf{P}}_{0}- \underline{\mathbf{P}}_{1}\|_{TV}&=& \int \big|\int \phi(\bY-\bA) [\underline{\mu}_0(d\bA) -\underline{\mu}_1(d\bA) ]   \Big| d\bY\\
& =& \int \phi(\bY) \big|\int e^{\langle \bY, \bA\rangle - \|\bA\|_2^2 /2} [\underline{\mu}_0(d\bA) -\underline{\mu}_1(d\bA) ]   \Big| d\bY \\
&\leq & \left[\int  \left[\int e^{\langle \bY, \bA\rangle - \|\bA\|_2^2 /2} [\underline{\mu}_0(d\bA) -\underline{\mu}_1(d\bA) ]   \right]^{2} \phi(\bY)d\bY\right]^{1/2}\\
 && =\left[ \sum_{j=0,1}\sum_{l=0,1}\int \left( \int \phi(\bY)e^{\langle \bY, \bA+\bA'\rangle - \|\bA\|_2^2 /2-  \|\bA'\|_2^2 /2} d\bY\right)(-1)^{j+l} \underline{\mu}_j(d\bA)\underline{\mu}_l(d\bA)\right]^{1/2} \\ 
 &&= \left[ \sum_{j=0,1}\sum_{l=0,1}(-1)^{j+l}e^{\langle \bA, \bA'\rangle}\underline{\mu}_j(d\bA)\underline{\mu}_l(d\bA)\right]^{1/2}\enspace .
\eeqn 
The inner product $\langle  \bA, \bA'\rangle$ decomposes as  $\langle  \bA, \bA'\rangle = \eta \eta' (u^T u') (v^T v')$. Since $u$ and $u'$ are independently sampled on the unit sphere, $(u^{T}u')$ is distributed as the first coordinate $u_1$. Similarly  $v^T v'$ is distributed as the first coordinate $v_1$. Decomposing $e^{\langle A, A'\rangle}= \sum_{k=0}^{\infty} (k!)^{-1} \eta^k \eta'^k (u^T u')^k (v^T v')^k$, this leads us to
\beqn 
\|\underline{\mathbf{P}}_{0}- \underline{\mathbf{P}}_{1}\|^{2}_{TV}&\leq &\sum_{k=0}^{\infty}\frac{1}{k!}\E[u_1^{k}]\E[v_1^{k}]\left[\sum_{j=0,1}\sum_{l=0,1}(-1)^{j+l} \int \eta^k \nu_j(d\eta)  \int \eta^{'k}\nu_l(d\eta') \right]\\
&\leq & \sum_{k=0}^{\infty}\frac{1}{(2k)!}\E[u_1^{2k}]\E[v_1^{2k}]\left(\int \eta^{2k} (\nu_0(d\eta)-\nu_1(d\eta))\right)^2\ , 
\eeqn 
since the distributions of $u_1$ and $v_1$  are symmetric. Since the $k_0$ first even moments of $\nu_0$ and $\nu_1$ are matching and since those measures are supported in $[0,0.125(pq)^{1/4}]$, this simplifies as 
\[
\|\underline{\mathbf{P}}_{0}- \underline{\mathbf{P}}_{1}\|^{2}_{TV}\leq  \sum_{k=k_0}^{\infty}\frac{\E[u_1^{2k}]\E[v_1^{2k}]}{(2k)!} \left(\frac{1}{8}\right)^{4k}(pq)^{k}\leq  \sum_{k=k_0}^{\infty}\frac{1}{\sqrt{4\pi k}}\cdot \frac{\E[u_1^{2k}(ep)^k]}{2^{7k}k^k} \cdot \frac{\E[v_1^{2k}(eq)^k]}{2^{7k}k^k} \ . 
\]
The random variable $u_1^2$ is distributed as $X^2/[X^2+Z^2]$ where $X\sim \cN(0,1)$ and $Z\sim \chi^2(p'-1)$. As a consequence, $\E[u_1^{2k}]\leq \frac{4^k\E[X^{2k}]}{(p')^k}+ \P[X^2 + Z^2 \leq p'/4]$. From Lemma~\ref{lem:chi2}, we deduce that  $\P[X^2+ Z^2 \leq p'/4]\leq e^{- p'/12}$, for any $p'\geq 1$. Besides, 
\[
\E[X^{2k}]= \frac{2^{k}}{\sqrt{\pi}}\Gamma(k+1/2)\leq  \frac{2^k}{\sqrt{\pi}} k!\leq e\sqrt{k}\left(\frac{2k}{e}\right)^{k}\ ,
\]
where we used Stirling's upper bound. 
Coming back to $\E[u_1^{2k}]$, we deduce from Stirling lower bound that 
\[
\frac{\E[u_1^{2k}(ep)^k]}{2^{7k}k^k}\leq e\sqrt{k}2^{-4k}\left(\frac{p}{p'}\right)^k+ 2^{-2k}\left(\frac{pe}{2^5k}\right)^{k}e^{-p'/12}\ . 
\]
By definition of $p'$, we have $p'\geq p/2$. 
The function $x \mapsto (\tfrac{pe}{2^6 x})^{x}$ achieves its maximum for $x=p/2^5$. As a consequence, we derive that 
\[
  \frac{\E[u_1^{2k}(ep)^k]}{2^{7k}k^k}\leq e\sqrt{k}2^{-3k}+ 2^{-2k}e^{p/32-p/24}\leq  2^{-2k}\left[1+ e\sqrt{k}2^{-k}\right]\leq 2^{-2k}(e/2+1)
\]
Arguing similarly for $v_1$, we arrive at 
\beqn 
\|\underline{\mathbf{P}}_{0}- \underline{\mathbf{P}}_{1}\|^2_{TV}&\leq& \sum_{k=k_0}^{\infty}\frac{(e/2+1)^2}{\sqrt{4\pi k}} 2^{-4k}\leq  3\cdot  2^{-4k_0}\leq e^{-k_0}\ .
\eeqn

\end{proof}

\begin{proof}[Proof of Lemma~\ref{lem:reduction}]
 In this proof, we interpolate the measures $\mu_0$ and $\mu_1$ by $\lceil q/2\rceil-1$ intermediary measures. 
  For $j=1,\ldots, \lceil q/2\rceil-1$, define $\mu^{(j)}$ the distribution on $\bA$ such that $\bA= \bU \bD(\theta) \bV^T$ where $\bU\sim \cU(p)$, $\bV\sim \cU(q)$, and $\theta\sim \nu_0^{\otimes ( \lceil q/2 \rceil-j )}\otimes \nu_1^{\otimes j}$, so that the distributions $\mu^{(j)}$ and $\mu^{(j+1)}$ only differ though the distribution of one entry. Denote $\mathbf{P}_{\mu^{(j)}}=\int \P_{\bA}\mu^{(j)}(d\bA)$ the corresponding marginal distributions on $\bY$.  Writing $\mu_0= \mu^{(0)}$ and $\mu_1=\mu^{(\lceil q/2\rceil)}$, we derive from triangular inequality that 
 \[
 \|\mathbf{P}_{\mu_0}- \mathbf{P}_{\mu_1}\|_{TV}\leq \sum_{j=0}^{\lceil q/2\rceil-1 }\|\mathbf{P}_{\mu^{(j)}}- \mathbf{P}_{\mu^{(j+1)}}\|_{TV}\ . 
\]
 We shall bound each of these total variation distance independently. Fix any integer $0\leq j\leq \lceil q/2\rceil-1$.

 Define the set $S=\{1,\ldots, \lceil q/2\rceil- j-1,\lceil q/2 \rceil- j+1,\ldots, \lceil q/2\rceil\}$.  Denote $\bU_{[S]}$, and $\bV_{[S]}$ the restrictions of $\bU$ and $\bV$ to its columns in $S$. If $\bU$ is sampled on the Haar measure, then conditionally to $\bU_{[S]}$, the $\lceil q/2 \rceil- j$-th column of $\bU$ is sampled uniformly on the intersection of the unit sphere of $\mathbb{R}^p$ and  the orthogonal subspace $U$ of that induced by the columns of  $\bU_{[S]}$. The same Property holds for the $\lceil q/2 \rceil- j$-th column of $\bV$. We respectively write $u$ and $v$ for these two columns.

 As a consequence, under $\mu^{(j)}$ (resp. $\mu^{(j+1)}$), $\bA$ decomposes as 
 \beq
  \bA = \eta u v^T + \bU_{[S]}^T\bD(\theta_{-j}) \bV_{[S]}=:\eta u v^T + \bB \ , 
 \eeq
where $\eta\sim \nu_0$ (resp. $\nu_1$),  conditionally to $\bU_{[S]}$ and $\bV_{[S]}$ and $\theta_{-j}\sim \nu_0^{\otimes \lceil q/2\rceil -j-1 }\otimes\delta_0 \otimes \nu^{\otimes  j }$. Writing $\underline{\pi}$ for the distribution of $(\bU_{[S]},  \bV_{[S]}, \bD(\theta_{-j}))$ and  $\pi_0$ (resp.$\pi_1$) for the conditional distribution of ($u$, $v$, $\eta$) given  $(\bU_{[S]},  \bV_{[S]}, \bD(\theta_{-j}))$ where $\eta\sim \nu_0$ (resp. $\nu_1$), we decompose the total variation distance  as follows 
\beqn 
\lefteqn{\|\mathbf{P}_{\mu^{(j)}}- \mathbf{P}_{\mu^{(j+1)}}\|_{TV}}&& \\ & =& \int \Big|  \int \phi(\bY-\bA  )\mu^{(j)}(d\bA) - \int \phi(\bY - \bA)\mu^{(j+1)}(d\bA) \Big|d\bY\\
&= &  \int \Big|  \int \phi\left[\bY- \eta u v^T -\bB   \right]\pi_0(du,dv,d\eta) \underline{\pi}(\bU_{[S]}^T,\bV_{[S]},\theta_{-j}) \\ &&  \quad - \int(\phi\left[\bY -\eta'u'  v^{'T} - \bB)\bV_{[S]}\pi_1(du',dv',d\eta'\right] \underline{\pi}(\bU_{[S]}^T,\bV_{[S]},\theta_{-j}) \Big|d\bY\\
&\leq & \int \left|\int \phi\left[\bY- \eta u v^T -  \bB \right]\pi_0(du,dv,d\eta) - \int \phi\left[\bY- \eta' u' v^{'T} - \bB \right]\pi_1(du',dv',d\eta')\right| \\ &&\quad \quad\quad \underline{\pi}(d\bU_{[S]}^T,d\bV_{[S]},d\theta_{-j}) d\bY\\
&\leq & \int \left|\int \phi\left[\bY- \eta u v^T \right]\pi_0(du,dv,d\eta) - \int \phi\left[\bY- \eta' u' v^{'T}  \right]\pi_1(du',dv',d\eta')\right| d\bY\underline{\pi}(d\bU_{[S]}^T,d\bV_{[S]},d\theta_{-j})  \ ,
\eeqn
where in the last line we replaced $\bY$ by $\bY-\bB$. The above integral does not depend on specific values of $U_{[S]}$ and $V_{[S]}$ since the normal distribution is invariant by orthogonal transformation.

Define $q'= q+1 - \lceil q/2\rceil$ , $p'= p+1- \lceil q/2\rceil$. Let $\underline{\mu}_0$ (resp $\underline{\mu}_1$) denote the distribution on $\bA$ such that $\bA= \sigma u v^T$ where $\sigma\sim \nu_0$ (resp. $\nu_1$), $u$ and $v$ are sampled uniformly on the unit sphere of $\mathbb{R}^{p'}$ and $\mathbb{R}^{q'}$. Then, upon considering the integrated distribution $\underline{\mathbf{P}}_{i}= \int \P_{\bA}\underline{\mu}_i(d\bA)$ for $i=0,1$, we have proved that 

\[
 \|\mathbf{P}_{\mu^{(j)}}- \mathbf{P}_{\mu^{(j+1)}}\|_{TV}\leq \|\underline{\mathbf{P}}_{0}- \underline{\mathbf{P}}_{1}\|_{TV}\ . 
\]
The result follows.

\end{proof}

\subsection{Proof of Corollary~\ref{prp:lower_general_norm}}

We start with $\|\bA\|_s^{s}$. First observe that the minimax estimation risk  is non-decreasing in $q$. Indeed, for $q_1\leq q_2$ one can extend each matrix $\bA\in \mathbb{R}^{p\times q_1}$ in $\bA'\in \mathbb{R}^{p\times q_2}$ by adding null column vectors. Since both $\bA$ and $\bA'$ share the same non-zero singular values we have
 \[
  \inf_{\widehat{T}}\sup_{\bA\in\mathbb{R}^{p\times q_1}:\ \|\bA\|_{\infty}\leq 0.125(pq_1)^{1/4} }\E\left[\big|\widehat{T}- \|\bA\|_s^s\big|\right]\leq \inf_{\widehat{T}}\sup_{\bA\in\mathbb{R}^{p\times q_2}:\ \|\bA\|_{\infty}\leq 0.125(pq_2)^{1/4} }\E\left[\big|\widehat{T}- \|\bA\|_s^{s}\big|\right]\ . 
 \]
 Besides, for $q=1$, $\bA$ has rank at most one and $\|\bA\|^s_s =\|\bA\|^{s}_2$. It then readily follows from the proof of Proposition~\ref{prp:lower_frob} that, for all $q\leq p$, 
 \[
   \inf_{\widehat{T}}\sup_{\bA \in \mathbb{R}^{p\times q}}\E\left[\big|\widehat{T}- \|\bA\|_s^{s}\big|\right]\geq c p^{s/4}
 \]
Thus, it suffices to prove~\eqref{eq:lower_nuclear_norm} for $q$ large enough. We apply Theorem~\ref{thm:lower_general} with the function $x\mapsto |x|^s$. Since for any $x$ and $M>0$, $|Mx|^s= M^s|x|^s$, we have 
$E_{\cP^{sym}_{2k^*}}(|x|^s;I_0])= \left((pq)^{1/4}/8\right)^s E_{\cP^{sym}_{2k^*}}\left(|x|^s;[0;1]\right)$. Since this function is symmetric, it follows from Lemma~\ref{lem:symmetrie_approximation} that  $E_{\cP^{sym}_{2k^*}}[|x|^s;[0;1]]=E_{\cP_{2k^*}}\left(|x|^s;[-1;1]\right)$. 
\[
 \inf_{\widehat{T}}\sup_{\bA: \ \|\bA\|_{\infty}\leq 0.125(pq)^{1/4}}\E\left[\big|\widehat{T}- \|\bA\|_s^s\big|\right]\geq c'_s (pq)^{s/4}\left[q E_{\cP_{2k^*}}[f;[0.1]]- q^{1/2}\right]_+\ . 
\]
Then, we deduce from Lemma~\ref{lem:approximation_theory_polynom}  that, for $q$ large enough, $E_{\cP_{2k^*}}\left(|x|^s;[-1;1]\right)\geq c/\log^s(q)$. The first result follows. 

Let us prove the second result by contradiction. Assume that there exists an estimator $\widehat{T}$ and a constant $c_s$ such that 
\[
 \sup_{\bA: \ \|\bA\|_{\infty}\leq 0.125(pq)^{1/4}}\E\left[\big|\widehat{T}- \|\bA\|_s\big|\right]\leq c_s \frac{q^{1/s}(pq)^{1/4}}{\log^{s}(q)\vee 1}
\]
Without loss of generality, we can also assume that $\widehat{T}\leq 0.125q^{1/s}(pq)^{1/4}$. Then, for any matrix $\bA$ with $\|\bA\|_{\infty}\leq 0.125(pq)^{1/4}$, we have 
\[
 \E\left[\big|\widehat{T}^{s}- \|\bA\|^{s}_s\big|\right]\leq s  [q^{1/s}(pq)^{1/4}]^{s-1}\E\left[\big|\widehat{T}- \|\bA\|_s\big|\right]\leq 
 c_s s \frac{q(pq)^{1/4}}{\log^s (q)}\ .
\]
Hence, if $c_s$ is take small enough, this contradicts the first result.

\subsection{Proofs for Wasserstein estimation}

We start with the key lemma in~\cite{kong2017spectrum} that allows them to control the Wasserstein distance between the spectral measures.

\begin{lem}\cite{kong2017spectrum}\label{lem:kong_valiant_original} Consider two probability measures  $\nu_1$ and $\nu_2$ supported on $[-1,1]$. Denote their first $k$ moments 
$\alpha =(\alpha_1,\ldots, \alpha_k)$ and $\beta= (\beta_1,\ldots, \beta_k)$, respectively. Then, 
\[
 W_1(\nu_1,\nu_2)\leq \frac{c}{k}+ c'3^k |\alpha-\beta|_2\ ,
\]
where $c$ and $c'$ are numerical constant. 
 
\end{lem}
It turns out that we can extend the above approximation bound by only considering the sequence of even moments. 

\begin{lem}\label{lem:kong_valiant:version_valeurs_singuliere} Consider two probability measures  $\nu_1$ and $\nu_2$ supported on $[0,1]$. Denote their first $k$ {\bf even} moments 
$\alpha =(\alpha_1,\ldots, \alpha_k)$ and $\beta= (\beta_1,\ldots, \beta_k)$, respectively. Then, 
\[
 W_1(\nu_1,\nu_2)\leq \frac{c}{2k}+ c'3^{2k} |\alpha-\beta|_2\ ,
\]
where $c$ and $c'$ are the same numerical constants as in Lemma~\ref{lem:kong_valiant_original}.

 \end{lem}
 .
 
\begin{proof}[Proof of Lemma~\ref{lem:kong_valiant:version_valeurs_singuliere}]

We start from  Kantorovich-Rubinstein variational representation of the Wasserstein distance:
$ W_{1}(\nu_1,\nu_2)=\sup_{f: Lip(f)\leq 1}\int f(x) (\nu_1(dx) - \nu_2(dx))$, where $f$ is supported on $[0,1]$ and $Lip(f)\leq 1$ means that $f$ is a Lipschitz function with constant $1$. Let $\cP^{sym}_{k}$ denote the space of the symmetric polynomials of degree less or equal to $k$. As Kong and Valiant~\cite{kong2017spectrum}, we approximate   
$\int f(x)(\nu_1(dx) - \nu_2(dx))dx$ by $\int P(x)(\nu_1(dx) - \nu_2(dx))dx$ where $P= \sum_{i=0}^{k} a_k x^{2k} \in \cP^{sym}_{k}$, the twist being that we now consider symmetric polynomials. 
\beqn 
 \int f(x)(\nu_1(dx) - \nu_2(dx))dx & = &  \int (f(x)-P(x))(\nu_1(dx) - \nu_2(dx))dx+ \int P(x)(\nu_1(dx) - \nu_2(dx))dx
 \\
 &\leq&  2\|f-P\|_{\infty}+ \sum_{i=1}^{k}|a_i||\alpha_i-\beta_i|\ . 
\eeqn 
Coming back to the Wasserstein distance, we have 
\beq\label{eq:objective_wasserstein}
 W_{1}(\nu_1,\nu_2)\leq \sup_{f: Lip(f)\leq 1} \inf_{P\in \cP^{sym}_{k}} 2\|f-P\|_{\infty} + \sum_{i=1}^{k}|a_i||\alpha_i-\beta_i|\ . 
\eeq
Fix any such Lipschitz function $f$ defined in $[0,1]$ and consider the symmetric function $\overline{f}$ on $[-1,1]$, such that $\overline{f}(x)=f(|x|)$. In the proof of Proposition 1 in~\cite{kong2017spectrum}, Kong and Valiant show that there exists a polynomial $Q_{2k}(x)= \sum_{i=0}^{2k} c_i x^{i}$ 
such that $\|Q-\overline{f}\|_{\infty}\leq c/(2k)$ ($c$ is the same universal contant as in Lemma~\ref{lem:kong_valiant_original}) and $(\sum_{i=1}^k c_i^2)^{1/2}\leq 3^{2k}/2$. Since $\overline{f}$ is symmetric, the symmetric  polynomial $Q^{sym}= \sum_{i=0}^{k}c_{2i}x^{2i}$ satisfies 
\beqn
\|Q^{sym}_{2k}-\overline{f}\|_{\infty}&=&  \sup_{x\in [-1,1]}\left| \frac{Q_{2k}(x)+ Q_{2k}(-x)}{2}-\overline{f}(x)\right|= \sup_{x\in [-1,1]}\left| \frac{Q_{2k}(x)+ Q_{2k}(-x)-\overline{f}(x)-\overline{f}(-x)}{2}\right|
\\
&\leq& \|Q- \overline{f}\|_{\infty}\leq \frac{c}{2k}\ .
\eeqn
In particular, this implies that $\|Q- f\|_{\infty}\leq c/(2k)$. Coming back to~\eqref{eq:objective_wasserstein}, we have proved that 
\[
 W_{1}(\nu_1,\nu_2)\leq \frac{c}{2k}+  \frac{3^{2k}}{2}|\alpha-\beta|_2 \enspace . 
\]

\end{proof}

\begin{proof}[Proof of Theorem~\ref{thm:consistent_singular_values}]
The proof closely follows that of Theorem 2 in Kong and Valiant~\cite{kong2017spectrum}.  Define $\overline{m}$ the vector of size $K$ of even moments such that $\overline{m}_k= \frac{1}{q}\big[\frac{\|\bA\|_{2k}}{M(pq)^{1/4}}\big]^{2k}$, for $k=1,\ldots, K$. Since we assume that $\sigma_1(\bA)\leq M(pq)^{1/4}$, it follows from Proposition~\ref{prp:risk_frobenius} and Theorem~\ref{prp:risk_U_k} that 
\[
 \E\left[|\overline{m}- \widehat{m}|_1\right] \leq \frac{1}{q} \sum_{k=1}^K c^{k}\left[\frac{1}{M^2}+ \left(\frac{q}{p}\right)^{1/4}\frac{1}{M}\right]\leq \frac{c^{K}}{q[M^2\wedge [(p/q)^{1/4}M]]}\ . 
\]
Consider a rounding $\mu_0$ of the true measure $\mu_{\sigma(\bA)/[M(pq)^{1/4}]}$ that is supported on the grid. By definition, the $2k$-th moments of $\mu_0$ is at most $[1- (1-\zeta)^2k]$ far from that of $\mu_{\sigma(\bA)/[M(pq)^{1/4}]}$. As a consequence, this distribution $\mu_0$ is a feasible point of the linear program and its objective value is less or equal to $|\overline{m}-\widehat{m}|_1+ 2\zeta \sum_{k=1}^{K} k$. 
As a consequence, the even moments vector $m^+$ of the  solution $p^+$ of the linear program satisfies in expectation
\beqn 
\E[|m^+-\overline{m} |_2]&\leq&\E[|m^+-\widehat{m}|_1]+ \E[\overline{m}- \widehat{m}|_1]\\
&\leq & \frac{c^{K}}{q}\frac{1}{M^2\wedge [(p/q)^{1/4}M]}+c \zeta K^{2} \\
&\leq & 2\frac{c^{K}}{q}\frac{1}{M^2\wedge [(p/q)^{1/4}M]}\ ,
\eeqn 
since $\zeta\leq 1/(M^2q)$. Denote $\mu^+$ the measure associated to the vector $p^+$ and  $\mu^+_{disc}$ the discrete distribution made of  $q$ Dirac which is derived from $p^+$ after step $2$ of the procedure. By triangular inequality, we have 
 $W(\mu^+_{disc}, \mu_{\sigma(\bA)})\leq W(\mu^+, \mu_{\sigma(\bA)})+ W(\mu^+_{disc}, \mu^+)$. Since $W(\mu^+_{disc}, \mu^+)\leq 1/q$, we deduce from Lemma~\ref{lem:kong_valiant:version_valeurs_singuliere} that 
 \[
  \E[W(\mu^+_{disc}, \mu_{\sigma(\bA)})] \leq  \frac{c}{K}+\frac{1}{q}+  \frac{c'^{K}}{q}\frac{1}{M^2\wedge [(p/q)^{1/4}M]} \ . 
 \]
From~\eqref{eq:connection_Wasserstein_l1}, we conclude that 
\[
 \sum_{i=1}^q |\widehat{\sigma}_i-\sigma_i(\bA)|\leq qM(pq)^{1/4} \left[ \frac{c}{K}+\frac{1}{q}+  \frac{c'^{K}}{q}\frac{1}{M^2\wedge [(p/q)^{1/4}M]} \right]\ . 
\]
Plugging the value $K = \lfloor c_1\log(q)\rfloor $ where $c_1$ is chosen small enough so that $c'^{K}\leq \sqrt{q}$ yields the desired result.

\end{proof}

\subsection{Proof of Propositions~\ref{prp:effective_rank_estimation} and \ref{prp:rank_estimatino_optimalite}}

\begin{proof}[Proof of Proposition~\ref{prp:effective_rank_estimation}]

 Recall that $\tr[\bY^T \bY- p\bI_q]= \|\bA\|_2^2+2 \tr[\bA^T \bE]+ (\tr[\bE^T\bE]-pq)$. The second expression follows a centered normal distribution with variance $4\|\bA\|_2^2$ whereas the last expression follows a $\chi^2$ distribution with $pq$ degrees of freedom. 
 From Lemma~\ref{lem:chi2}, we deduce that, with probability higher than $1-4e^{-t}$, 
 \[
 \big|\tr[\bY^T \bY- p\bI_q]- \|\bA\|_2^2 \big| \leq 2\sqrt{pqt}+ 2t + 2\|\bA\|_2\sqrt{2t}\ .
 \]
 Regarding the operator norm, we start from \eqref{eq:upper_error_operator}
\[
 \big|\sigma_1(\bY^T \bY -p \bI_q) - \|\bA^T\bA\|_{\infty}\big| \leq 2\|\bA^T\bE\|_{\infty}+ \|\bE^T\bE - p\bI_q\|_{\infty}\ . 
\]
From Lemma~\ref{lem:DavSza}, we deduce that, with probability higher than $1-e^{-t}$,  $\|\bE^T\bE - p\bI_q\|_{\infty}\leq q+ 2\sqrt{pq}+ 4\sqrt{2pt}+ 2t$. In the proof of Lemma~\ref{croise}, we have shown that $ \|\bA^T\bE\|_{\infty}/\|\bA\|_{\infty}$ is stochastically dominated by the operator norm of a $q\times q$ matrix with independent normal entries. Invoking again Lemma~\ref{lem:DavSza}, we derive that $2\|\bA^T\bE\|_{\infty}\leq 2\sigma_1(\bA)(2\sqrt{q}+\sqrt{2t})$ with probability higher than $1-e^{-t}$. Putting everything together, this yields 
\[
     \frac{1 - [2\sqrt{pqt}+2t]/\|\bA\|_2^2- 2\sqrt{2t}/\|\bA\|_2 }{1 +2\frac{2\sqrt{q}+\sqrt{2t}}{\sigma_1(\bA)}+ \frac{3\sqrt{pq}+ 4\sqrt{2pt}+ 2t}{\sigma^2_1(\bA)} } \leq  \frac{\sigma^2_1(\bA)}{\|\bA\|_2^2}\widehat{\mathrm{ER}}_{2,\infty}(\bA)\leq
    \frac{1 + [2\sqrt{pqt}+2t]/\|\bA\|_2^2+ 2\sqrt{2t}/\|\bA\|_2 }{1 -2\frac{2\sqrt{q}+\sqrt{2t}}{\sigma_1(\bA)}- \frac{3\sqrt{pq}+ 4\sqrt{2pt}+ 2t}{\sigma^2_1(\bA)} }\ , 
\]
which, assuming that both $\|\bA\|_2$  and  $\|\bA\|_{\infty}$ are  large enough, implies that 
\[
 \frac{\big|\widehat{\mathrm{ER}}_{2,\infty}(\bA)  - \mathrm{ER}_{2,\infty}(\bA)\big|}{\mathrm{ER}_{2,\infty}(\bA)}\lesssim  \frac{\sqrt{pqt}}{\|\bA\|_2^2}+ \frac{\sqrt{pt}+ \sqrt{pq}}{\|\bA\|^2_{\infty}}+ \frac{\sqrt{q}+\sqrt{t}}{\|\bA\|_{\infty}}\enspace .
\]
 
\end{proof}

\begin{proof}[Proof of Proposition~\ref{prp:rank_estimatino_optimalite}]
The first bound is a straightforward consequence of \eqref{eq:condition_rank_estimation}, so we focus on the minimax lower bound. 
 Consider any $r\geq (pq)^{1/4}$. We shall first prove that 
\beq\label{eq:lower_11}
 \inf_{\widehat{R}}\sup_{\bA: \sigma_1(\bA)\geq r }\P\left[\frac{|\widehat{R} -\mathrm{ER}_{2,\infty}(\bA)\big|}{\mathrm{ER}_{2,\infty}(\bA)}\geq c \frac{\sqrt{pq}}{r^2} \right] \geq 0.2\ .  
\eeq
Define the  matrices $\bE_{i,j}$  by $(\bE_{ij})_{kl}=\1_{i=k}\1_{j=l}$. Fix $s \in (0, (pq)^{1/4})$.    We consider the matrix $\bA_0= r \bE_{11}$  and 
the class $\mathcal{C}_{s}$ of matrices of the form $\bA= \bA_0 + s \bB$, where $\bB$ is a rank-one matrix whose only non-zero singular is equal to one and whose  first row and first column are zero. Obviously, the effective rank of $\bA_0$ is $1$, whereas that of $\bA$ in $\mathcal{C}_{s}$ is $1+ s^2/r^2$. This allows us to reducing the problem of estimating the effective rank to that of testing whether $\bA=\bA_0$ versus $\bA\in \mathcal{C}_{s}$. 
\beqn 
\inf_{\widehat{R}}\sup_{\bA: \sigma_1(\bA)\geq r }\P\left[\frac{|\widehat{R} -\mathrm{ER}_{2,\infty}(\bA)\big|}{\mathrm{ER}_{2,\infty}(\bA)}\geq  \frac{s^2}{4r^2}\right] &\geq& \inf_{\widehat{T}\in \{0,1\}}\left[\max \left(\P_{\bA_0}[\widehat{T}=1], \max_{\bA\in \mathcal{C}_s}\P_{\bA}[\widehat{T}=0]\right)\right]\ . 
\eeqn
Since the first row and the first column of $\bY$ are uninformative for this test, this reduces to the problem of testing whether a $(p-1)\times (q-1)$ matrix $\bA$ is null versus $\bA$ is a rank-one matrix with non-zero singular value $s$. We rely again on Le Cam's method of fuzzy hypotheses~\cite{tsybakov_book}. Considering again the same mixture distribution $\mu$ as in the proof of Proposition~\ref{prp:lower_frob}, we deduce that 
\beqn 
\inf_{\widehat{R}}\sup_{\bA: \sigma_1(\bA)\geq r }\P\left[\frac{|\widehat{R} -\mathrm{ER}_{2,\infty}(\bA)\big|}{\mathrm{ER}_{2,\infty}(\bA)}\geq  \frac{s^2}{4r^2}\right] &\geq& \frac{1- \sqrt{\chi^2(\mathbb{P}_0, \mathbf{P})}}{2}. 
\eeqn
By Lemma~\ref{lem:controlL2}, the latter expression is larger than $0.2$ if we take $s= [0.25(p-1)(q-1)]^{1/4}$. We proved~\eqref{eq:lower_11}.

\bigskip 

Let us now prove that 
\beq\label{eq:lower_12}
 \inf_{\widehat{R}}\sup_{\bA: \sigma_1(\bA)\geq r }\P\left[\frac{|\widehat{R} -\mathrm{ER}_{2,\infty}(\bA)\big|}{\mathrm{ER}_{2,\infty}(\bA)}\geq c \frac{\sqrt{q}}{r} \right] \geq 0.2\ .  
\eeq
Without loss of generality, we restrict ourselves to square matrices of dimension $q$. Introduce $\bA_1= r \bI_q$ and, for $s>0$, $\mathcal{C}'_s$ the collection of matrices of the form $\bA= \bA_1 + s uu^T$ where $u$ is a unit vector. The effective rank of $\bA_1$ is $q$, whereas that of matrices $\bA$ in $\mathcal{S}'$ equals $q - (q-1)\frac{s^2 + 2sr}{(r+ s)^2}$. As long as $r\geq 2s$ and $q\geq 2$, the relative difference between these effective ranks is at least $\tfrac{4s}{9r}$. Arguing as above, we deduce that 
\beqn 
\inf_{\widehat{R}}\sup_{\bA: \sigma_1(\bA)\geq r }\P\left[\frac{|\widehat{R} -\mathrm{ER}_{2,\infty}(\bA)\big|}{\mathrm{ER}_{2,\infty}(\bA)}\geq  \frac{2s}{9r}\right] &\geq& \inf_{\widehat{T}\in \{0,1\}}\left[\max \left(\P_{\bA_1}[\widehat{T}=0], \max_{\bA\in \mathcal{C}'_{s}}\P_{\bA}[\widehat{T}=0]\right)\right]\ , 
\eeqn
the latter quantity being equivalent to the optimal error of the test of the hypotheses ``$\bA=0$'' versus ``$\bA= s uu^T$ for some unit vector $u$''.
Introduce the uniform distribution $\nu$ on the hypercube $\{-1,1\}^q$ and define $\mu$ the corresponding distribution of $\bA= s/q v v^T$ when $v$ is sampled according to $\nu$. Writing $\mathbf{P}'$ the marginal distribution of $Y$ when $\bA$ is sampled according to $\mu$, we deduce that 
\beqn 
\inf_{\widehat{R}}\sup_{\bA: \sigma_1(\bA)\geq r }\P\left[\frac{|\widehat{R} -\mathrm{ER}_{2,\infty}(\bA)\big|}{\mathrm{ER}_{2,\infty}(\bA)}\geq  \frac{2s}{9r}\right] &\geq& \frac{1- \sqrt{\chi^2(\mathbb{P}_0, \mathbf{P}')}}{2}. 
\eeqn 
We claim that, for $s=\sqrt{q/16}$, we we have $\chi^2(\mathbb{P}_0, \mathbf{P}')\leq 1/3$. Before showing this claim, let us finish this proof. Since $r\geq (pq)^{1/4}\geq \sqrt{q}$, we have $r\geq 2s$ and \eqref{eq:lower_12} follows from the previous inequality.

It remains to prove the claim. Recall that $\chi^2(\mathbb{P}_0, \mathbf{P}')=\E_{0}[L^2]-1$, where $L$ is the likelihood ratio. 
 Arguing as in the beginning of the proof of Lemma~\ref{lem:controlL2}, we deduce that 
\[
 \E_0[L^2]= \E\left[\exp\left(\frac{s^2Z^2}{q^2}\right)\right]= \E\left[\exp\left(\frac{Z^2}{16q}\right)\right]\ , 
\]
where $Z$ is distributed as a sum of $q$ independent Rademacher random variables. By Hoeffding's inequality $\P[|Z|\geq u]\leq 2e^{-u^2/(2q)}$. Hence, 
\[
 \E_0[L^2]-1\leq \int_{1}^{\infty}\P\left[\exp\left(\frac{Z^2}{16q}\right)\geq t\right]dt\leq 2\int_{1}^{\infty}\frac{1}{t^8}dt= \frac{2}{9}\leq \frac{1}{3} \ . 
\]
The result follows.
\end{proof}

\section{Proofs for Sub-Gaussian noise}

\subsection{Proof of Theorem~\ref{thm:non_gaussian}}

\begin{proof}[Proof of Theorem~\ref{thm:non_gaussian}]

  Recall that our estimator $U_k$ is defined as 
  \[
  U_k = \sum_{i=i_1,\ldots, i_k}\sum_{j=j_1,\ldots, j_k}\prod_{r,s}H_{N_{rs}(ij)}(\bY_{rs})\enspace , 
  \]
  so that its square risk satisfies 
  \beq\label{eq:defi_risk_U_k}
  \E\left[\left(U_k - \|\bA\|_{2k}^{2k}\right)^2\right]= \sum_{i,j,i'j'}\underset{= T_{iji'j'}}{\underbrace{\E\left[\left(\prod_{r,s}H_{N_{rs}(ij)}(\bY_{rs})-  \prod_{rs}\bA_{rs}^{N_{rs}(ij)}\right)\left(\prod_{r,s}H_{N_{rs}(i'j')}(\bY_{rs})-  \prod_{rs}\bA_{rs}^{N_{rs}(i'j')}\right)\right]}} \ . 
  \eeq  
  The random variables $Y_{rs}$ are independent. Besides, we have  $\E[H_1(Y_{rs})]= \bA_{rs}$  and $\E[H_2(Y_{rs})]= \bA_{rs}^2$ since  $\var{E_{rs}}=1$. As a consequence, $T_{iji'j'}$ is equal to zero if 
  \[
    \max_{rs}N_{rs}(ij) N_{rs}(i'j')=0\quad  \text{ and }\quad  [\max_{rs}N_{rs}(ij)]\wedge [\max_{rs}N_{rs}(i'j')]\leq 2\enspace .
  \]
  Indeed, the first condition ensures that $T_{iji'j'}$ is the expectation of a product of two independent random variables $Z_1$ and $Z_2$ whereas the second condition ensures that either $\E[Z_1]=0$ or $\E[Z_2]=0$ so that $T_{iji'j'}=0$. In order to work out the expression of the non-zero $T_{iji'j'}$, we shall rely on the following lemma which control that expectation of $H_{k}(Y)$ and of the products $H_{k}(Y)H_{l}(Y)$ when $Y$ is not normally distributed.

  \begin{lem}\label{lem:hermite_subgaussian}
  Let $Y=\theta + E$ where $E$ is a mean-zero random variable satisfying $\var{E}=1$. There exists a universal constant $c>0$, such that the following holds. For any integer $k\geq 3$, we have
  \[
    \E[H_k(Y)]= \theta^k+ P_{k}(\theta) \ , 
  \]
  where $P_{k}$ is a polynomial of degree at most $k-3$ and whose coefficients are, in absolute value, uniformly bounded by $(c\|E\|_{\psi_2})^k k^{k+1}$. 
  Let $k$ and $l$ be two non-negative integers satisfying  $k+l \geq 2$. Then,  
  \[
    \E[H_kH_l(Y)]= \theta^{k+l}+ Q_{k,l}(\theta)  
  \]
  where the degree of $Q_{k,l}$ is at most $k+l-2$ and whose coefficients are, in absolute value, uniformly bounded by $(c\|E\|_{\psi_2})^{k+l} (k+l)^{k+l+2}$.  Besides, if $k+l=2= k\vee l$, then  $Q_{k,l}=0$. 
  \end{lem}

  Given a partition $\cP=\{\cP_1,\ldots, \cP_l\}$ of $[4k]$ with $l$ groups, we consider the function  $F_{\cP}$ defined by 
  \begin{eqnarray}
  F_{\cP}[x_1,\ldots, x_l]&=& \prod_{r=1}^{l} \mathbb{E}\left[H_{|\cP_r\cap [2k]|}(y_r)H_{|\cP_r\cap [2k+1:4k]|}(y_r)\right] - \prod_{r=1}^{l} x_r^{|\cP_r\cap [2k]|}\E\left[H_{|\cP_r\cap [2k+1:4k]|}(y_r)\right]\nonumber \\  & &- \prod_{r=1}^{l} x_r^{|\cP_r\cap [2k+1:4k]|}\E\left[H_{|\cP_r\cap [2k]|}(y_r)\right]+  \prod_{r=1}^{l} x_r^{|\cP_r|}\ , \nonumber
  \end{eqnarray}
  where, for $r=1,\ldots, l$,  $y_r=x_r+ \epsilon_r$ and the $\epsilon_r$'s are independent and identically distributed.
  By Lemma~\ref{lem:hermite_subgaussian} above, $F_{\cP}$ is a polynomial of $l$ variables.
  \begin{eqnarray}\nonumber
    F_{\cP}[x_1,\ldots, x_l]&=& \prod_{r:\ |\cP_r|=1} x_r \left[\prod_{r:\ |\cP_r|\geq 2}  x_r^{|\cP_r|} + \prod_{r:\ |\cP_r|\geq 2} \left[ x_r^{|\cP_r|} + Q_{|\cP_r\cap [2k]|,|\cP_r\cap [2k+1:4k]|}(x_k)\right]  \right.\\ 
  &&  \quad \quad - \prod_{r:\ |\cP_r|\geq 2}\left[  x_r^{|\cP_r|}+ x_r^{|\cP_r\cap [2k]|}P_{|\cP_r\cap [2k+1:4k]|}(x_r) \right] \nonumber \\ & &  \quad \quad  \left. -   \prod_{r:\ |\cP_r|\geq 2}\left[  x_r^{|\cP_r|}+ x_r^{|\cP_r\cap [2k+1:4k]|}P_{|\cP_r\cap [2k]|}(x_r) \right]  \right]\ .\label{eq:definition_F_P}
    \end{eqnarray}

  \medskip

  Given $m=(i,j,i',j')$, we define the following sequence $\underline{m}$  of $2$-tuples of size $4k$
  \begin{eqnarray}\label{eq:definition_u_m} 
  \underline{m}&= &((i_1,j_1), (i_1,j_2),\ldots, (i_{k},j_1), (i'_1,j'_1),\ldots, (i'_k,j'_1)) \ .
  \end{eqnarray}
  Then, we define $\cP[m]$ as the partition of $[4k]$ that groups identical 2-tuples in $\underline{m}$. 
  Besides, we consider the two sequences $\underline{m}^{ev}$ and $\underline{m}^{od}$ defined by
  \begin{eqnarray}\label{eq:definition_u_m_ev} 
  \underline{m}^{ev}= (i_1,i_1,i_2,i_2,\ldots, i_k,i_k, i'_1,i'_1,\ldots,i'_k,i'_k)\ ;\,\,\, 
  \underline{m}^{od}= (j_1,j_2,j_2,\ldots, j_k,j_1, j'_1,j'_2,\ldots,j'_k,j'_1)\ .\nonumber
  \end{eqnarray}
  Given a partition $\cP$ of $[4k]$ into $l$ groups, we pick representatives $s_1, \ldots, s_l$ for each group. Then, 
  for all $m$ satisfying $\cP[m]=\cP$, we have 
  \[
   T_m= F_{\cP}[\bA_{\underline{m}_{s_1}}, \ldots, \bA_{\underline{m}_{s_l}}]\ .
  \]
   For short, we write henceforth $F_{\cP}[m]$ for $F_{\cP}[\bA_{\underline{m}_{s_1}}, \ldots, \bA_{\underline{m}_{s_l}}]$. 
  Write $\boldsymbol{\cP}_k$ for the collection of all  partitions $\cP$ of $[4k]$. Equipped with this notation, we group the expressions $T_m$ according to their partition $\cP[m]$. 
  \begin{equation}\label{eq:upper_0_risk_GP}
  \E\left[\left(U_k - \|\bA\|_{2k}^{2k}\right)^2\right]= \sum_{\cP\in \boldsymbol{\cP}_k} \quad\sum_{m: \ \cP[m]=\cP }F_{\cP}(m)\ .
  \end{equation}
Unfortunately, it is not simple to enumerate the sequences such that $\cP[m]=\cP$ and these sums $\sum_{m: \ \cP[m]=\cP }F_{\cP}(m)$ do not simplify much. For this reason, we prefer to reformulate the expressions using sums over sequences $m$ such that $\cP[m]$ is coarser than $\cP$ 
  For two partitions $\cP'$ and $\cP$, we write $ \cP' \trianglelefteq \cP$ if $\cP'$ is finer or equal to $\cP$. First, we extend the definition of $F_{\cP}(m)$ to any sequence $m$ whose partition $\cP[m]$ is coarser than $\cP$, that is $\cP\trianglelefteq \cP[m]$. Then, we  define a new functional $G_{\cP}$ for $\cP\in \boldsymbol{\cP}_k$ by recursion. 
  If $\cP$ is the finest partition (with $4k$ groups), we fix $G_{\cP}=F_{\cP}=0$. Otherwise, we define
  \[
   G_{\cP}(m) = F_{\cP}(m)- \sum_{\cP'\in \boldsymbol{\cP}_k: \ \cP'\triangleleft \cP}G_{\cP'}(m)\ , 
  \]
  where $F_{\cP}(m)$ is defined in~\eqref{eq:definition_F_P}.  
  One then easily deduce from this definition that  
  \beq\label{eq:upper_risk_GP}
  \E\left[\left(U_k - \|\bA\|_{2k}^{2k}\right)^2\right]= \sum_{\cP\in \boldsymbol{\cP}_k}W_{\cP}\quad \text{ where }\quad  W_{\cP}=\sum_{m: \ \cP \trianglelefteq \cP[m]}G_{\cP}(m)\ .
  \eeq
  In the remainder of the proof, we shall carefully control $W_{\cP}$. To do this, we need to evaluate the functional $G_{\cP}$, which is done using 
  M\"obius inversion formula for posets (see the proof of the next lemma). Equipped with this formula, we first establish some structural properties of $G_{\cP}$. 
  
  \begin{lem}\label{lem:decomposition_G_P}
  $G_{\cP}$ can be represented as a polynomial in $|\cP|$ variables that satisfies the two following properties:
  \begin{enumerate}
  \item[(a)] if $\cP$ contains at least one group, say $\cP_1$, of size $2$ such that either $\cP_1\subset [2k]$ or $\cP_1\subset [2k+1:4k]$, then $G_{\cP}=0$. 
  \item[(b)] For any $l=1,\ldots, |\cP|$ such that $|\cP_l|\geq 2$, the degree of $\cP$ in $x_l$ is at most $|\cP_l|-3$ if $\cP_l\subset[2k]$ or  $\cP_l\subset[2k+1:4k]$  and at most $|\cP_l|-2$ otherwise. 
  \end{enumerate}

  \end{lem}
  
  Writing  $\kappa$ the number of groups of $\cP$ of size higher than one, we order the groups of the partition $\cP$ in such a way that the $\kappa$ first groups are of size at least $2$. 
  We define  
  the collection $\mathrm{deg}(\cP)$ of non-negative integer-valued vectors $D$ of size $\kappa$ satisfying the two following properties  for all $1\leq r\leq \kappa$, 
  \begin{itemize}
    \item[(a)] $D_r\leq |\cP_r|-2$ if $|\cP_r|\geq 2$ ;
    \item[(b)] $D_r\leq |\cP_r|-3$ if $|\cP_r|\geq 3$ and $\cP_r\subset [2k]$ or  $\cP_r\subset  [2k+1:4k]$ .
  \end{itemize}
  Define the monomial,  $G_{\cP,D}[x_1,\ldots, x_{|\cP|}]= \prod_{l=\kappa+1}^{|\cP|}x_l \prod_{l=1}^{\kappa}x_l^{D_l}$ for any $D\in \mathrm{deg}(\cP)$. By Lemma~\ref{lem:decomposition_G_P}, we deduce that there exists a collection of numbers $\alpha_{D,\cP}$ for $D\in \mathrm{deg}(\cP)$ such that 
  \[
  G_{\cP}[x_1,\ldots, x_{|\cP|}]= \sum_{D\in \mathrm{deg}(\cP)}\alpha_{D,\cP} G_{\cP,D}[x_1,\ldots, x_{|\cP|}] \ .
  \]
  The next lemma provides an uniform bound for the coefficients $\alpha_{D,\cP}$. It is based on M\"obius inversion formula (as the previous lemma) together with Lemma~\ref{lem:hermite_subgaussian}. 
  \begin{lem}\label{lem:G_P}
  For any partition $\cP$, we have 
  \[
  \sum_{D\in \mathrm{deg}(\cP)}|\alpha_{D,\cP} |\leq (c'\|E\|_{\psi_2})^{4k}k^{12k+2}
  \]
  where $c'$ is a numerical constant. 
  \end{lem}

  Now define the associated sums 
  \beq\label{eq:definition__g_m_cp_Ps_s}
  W_{\cP,D}=  \sum_{m: \ \cP \trianglelefteq \cP[m]}  G_{\cP,D}[(\bA_{\underline{m}_{s_1}},\ldots, \bA_{\underline{m}_{s_l} })]  \ , 
  \eeq
  where we recall that $(s_1,\ldots, s_l)$ are representatives of the partition $\cP$.
  In light of the last lemma, we deduce from~\eqref{eq:upper_risk_GP}
  \begin{eqnarray}
    \E\left[\left(U_k - \|\bA\|_{2k}^{2k}\right)^2\right]&\leq &  \sum_{\cP}\sum_{D\in \mathrm{Deg}(\cP)}|\alpha_{D,\cP}||W_{\cP,D}| \nonumber \\
    &\leq & (c''\|E\|_{\psi_2})^{4k}k^{16k+2}  \max_{\cP,\ D}|W_{\cP,D}|  \enspace . \label{eq:upper_max}
  \end{eqnarray}
  Hence, it remains to upper bound $W_{\cP,D}$ defined in~\eqref{eq:definition__g_m_cp_Ps_s} to conclude. 
  
  Consider any partition $\cP$ and any $D\in \mathrm{deg}(P)$ such that $W_{\cP,D}\neq 0$. In particular,  $\cP$ contains a least a group of size larger than one and $\cP$ does not contain any group of size $2$ which is either included in $[2k]$ or in $[2k+1,4k]$ --see Lemma~\ref{lem:decomposition_G_P}. 
  
  Define  $\cP_{ev}$  as the finest partition of $[4k]$ that is coarser than $\cP$ and coarser than $\{\{1,2\},\{3,4\}, \ldots, \{4k-1,4k\}\}$. 
  Similarly, we define  $\cP_{od}$ the as the finest partition of $[4k]$ that is coarser than $\cP$ and coarser than $\{\{2k,1\},\{2,3\}, \ldots, \{2k-2,2k-1\}, \{4k,2k+1\},\{2k+2,2k+1\},\ldots,  \{4k-2,4k-1\}\}$. Denote $l_1$ the number of groups of $\cP_{ev}$ and $l_2$ the number of groups of $\cP_{od}$. Equipped with this notation, we deduce that there are $p^{l_1}q^{l_2}$ sequences $m=(i,j,i',j')$ satisfying $\cP \trianglelefteq \cP[m]$.

  We denote $\cP_{\geq 2}$ the subcollection of $\cP$ made of groups of size at least $2$. By definition of $\kappa$ above, we have $\kappa=|\cP_{\geq 2}|$.  
  Finally, we define $\cP_{\geq 2,ev}\subset \cP_{ev}$ the subcollection of $\cP_{ev}$ made of groups that intersect at least one group of $\cP_{\geq 2}$. We denote $\kappa_1\leq \kappa$ the number of groups of $\cP_{\geq 2,ev}$. Analogously, we define $\cP_{\geq 2,od}$ and we denote $\kappa_2\leq \kappa$ the corresponding number of groups. In order to evaluate $W_{\cP,D}$, we first fix the values of $\underline{m}_{ev}$, $\underline{m}_{od}$ on $\cP_{\geq 2,ev}$ and $\cP_{\geq 2,od}$ respectively --recall the definitions of $\underline{m}_{ev}$ and $\underline{m}_{od}$ in \eqref{eq:definition_u_m_ev}.  
  Consider any  $s=(s_1,\ldots, s_{\kappa_1})\in [p]^{\kappa_1}$ and $t=(t_1,\ldots, t_{\kappa_2})\in [q]^{\kappa_2}$, where $s_r$ (resp. $t_r$) corresponds to the value of $(i_{ev},i'_{ev})$ (resp. $(j_{od},j'_{od})$) in the $r$-th group of $\cP_{\geq 2,ev}$ (resp. $\cP_{\geq 2,od})$.  Besides denote $\lambda_1$ (resp. $\lambda_2$) the number of groups of $\cP_{1,ev}=\cP_{ev}\setminus \cP_{\geq 2,ev}$ (resp. $\cP_{1,od}=\cP_{od}\setminus \cP_{\geq 2,od}$).
  
  By Definition~\eqref{eq:definition__g_m_cp_Ps_s} of $W_{\cP; D}$, there exist sequences $(h_r)_{r=1,\ldots,\kappa}$ in $[\kappa_1]^{\kappa}$ and  $(h'_r)_{r=1,\ldots,\kappa}$ in $[\kappa_2]^{\kappa}$,  $(h_r)_{r=\kappa+1,\ldots,l}$ in $[\kappa_1+\lambda_1]^{l-\kappa}$  and $(h'_{r})_{r=\kappa+ 1,\ldots, \lambda}$ in $[\kappa_2+\lambda_2]^{l-\kappa}$ such that  
  \beq\label{eq:first_expression_W_P_d}
   W_{\cP; D}= \sum_{s_1,\ldots, s_{\kappa_1}=1}^{p} \sum_{t_1,\ldots, t_{\kappa_2}=1}^{q}\left[\prod_{r=1}^{\kappa} {\bA}_{s_{h_r},t_{h'_r}}^{D_r}\right] \sum_{s_{\kappa_1+1}, \ldots, s_{\kappa_1+\lambda_1}=1}^{p} \quad \sum_{t_{\kappa_2+1}, \ldots, t_{\kappa_2+\lambda_2}=1}^{q}\, \prod_{r=\kappa+1}^{l}  \bA_{s_{h_r},t_{h'_r}} \enspace .
  \eeq
  Besides, the sequence $(h_r)_{r=\kappa+1,\ldots,l}$ contains exactly twice the value $u$ for all $u\in  [\kappa_1+1, \kappa_1+\lambda_1]$. Similarly, the sequence  $(h'_r)_{r=\kappa+1,\ldots,l}$ contains exactly twice the value $u$ for all $u \in [\kappa_2+1, \kappa_2+\lambda_2]$. The sequence $(h_r)_{r=1,\ldots, \kappa}$ contains at least once each value in $[\kappa_1]$  and the sequence $(h'_r)_{r=1,\ldots, \kappa}$ contains at least once each value in $[\kappa_2]$. Since the indices $s_{\kappa_1+1},\ldots,s_{\kappa_1+\lambda_1}$ and $t_{\kappa_2+1}, \ldots, t_{\kappa_2+\lambda_2}$ occur exactly twice in the sum~\eqref{eq:first_expression_W_P_d}, we can express the corresponding sums in~\eqref{eq:first_expression_W_P_d} as products of matrices $\bA$ and $\bA^T$. This leads us to the following expression for $W_{\cP; D}$ (the proof is omitted).
   
  \begin{lem}\label{lem:decompo_W_CP,d}
There exists a positive integer $z_{\cP}$ only depending on the partition $\cP$ such that 
  \beq\label{eq:upper_W_P_d}
  W_{\cP; D} =  \sum_{s_1,\ldots, s_{\kappa_1}=1}^{p} \sum_{t_1,\ldots, t_{\kappa_2}=1}^{q}\left[\prod_{r=1}^{\kappa} {\bA}_{s_{h_r},t_{h'_r}}^{D_r}\right]\prod_{r=1}^{z_{\cP}} \bB^{(r)}_{f_{r},f'_{r}}
  \eeq
  where $\bB^{(r)}=\bA^{\omega_r}(\bA^{T}\bA)^{\omega'_r}(\bA^{T})^{\omega''_r}$ with $\omega_r$ and $\omega''_r$ in $\{0,1\}$ and $\omega'_r$ is a non-negative integer. If $\omega_r= 1$, we have $f_r= s_{v_r}$ for some $1\leq v_r\leq \kappa_1$, otherwise we have $f_r= t_{v_r}$ for some $1\leq v_r\leq \kappa_2$. If $\omega''_r= 1$, we have $h'_r= s_{v'_r}$ for some $1\leq v'_r\leq l$, otherwise we have $h_r= t_{v'_r}$ for some $1\leq v_r\leq l'$.
  
  \end{lem}
  
Now that we have expressed $W_{\cP; D}$ as a sum indexed by the group in $\cP_{\geq 2,ev}$ and $\cP_{\geq 2,od}$ in \eqref{eq:upper_W_P_d}, we cannot simply express $W_{\cP; D}$ as the trace of a matrix or even the product matrix traces. Indeed, the indices $s_{i}$ or $t_{j}$ may occur more than two times in~\eqref{eq:upper_W_P_d}. The next bound is the crux of the proof. 
  \begin{prp}\label{lem:upper_W_cp,p}
   Upon denoting $a = \sum_{r=1}^{\kappa}(|\cP_r|-D_r)$, we have 
   \[
    \big|W_{\cP; D}\big|\leq\left\{
  \begin{array}{cc}
    \|\bA\|_{\infty}^{4k-a} p^{\lfloor a/4\rfloor}q^{\lceil a/4\rceil} & \text{ if $a$ is even}\\
    \|\bA\|_{\infty}^{4k-a} p^{ (a-1)/4}q^{(a-1)/4} & \text{ if $a$ is odd}
  \end{array}
    \right.
   \]

  \end{prp}
  We first show how to conclude the proof from this bound before establishing Proposition~\ref{lem:upper_W_cp,p}. 
  Recall that $a = \sum_{r=1}^{\kappa}(|\cP_r|-D_r)$ lies in $[2,4k]$. If $a\equiv 3[4]$, then Young's inequality yields 
  \begin{eqnarray*}
    \|\bA\|_{\infty}^{4k-a} p^{(a-1)/4}q^{(a-1)/4}&\leq& \frac{1}{2}\|\bA\|_{\infty}^{4k-a-1}p^{ (a-3)/4}q^{ (a-3)/4}\left[\|\bA\|_{\infty}^{2} +  qp  \right]\\
    &\leq & \frac{1}{2}\|\bA\|_{\infty}^{4k-a-1}p^{\lfloor (a+1)/4\rfloor}q^{\lceil (a+1)/4\rceil}+ \frac{1}{2}\|\bA\|_{\infty}^{4k-a+1}p^{\lfloor (a-1)/4\rfloor}q^{\lceil (a-1)/4\rceil} \ . 
   \end{eqnarray*}
   If $a\equiv 1[4]$, then Young's inequality yields 
  \begin{eqnarray*}
    \|\bA\|_{\infty}^{4k-a} p^{(a-1)/4}q^{(a-1)/4}&\leq& \frac{1}{2}\|\bA\|_{\infty}^{4k-a-1}p^{ (a-1)/4}q^{ (a-1)/4}\left[\|\bA\|_{\infty}^{2} +  1 \right]\\
    &\leq & \frac{1}{2}\|\bA\|_{\infty}^{4k-a-1}p^{\lfloor (a+1)/4\rfloor}q^{\lceil (a+1)/4\rceil}+ \frac{1}{2}\|\bA\|_{\infty}^{4k-a+1}p^{\lfloor (a-1)/4\rfloor}q^{\lceil (a-1)/4\rceil} \ . 
   \end{eqnarray*}
  This allows us to derive 
  \[
    \big| W_{\cP; D}|\leq \max_{s= 1}^{2k}  \|\bA\|_{\infty}^{4k-2s}p^{\lfloor s/2\rfloor}q^{\lceil s/2\rceil}\leq \max\left[\|\bA\|_{\infty}^{4k-2}q,\|\bA\|_{\infty}^{4k-4}pq,(pq)^{k} \right]
  \]
  Then, together with~\eqref{eq:upper_max}, we conclude that
  \[
    \E\left[\left(U_k - \|\bA\|_{2k}^{2k}\right)^2\right]\leq  (c''\|E\|_{\psi_2})^{4k}k^{16k+2} \max\left[\|\bA\|_{\infty}^{4k-2}q,\|\bA\|_{\infty}^{4k-4}pq,(pq)^{k} \right]\ .
  \]
  The result follows.

  \end{proof}

  \subsection{Proof of Proposition~\ref{lem:upper_W_cp,p}}

  First, we factorize the expression  the operator norm of $\bA$ in order to consider matrices with operator norm at most one. 
  \beq\label{eq:W_cp_d}
   W_{\cP; D}= {\|\bA\|_{\infty}^{4k-\sum_{r=1}^{\kappa}(|\cP_r|-D_r)}}\sum_{s_1,\ldots, s_{\kappa_1}=1}^{p} \sum_{t_1,\ldots, t_{\kappa_2}=1}^{q}\prod_{r=1}^{\kappa} \overline{\bA}_{s_{h_r},t_{h'_r}}^{D_r} \sum_{r=1}^{z_{\cP}} \overline{\bB}^{(r)}_{f_{r},f'_{r}}\enspace , 
  \eeq
where $\overline{\bA}= \bA/\|\bA\|_{\infty}$ and $\overline{\bB}^{(r)}= \bB^{(r)}/ \|\bA\|_{\infty}^{\omega_r+\omega'_r+\omega''_r}$. In order to bound the latter expression, we shall introduce a multigraph $G$ and bound $W_{\cP,D}$ relying on the topology of $G$.

  Before explaining how we translate $(\cP,d)$ into a multigraph, we introduce some notation for a multigraph $G=(V,E)$. Since there are possibly multiple edges between the same two vertices, the edges $e\in E$ are defined as a $3$-tuple $(i,j,a)$ where $i$ and $j$ are two  vertices and $a$ is a positive integer corresponding to the index of $e$ among all the edges between $i$ and $j$. Besides, we respectively denote $e_1=i$ and $e_2=j$ the vertices incident to $e$. Although the edges are directed, we shall most often ignore their direction and the degree of vertex $i$ henceforth corresponds to the number of edges that are incident to $i$. Finally, each vertex of $G$ is either labelled as type '$r$' (for row) or as type '$c$' (for column). We write $R\subset V$ for the collection of vertices of type '$r$' and $C=V\setminus R$ for collection of vertices of type '$c$'.  We write  $E_{\mathrm{lo}}\subset E$ for the collection of self-loops, that is edges satisfying $e_1=e_2$ and $E_0= E\setminus E_{\mathrm{lo}}$ for the collection of remaining edges. 
  In the sequel, the external degree $d_i$ of a vertex $i\in V$ is the number of edges $e\in E_0$ that are incident to $i$. We also $d_{i,\mathrm{lo}}$ for the number of self-loops of $i$.

  A matrix $\bC$ is said to satisfy Condition $(\cB_0)$ if its operator norm is at most one and if its rank is at most $q$. A real vector $b$ is said to satisfy Condition~$(\cB_{\mathrm{lo}})$ if $|b|_{\infty}\leq 1$ and $|b|_1\leq q$. Finally, a real vector $b$ is said to satisfy the slightly weaker condition  $(\cB^{-}_{\mathrm{lo}})$ if $|b|_{\infty}\leq 1$ and $|b|_2\leq \sqrt{q}$

  For a multigraph $G=(V,E)$, we consider a collection $\bB= (\bB^{(e)})_{e\in E_0}$ of matrices and a collection $\beta= (\beta^{(e)})_{e\in E_{\mathrm{lo}}}$ of vectors. The matrix $\bB^{(e)}$ has $p$ (resp. $q$) rows  if $e_1\in R$ (resp. $e_1\in C$).  Similarly, $\bB^{(e)}$ has $p$ (resp. $q$) columns  if $e_2\in R$ (resp. $e_2\in C$). The size of $\beta^{(e)}$ is $p$ (resp. $q$) if $e_1=e_2\in R$ (resp. $\in C$). Finally, $(G,\bB,\beta)$ is said to satisfy Property $(\cA_0)$ if:
  \begin{itemize}
  \item[(i)] all matrices $\bB^{(e)}$ with $e\in E_0$ satisfy Condition $(\cB_0)$.
  \item[(ii)] all vectors $\beta^{(e)}$ with $e\in E_{\mathrm{lo}}$ satisfy Condition $(\cB^{-}_{\mathrm{lo}})$.
  \item[(iii)] all vectors $\beta^{(e)}$ with $e\in E_{\mathrm{lo}}$ where the self-loop $e$ is incident to a vertex $i$ such that $d_i\leq 1$ satisfy Condition $(\cB_{\mathrm{lo}})$.   
  \end{itemize}

  \medskip 
  
  Given a multigraph $G$, a collection $\bB$ of matrices and a collection $\beta$ of vectors, we define the quantity $W_{G,\bB,\beta}$ by 
  \begin{equation}\label{eq:definition_W}
    W_{G,\bB,\beta} := \sum_{(\omega_i)_{i\in V}}\prod_{e\in E_0} \bB^{(e)}_{\omega_{e_1},\omega_{e_2}}\prod_{e\in E_{\mathrm{lo}}}\beta^{(e)}_{\omega_{e_1}}\ , 
  \end{equation}
  where the sum runs over $\omega_i$ runs from $1$ to $p$ if $i\in R$ and from $1$ to $q$ is $i\in C$. 
  
  \medskip 
  
  Denote $V_{r,0,0}$ the subset of isolated vertices $i$ in $G$ that have a label $r$ and satisfy $d_i=d_{i,\mathrm{lo}}=0$. Similarly, 
  $V_{*,0,*}$ stands for the collection of vertices $i$ satisfying $d_i=0$, whereas  $V_{r,1,*}$ is the collection of vertices of type '$r$' that satisfy $d_i=1$. Finally, 
  $V_{r,\geq 2,*}$ is the collection of vertices of type '$r$' with $d_i\geq 2$. 
  

  \begin{prp}\label{prop:graph_majoration}
  If $(G,\bB,\beta)$ satisfies  Property $(\cA_0)$, then 
    \[
    \big|W_{G,\bB,\beta} \big|\leq  p^{|V_{r,0,0}|+|V_{r,1,0}|/2}q^{ |V_{*,0,*}|- |V_{r,0,0}|+ [|V|- |V_{r,1,0}|/2-|V_{*,0,*}|  ]/2]}\left(\frac{p}{q}\right)^{(|V_{r,\geq 2,*}|-1)_+/2} \enspace . 
    \]
  
  \end{prp}

  Let us now apply explicit how we can apply this result to the expression~\eqref{eq:W_cp_d} of $W_{\cP; D}$. We build a multigraph with $\kappa_1$ vertices of type 'r' and $\kappa_2$ vertices of type 'c'. For $r=1,\ldots, \kappa$, we add $D_r$ edges between the vertex $h_r$ of type 'r' and the vertex $h'_r$ of type 'c' and we associate the matrix  $\overline{\bA}$ to this edge. Obviously, we have $\|\overline{\bA}\|_{\infty}\leq 1$ and $\mathrm{Rank}(\bA)\leq q$ so that $\overline{\bA}$  satisfies $\mathcal{B}_0$. For $r=1,\ldots, z_{\cP}$, we also add an edge between the vertices $v_r$ and $v'_r$ where $v_r$ and $v'_r$ are defined in Lemma~\ref{lem:decompo_W_CP,d}. If $v_r\neq v'_r$, we associate the matrix $\overline{\bB}^{(r)}$ which satisfies also $\cB_0$. If $v_r=v'_r$, then we associate the vector $(\bB^{(r)}_{a,a})$ to this self-loop. Since $\|\bB^{(r)}\|_{\infty}\leq 1$ and its rank is at most $q$, one readily checks that it satisfies $\cB_{\mathrm{lo}}$. We are therefore in position to apply Proposition~\ref{prop:graph_majoration} to $W_{\cP,D}$.
  \[
    |W_{\cP,D}|\leq \|\bA\|_{\infty}^{4k- a}p^{|V_{r,0,0}|+|V_{r,1,0}|/2}q^{ |V_{*,0,*}|- |V_{r,0,0}|+ [|V|- |V_{r,1,1}|-|V_{*,0,*}|  ]/2]}\left(\frac{p}{q}\right)^{(|V_{r,\geq 2,*}|-1)_+/2} \enspace ,
  \]
  where we recall that $a= \sum_{r=1}^{\kappa}(|\cP_r|-D_r)$. For $r=1,\ldots, \kappa$, define $a_r = |\cP_r|-D_r$ so that $a= \sum_{r=1}^{\kappa} a_r$. From Lemma~\ref{lem:decomposition_G_P}, we know that $a_r\geq 2$ and that $a_r\geq 3$ if $|\cP_r|\subset[2r]$ or $|\cP_r|\subset[2r+1:4r]$.
  
  Let us consider any vertex $i$ of the $\kappa_2$ vertices of type '$c$'. The vertex $i$ corresponds to a group $\cP_{i,od}$ of $\cP_{\geq 2,od}$ and we write $s_i$ for the number of groups of $\cP_{\geq 2}$ whose intersection with this group of  $\cP_{\geq 2,ev}$ is non-empty. By definition of $\cP_{\geq 2,ev}$, we have $s_i\geq 1$ and $\sum_{i=1}^{\kappa_2} s_i = \kappa$. If $s_i=1$, we denote $r_i$ for the unique group $\cP_{r_i}$ which is included in $\cP_{i,ev}$.

\begin{lem}\label{lem:d_i}
  Any node $i$ of type '$c$' such that $s_i=1$ and $a_{r_i}\leq 3$ has an external degree $d_i$ satisfying $d_i\geq 1$. 
\end{lem}

  Similarly, we define for any vertex $i=1,\ldots, \kappa_1$ of type '$r$' the group $\cP_{i,ev}$, the number $s'_i$ of groups of $\cP_{\geq 2}$ that are included in $\cP_{i,ev}$ and, if $s'_i=1$, the index $r'_i$ such that $\cP_{r'_i}\subset \cP_{i,ev}$.

\begin{lem}\label{lem:d_i_2}
  Any node $i$ of type '$r$' such that $s'_i=1$ and $a_{r'_i}\leq 3$ has an external degree $d_i$ satisfying $d_i\geq 1$. Besides,  if $s'_i=1$ and $a_{r'_i}=2$, then $d_i\geq 2$. 
\end{lem}

  Since $p\geq q$, we deduce that 
  \[
    |W_{\cP,D}|\leq \|\bA\|_{\infty}^{4k- a}p^{T_1 -T_3}q^{T_2+T_3}\ , 
    \]
  where 
  \[T_1= \sum_{i=1}^{\kappa_1}\left[\1_{s'_i>1}+ \1_{s'_i=1}\1_{a_{r'_i}\geq 4} + \frac{1}{2} \1_{s'_i=1}\1_{a_{r'_i}\in [2,3]}\right] \enspace ;
  \]
  \[
    T_2= \sum_{i=1}^{\kappa_2}\left[\1_{s_i>1}+ \1_{s_i=1}\1_{a_{r_i}\geq 4} + \frac{1}{2} \1_{s_i=1}\1_{a_{r_i}\in [2,3]}\right]\ ,
  \]
  and $T_3= \tfrac{1}{2} \1_{\sum_{i=1}^{\kappa_1}\1_{s'_i=1}\1_{a_{r'_i}=2}>0}$. Since $\sum_{r=1}^\kappa a_r =a$ and $a_r\geq 2$, we derive that 
  \[
  T_2\leq \frac{1}{4} a - \sum_{i=1}^{\kappa_2}\frac{1}{2}(s_i-2)_+ - \frac{1}{4}\sum_{r=1}^{\kappa}(a_r-2)_+ +\frac{1}{2}\sum_{i=1}^{\kappa_2}\1_{s_i=1}\1_{a_{r_i}\geq 4}
  \]
  and similarly that 
  \beq\label{eq:upper_T_1}
  T_1 \leq \frac{1}{4} a - \sum_{i=1}^{\kappa_1}\frac{1}{2}(s'_i-2)_+ - \frac{1}{4}\sum_{r=1}^{\kappa}(a_r-2)_+ +\frac{1}{2}\sum_{i=1}^{\kappa_1}\1_{s'_i=1}\1_{a_{r'_i}\geq 4}\ .
  \eeq
  To finish, we consider four different cases depending on the value of $a$. If $a\equiv 0[4]$, then one readily checks that $T_1\leq a/4$ and $T_2\leq a/4$ which leads to 
  $|W_{\cP,D}|\leq \|\bA\|_{\infty}^{4k- a}p^{a/4}q^{a/4}$.  If $a\equiv 1[2]$, then at least one $a_{r}$ is odd and we deduce that $T_1\leq (a-1)/4$ and $T_2\leq (a-1)/4$ so that 
  $|W_{\cP,D}|\leq \|\bA\|_{\infty}^{4k- a}p^{(a-1)/4}q^{(a-1)/4}$. It remains to consider the case $a\equiv 2[4]$. We observe again that $T_2\leq a/4$ so that $T_2+T_3\leq \lceil a/4\rceil$. If $T_3=1/2$, we can use the simple bound $T_1\leq a/4$ so that $T_1-T_3\leq \lfloor a/4\rfloor$. If $T_3=0$, this implies that no index $i$ satisfies  $s'_i=1$ and $a_{r'_i}=2$, and we claim that $T_1\leq \lfloor a/4\rfloor = (a-2)/4$ which implies that $|W_{\cP,D}|\leq \|\bA\|_{\infty}^{4k- a}p^{\lfloor a/4\rfloor}q^{\lceil a/4\rceil}$.

  It remains to establish the claim. 
  Since we know that $T_1\leq a/4$, and that 4 times the rhs of~\eqref{eq:upper_T_1} is an integer, it suffices to prove that the rhs of ~\eqref{eq:upper_T_1} is neither equal to $a/4$ nor to $(a-1)/4$. First, this expression cannot be equal to $(a-1)/4$ because this would imply that 
  \[
    2\sum_{i=1}^{\kappa_1}\frac{1}{2}(s'_i-2)_+ -\sum_{r=1}^{\kappa}(a_r-2)_+  + 2\sum_{i=1}^{\kappa_1}\1_{s'_i=1}\1_{a_{r'_i}\geq 4}= 1\ ,
  \]
 which in turn would imply that $\sum_{r=1}^{\kappa}(a_r-2)_+ $ is odd  which contradicts that the fact that $a=\sum_{r=1}^{\kappa} a_r$ is even (recall that $a_r\geq 2$). Next, the only way of having the rhs of~\eqref{eq:upper_T_1} equal to $a/4$ is that (i)  all $s'_i$ are either equal to $1$ or $2$, (ii) that, for all $i$ such that $s'_i=1$, we have $a_{r'_i}=4$, and (iii) that all the remaining values of $a_{r}$ are equal to $2$.  
  This implies that $a\equiv 0[4]$ which contradicts the hypothesis $a\equiv 2[4]$.

\begin{proof}[Proof of Lemma~\ref{lem:d_i}]
  In order to have $d_i=0$, we need that $a_{r_i}=|\cP_{r_i}|$, otherwise the graph contains at least $D_{r_i}= |\cP_{r_i}|-a_{r_i}$ external edges incident to $i$ (see \eqref{eq:W_cp_d}). If $a_{r_i}=|\cP_{r_i}|=2$, then this implies that $|\cP_{r_i}\cap [2k]|=1$ and $|\cP_{r_i}\cap [2k+1:4k]|=1$, otherwise $W_{\cP,d}=0$ (Lemma~\ref{lem:decomposition_G_P}). Hence, we have  $|\cP_{i,od}|=4$ with $|\cP_{i,od}\cap[2k]|=2$ and $|\cP_{i,od}\cap[2k+1;4k]|=2$. From the construction of the graph $G$, we deduce that $d_i=2$ (see \eqref{eq:first_expression_W_P_d}).
  
  Now assume that  $a_{r_i}=|\cP_{r_i}|=3$.Since $|\cP_{i,od}|$ is an even integer, and since the number of half edges incident to $i$ in the graph is $|\cP_{i,od}|- a_{r_i}$, it follows that the number of half edges is odd and at least one of these half edges is connected to a distinct vertex. Hence, we have $d_i\geq 1$. 
  
\end{proof}
  
\begin{proof}[Proof of Lemma~\ref{lem:d_i_2}]
  The first claim is proved analogously to Lemma~\ref{lem:d_i}. We focus on the second claim. 
Also, we have already shown in the previous proof that $d_i=2$ if $|a_{r'_i}|=|\cP_{r'_i}|=2$. If $|\cP_{r'_i}|\geq 4$ and $a_{r'_i}=2$, then there are $D_{r_i}= |\cP_{r_i}|-a_{r_i}\geq 2$ external edges incident to $i$ (see \eqref{eq:W_cp_d}). It remains to consider the case where  $|\cP_{r'_i}|=3$ and $a_{r'_i}=2$. Then, $D_{r_i}=|\cP_{r_i}|-a_{r_i}=1$ and there is one external edge which is associated to $D_{r_i}$ (see \eqref{eq:W_cp_d}). Besides, $|\cP_{i,ev}|-|\cP_{r'_i}|$ is odd and is the number of half-edges corresponding to the right-hand side term in (see \eqref{eq:W_cp_d}). Since this number of half-edges is odd, at least one them corresponds to an external edge. Hence, we have proved that $d_i\geq 2$. 
\end{proof}

  \subsection{Proof of Proposition~\ref{prop:graph_majoration}}

The proof is based on an induction-like approach. We shall first iteratively prune the graph $G$ by removing small external degree vertices and we will show that the functional $\big|W_{G,\bB,\beta}\big|$ is not much larger than that of the pruned graph. We start with the following lemma that summarizes the pruning operations that we are able to consider.

  \begin{lem}\label{lem:decompo_graph_1}
  Let $(G,\bB,\beta)$ satisfying $(\cA_0)$ and let $i\in V$ be a vertex  satisfying $d_i\leq 2$. If $d_i\leq 1$, we define $G^{(-i)}$ as the subgraph induced by $V\setminus\{i\}$. If  $d_i=2$ and $i$ has two distinct neighbors, then  $G^{(-i)}$ is defined as the subgraph induced by $V\setminus\{i\}$ where we add an edge $e$ connecting the two neighbors of $i$. If $d_i=2$ and $i$ is connected twice to the same neighbor, say $j$,  then  $G^{(-i)}$  is defined as the subgraph induced by $V\setminus\{i\}$ such that $j$ is incident to a single self-loop. 
  Then, the following holds:
  \begin{enumerate}
    \item[(a)] If $d_i=0$, then, for some $\bB'$, $\beta'$ such that ($G^{(-i)},\bB',\beta'$) satisfies $(\cA_0)$, we have 
     \[
      \big|W_{G,\bB,\beta}\big|\leq  \big|W_{G^{(-i)},\bB',\beta')}\big| [q+ (p-q)\1_{i\in R}\1_{d_{i,\mathrm{lo}}=0}]\ .
    \]  
    \item[(b)] If $d_i=2$, then, for some $\bB'$, $\beta'$ such that ($G^{(-i)},\bB',\beta'$) satisfies $(\cA_0)$, we have
    \[
      W_{G,\bB,\beta}= W_{G^{(-i)},\bB',\beta'}\ . 
      \]
      \item[(c)] If $d_i=1$ and the unique neighbor of $j$ satisfies $d_j\geq 3$, then, for some $\bB'$, $\beta'$ such that ($G^{(-i)},\bB',\beta'$) satisfies $(\cA_0)$, we have    
      \[
        |W_{G,\bB,\beta}| \leq    \left|W_{G^{(-i)},\bB',\beta'} \right|\left[\sqrt{q}+ [\sqrt{p}-\sqrt{q}]\1_{i\in R}\1_{d_{i,\mathrm{lo}}=0}\right]\enspace .
      \] 
  \item[(d)] If $d_i=1$ and the unique neighbor of $j$ satisfies $d_j=1$, then, for some $\bB'$, $\beta'$ such that ($G^{(-i,-j)},\bB',\beta'$) satisfies $(\cA_0)$, we have
  \[
    \left|W_{G,\bB,\beta}\right| \leq   \left|W_{G^{(-i,-j)},\bB',\beta'}\right|\left[(\sqrt{q}+ [\sqrt{p}-\sqrt{q}]\1_{i\in R}\1_{d_{i,\mathrm{lo}}=0})(\sqrt{q}+ [\sqrt{p}-\sqrt{q}]\1_{i\in R}\1_{d_{i,\mathrm{lo}}=0})\right]\enspace .
  \]
  \end{enumerate}
  \end{lem}

  In view of this lemma, we iteratively prune the graph $G$ by  removing   vertex with external degree $0$  (case (a) above), or a vertex with external degree 2 (case (b) above) if no vertex with external degree 0 is left, or a vertex with external degree 1(cases (c) and (d) above) if no vertices with external  degrees $0$ or $2$ vertices are left. Then, the pruning process stops when the resulting graph $G'=(V',E')$ is  either  empty or when all its vertices have an external degree higher or equal to $3$. Below, we use the convention that $W_{G',\bB',\beta'}=1$ if $G'$ is empty
  
  Denote $V_{r,0,0}\subset R$ the subset of vertices $i$ in $G$ that have a label $r$ and satisfy $d_i=d_{i,\mathrm{lo}}=0$. Define $V_{*,0,*}\subset V$ the subset of vertices $i$ in $G$ satisfying $d_i=0$.  
  Denote $V_{r,1,0}\subset R$ the subset of vertices $i$ in $G$ that have a label $r$ and satisfy $d_i=1$ and $d_{i,\mathrm{lo}}=0$.  Denote $R'= R\cap V'$ and $C'=C\cap V'$ the vertices of $G'$ of type $r$ and $c$ respectively. The following lemma is an application of Lemma~\ref{lem:decompo_graph_1} above and controls $|W_{G,\bB,\beta}|$ in terms of $|W_{G',\bB',\beta'}|$. 
  
  \begin{lem}\label{lem:result_pruning}
  We have $V'\subset V\setminus  (V_{*,0,*}\cup V_{r,1,0})$ and  
  \begin{equation}\label{eq:upper_WGBbeta}
    |W_{G,\bB,\beta}|\leq |W_{G',\bB',\beta'}|p^{|V_{r,0,0}|+ |V_{r,1,0}|/2}q^{|V_{*,0,*}|-|V_{r,0,0}||+ [|V|-|V'|-|V_{*,0,*}|- |V_{r,1,0}|]/2]}\enspace ,
  \end{equation}
  where  $(G',\bB',\beta')$ satisfy $(\cA_0)$. 
  \end{lem}
  
  It remains to control quantities of the form $|W_{G',\bB',\beta'}|$ for graphs $G'$ whose vertices have all an external degree at least $3$. This is slightly more tricky than the previous arguments. See the proof of the following dedicated result.

  \begin{prp}\label{property_graph_multiple}
  Consider any multi-graph $G'=(V',E')$ whose vertices have an external degree at least $3$. Let $(\bB', \beta')$ be any matrices and vectors  such that $(G',\bB',\beta')$ satisfies Property $(\cA_0)$. Then, we have
    \beq\label{eq:upper_W}
    W_{G',\bB',\beta'}\leq q^{|V'|/2}\left(\frac{p}{q}\right)^{(|R'|-1)_+/2}\ ,
    \eeq
    where we recall that $[x]_+= \max(x,0)$ and $R'\subset V'$ is the subset of vertices of type '$r$'. 
  \end{prp}
  
  Gathering the two previous lemmas leads to 
  \[
    |W_{G,\bB,\beta}|\leq p^{|V_{r,0,0}|+ |V_{r,1,0}|/2}q^{|V_{*,0,*}|-|V_{r,0,0}||+ [|V|-|V_{*,0,*}|- |V_{r,1,0}|]/2]}\left(\frac{p}{q}\right)^{(|R'|-1)_+/2} \enspace . 
  \]
To conclude, we need to upper bound $|R'|$. Since all vertices of with external  degree $0$ or $1$ have been pruned in the above iterative algorithm, the vertex set $V'$ is only made of vertices of degree at least $2$. Hence, $R'\subset V_{r,2,*}$, the collection of vertices of type '$r$' with external degree at least $2$. The result follows.

  \subsection{Proof of  Lemmas~\ref{lem:decompo_graph_1} and \ref{lem:result_pruning}}
  
  \begin{proof}[Proof of Lemma~\ref{lem:result_pruning}]
    Let us explain how we obtain~\eqref{eq:upper_WGBbeta}. As we first prune all vertices with external degree $0$ , it follows from Cases (a) and (b) in Lemma~\ref{lem:decompo_graph_1} that 
    \begin{equation}\label{eq:upper_WGBbeta0}
      |W_{G,\bB,\beta}|\leq |W_{G'',\bB'',\beta''}|p^{|V_{r,0,0}|}q^{|V_{*,0,*}|-|V_{r,0,0}|}\enspace ,
    \end{equation}
    where $G''$ is subgraph of $G$ induced by $V\setminus V_{*,0,*}$ and $(G'',\bB'',\beta'')$ satisfies $(\cA_0)$. Then, we observe that no pruning step in Lemma~\ref{lem:decompo_graph_1} increases the external degree of any vertex. Therefore, at some point, the procedure prunes all the vertices $i$ satisfying $d_i=1$ in the original graph $G$.

    If, at some point of the pruning procedure, we have decreased the degree $d_j$ of a vertex to $0$, then this implies that, in the previous step, we have pruned its unique neighbor, which was therefore an external degree $2$ vertex (case (b) in Lemma~\ref{lem:decompo_graph_1}), so that the vertex $j$ must have now a self-loop. When we prune this vertex $j$ with external degree $0$, we pay a total price of $q$ for having pruned two vertices (cases (a) and (b) in Lemma~\ref{lem:decompo_graph_1}). 

    If, at some point of the pruning procedure, we have decreased the degree $d_j$ of a vertex to $1$, then this vertex must now contain a self-loop. Indeed, the external degree of $j$ was at least $2$ in the previous step. It is not possible that it was equal to $2$, because this would imply that we have pruned an external degree one vertex before an external degree 2 vertex. Similarly, it is not possible that the external degree of $j$ was larger than $3$, otherwise its pruned neighbor would have been of external degree at least $3$. Finally, if the external degree of $j$ was equal to $3$, than its pruned neighbor $i$ was satisfying $d_i=2$ and pruning $i$ creates a self-loop in $j$ (see case (b) in Lemma~\ref{lem:decompo_graph_1}). Since $j$ has a self-loop, it follows from  cases (c) and (d) in Lemma~\ref{lem:decompo_graph_1}), we pay a price $\sqrt{q}$ when we prune $j$.

    In summary, in the pruning process, all vertices with external degree at most $2$ have been pruned. We pay a price $p$ only when pruning vertices $i\in R$ with $d_i=d_{i,lo}=0$ in the original graph $G$. We pay a price $\sqrt{p}$ only when pruning vertices $i\in R$ with $d_i=1$ and $d_{i,lo}=0$ in the original graph $G$. We pay a price $q$ only for vertices satisfying $d_i=0$ and  $i\in C$ or $d_{i,lo}>0$ in the original graph $G$. Otherwise, we pay a $\sqrt{q}$ price for all the other pruned vertices. The result follows.  
    \end{proof}

\begin{proof}[Proof of Lemma~\ref{lem:decompo_graph_1}]

    We start with a straightforward property that will be used several times through the manuscript.
  
   \begin{lem}\label{lem:op+2} If a matrix   $\bB$ satisfies $(\cB_{0})$, then  $\|\bB\|_{2}\leq \sqrt{q}$.
   \end{lem}
    
    \begin{proof}[Proof of lemma~\ref{lem:op+2}]
    Indeed, $\|\bB\|^2_{2}=\sum_{i} \sigma_i^2(\bB)\leq \sigma_1^2(\bB)q = q$ because $\|\bB\|_{\infty}=1$ and the rank of $\bB$ is at most $q$. 
   \end{proof}

    If $(G,\bB,\beta)$ satisfies $(\cA_0)$ and if $G$ contains a vertex $i$ with external degree $d_i=0$. Then, we consider $G^{(-i)}$ as the graph induced by $V\setminus\{i\}$, $\bB^{(-i)}= \bB$, and $\beta^{(-i)}$ as the restriction of $\beta$ to self-loops that not incident to $i$. 
    We start from the decomposition
    \[
      W_{G,\bB,\beta}= W_{G^{(-i)},\bB^{(-i)},\beta^{(-i)}} \sum_{\omega_i}\prod_{e\in E_{lo,i}}\beta^{(e)}_{\omega_i}\ , 
    \]
    where $E_{lo,i}$ is the subset of self-loops that are incident to $i$ and where we use the convention $\prod_{e\in \emptyset}\beta^{(e)}_{\omega_i} =1$. Assume first that $E_{lo,i}=\emptyset$, then the above sum over $\omega_i$ equals $p$ if $i\in R$ and $q$ if $i\in C$.  If $E_{lo,i}\neq\emptyset$, then one deduces from Condition~($\cA_0$) that the vector  $\gamma=(\prod_{e\in E_{lo,i}}\beta^{(e)}_{\omega_i})$ satisfies $|\gamma|_1\leq q$, which yields
   \beq\label{eq:upper_degre0}
      \big|W_{G,\bB,\beta}\big|\leq  \big|W_{G^{(-i)},\bB^{(-i)},\beta^{(-i)}}\big| [q+ (p-q)\1_{i\in R}\1_{d_{i,\mathrm{lo}}=0}]\ .
   \eeq

  \medskip
  
    If $G$ contains a vertex $i$ with external degree $d_i=2$, then we define 
    $G^{(-i)}$ as the subgraph $G$ induced by $V\setminus\{i\}$ where we have added a new edge $e'$ between the neighbors of $i$. If these neighbors, say $j$ and $k$, are distinct, then the new matrix $\bB^{(e')}$ is defined by 
    \[
    \bB^{(e')}_{\omega,\omega'}= \sum_{\omega_i} \bB^{(j,i,1)}_{\omega,\omega_i}\bB^{(k,i,1)}_{\omega_i,\omega'} \prod_{e\in E_{lo,i}}\beta^{(e)}_{\omega_i}\ .
    \]
This matrix $\bB^{(e')}$ satisfies $(\cB_0)$.  Indeed, $\bB^{(e')}$ is a product of the matrices $\bB^{(j,i,1)}$ and  $\bB^{(k,i,1)}$ and of $|E_{lo,i}|$ diagonal matrices with diagonal entries given by the $\beta^{(e)}$. By Conditions $(\cB_0)$ and $(\cB_{\mathrm{lo}}^{-})$, the operator norm of all these matrices is smaller or equal to one, whereas the rank of the two first matrices in the product is smaller or equal to $q$. Hence, $\bB^{(e')}$ satisfies $(\cB_0)$. We have
 $W_{G,\bB,\beta}=  W_{G^{(-i)},\bB^{(-i)},\beta^{(-i)}}$, where $\bB^{(-i)}$ is the collection $\bB$ of matrices  where we have removed $\bB^{(j,i,1)}$ and $\bB^{(k,i,1)}$ and added $\bB^{(e')}$ and where $\beta^{(-i)}$ is the collection $\beta$ of vectors  where we have removed all vectors corresponding to self-loops incident to $i$. Since the external degrees of the vertices $j$ and $k$ are left unchanged, this implies that $(G^{(-i)},\bB^{(-i)},\beta^{(-i)})$ satisfies ($\cA_0$). 

 If $i$ has a unique neighbor, say $j$, then the new edge $e'$ is a self-loop incident to $j$. The  new vector $\beta^{(e')}$ is defined by 
    \[
    \beta^{(e')}_{\omega}=   \sum_{\omega_i} \bB^{(j,i,1)}_{\omega,\omega_i}\bB^{(j,i,2)}_{\omega,\omega_i} \prod_{e\in E_{lo,i}}\beta^{(e)}_{\omega_i}\ . 
    \]
  Applying Young's inequality and relying on the fact that the Frobenius norms of $\bB^{(j,i,1)}$ and $\bB^{(j,i,2)}$ are at most $\sqrt{q}$ (Lemma~\ref{lem:op+2}), we derive that $|\beta^{(e')}|_1\leq q$. Besides, $|\beta^{(e')}|_{\infty}\leq \|\bB^{(j,i,1)}\|_{\infty}\|\bB^{(j,i,2)}\|_{\infty}\leq 1$. As a consequence, $\beta^{(e')}$ satisfies $(\cB_{\mathrm{lo}})$. If the vertex $j$ had no previous self-loop, we consider $\bB^{(-i)}$ as the collection  $\bB$ of matrices where we have removed $\bB^{(j,i,1)}$ and $\bB^{(j,i,2)}$. Besides, we consider  $\beta^{(-i)}$ as the collection $\beta$ of vectors  where we have removed all vectors corresponding to self-loops incident to $i$ and we have added $\beta^{(e')}$. Hence, $(G^{(-i)},\bB^{(-i)},\beta^{(-i)})$ satisfies ($\cA_0$) and we have $W_{G,\bB,\beta}=  W_{G^{(-i)},\bB^{(-i)},\beta^{(-i)}}$. It remains to consider the case where $j$ had already at least one self-loop. We cannot simply use $(G^{(-i)},\bB^{(-i)},\beta^{(-i)})$ as before because the external degree of $j$ is possibly smaller than one while some of the vectors  $\beta^{(e)}$ associated to self-loops incident to $j$ do not necessarily satisfy ($\mathcal{B}_{\mathrm{lo}}$). To resolve this small issue, we merge all the self-loops of $j$ together and replace them  by a single self-loop, say $e''$ whose associated vector $\beta^{(e'')}$ is the coordinate-wise product of the vectors, i.e. $\beta^{(e'')}_{\omega}=\beta^{(e')}_{\omega}\prod_{e\in E_{l_0,j}}\beta^{(e)}_{\omega}$. One easily checks that $\beta^{(e'')}$ satisfied ($\cB_{\mathrm{lo}}$). The result follows.

    \medskip 
  
    Finally, we consider the case where  $G$ contains a vertex $i$ with external degree $d_i=1$ and we denote $j$ its unique neighbor.  
  \begin{itemize}
    \item If $d_j=1$. Then, we have 
    \[
      |W_{G,\bB,\beta}|\leq   \left|W_{G^{(-i,-j)},\bB^{(-i,-j)},\beta^{(-i,-j)}}\right|\left[(\sqrt{q}+ [\sqrt{p}-\sqrt{q}]\1_{i\in R}\1_{d_{i,\mathrm{lo}}=0})(\sqrt{q}+ [\sqrt{p}-\sqrt{q}]\1_{i\in R}\1_{d_{i,\mathrm{lo}}=0})\right]\enspace , 
   \]
  where $G^{(-i,-j)}$ is the subgraph induced by $V\setminus \{i,j\}$. 
  Indeed, denoting $e'$ the unique edges between $i$ and $j$, we have the decomposition 
  \[
    W_{G,\bB,\beta} =   W_{G^{(-i,-j)},\bB^{(-i,-j)},\beta^{(-i,-j)}} \sum_{\omega,\omega'}\bB^{(e')}_{\omega,\omega'} \prod_{e\in E_{lo,i}}\beta^{(e)}_{\omega}\prod_{e\in E_{lo,j}}\beta^{(e)}_{\omega'}\ , 
  \]
  and we only have to bound this sum over $\omega$ and $\omega'$. Since $\|\bB^{(e')}\|_{\infty}\leq 1$, we have 
  \[
  \left|\sum_{\omega,\omega'}\bB^{(e')}_{\omega,\omega'} \prod_{e\in E_{lo,i}}\beta^{(e)}_{\omega}\prod_{e\in E_{lo,j}}\beta^{(e)}_{\omega'} \right| \leq \Bigg[\underset{=A}{\underbrace{\sum_{\omega}\left(\prod_{e\in E_{lo,i}}\beta^{(e)}_{\omega}\right)^2}}\Bigg]^{1/2} \Bigg[\sum_{\omega'}\left(\prod_{e\in E_{lo,j}}\beta^{(e)}_{\omega'}\right)^2\Bigg]^{1/2}\ .
  \]
  By symmetry, we only need to consider the sum over $\omega$. If $d_{i,\mathrm{lo}}=0$, then we see that $A=q$ if $i\in C$ and $A=p$ if $i\in R$.  If $d_{i,\mathrm{lo}}>0$, then, we 
  deduce from Condition $(\cB_{\mathrm{lo}})$ that 
  \[
    A\leq \sum_{\omega } \left|\prod_{e\in E_{lo,i}}\beta^{(e)}_{\omega}\right|\leq q\enspace , 
  \]
  which concludes the proof. 
  
    \item If $d_j\geq 3$ then we consider the subgraph $G^{(-i)}$ induced by $V\setminus \{i\}$ where we have added a self-loop $e'$ incident to $j$. We have 
    \[
      W_{G,\bB,\beta}= \sum_{(\omega_l)_l\in V\setminus\{i,j\}}\sum_{\omega_j}\prod_{e\in E_0^{(-i)}}\bB^{(e)}_{\omega_{e_1}, \omega_{e_2}}\prod_{e\in E_{\mathrm{lo}}^{(-i)}}\beta^{e}_{\omega_{e_1}}\left[
     \sum_{\omega_i}\bB^{(j,i,1)}_{\omega_j,\omega_i} \prod_{e\in E_{lo,i}}\beta^{(e)}_{\omega_j}\right]\ , 
    \]
    where $E_{\mathrm{lo}}^{(-i)}$ is the set of self-loops that are not incident to $i$ and $E_0^{(-i)}$ is the subset of edges not incident to $i$.
    Define the vector $\beta^{e'}$ by 
    \[
    \beta^{e'}_{\omega}= \left[\sqrt{q}+ [\sqrt{p}-\sqrt{q}]\1_{i\in R}\1_{d_{i,\mathrm{lo}}=0}\right]^{-1} \sum_{\omega_i}\bB^{(j,i,1)}_{\omega,\omega_i} \prod_{e\in E_{lo,i}}\beta^{(e)}_{\omega_j}\ . 
    \]
Obviously we have 
\[
  |W_{G,\bB,\beta}| \leq    \left|W_{G^{(-i)},\bB^{(-i)},\beta^{(-i)}}\right|\left[\sqrt{q}+ [\sqrt{p}-\sqrt{q}]\1_{i\in R}\1_{d_{i,\mathrm{lo}}=0}\right]\enspace , 
\] 
Since the external degree of $j$ is still larger or equal to $2$ after the pruning step, we only have to prove that  $\beta^{(e')}$ satisfies $(\cB^{-}_{\mathrm{lo}})$.
 Indeed, define the vectors $\gamma$ and $\delta$ by 
  \[
  \gamma_{\omega}= \sum_{\omega_i}\bB^{(j,i,1)}_{\omega,\omega_i} \prod_{e\in E_{lo,i}}\beta^{(e)}_{\omega_i} \ ; \quad \delta_{\omega}=\prod_{e\in E_{lo,i}}\beta^{(e)}_{\omega} .
  \]
Since the  $\beta^{(e)}$'s satisfy $\cB_{\mathrm{lo}}$ (as $d_i=1$), we have $|\delta|_{\infty}=1$ and $|\delta|_1= \sqrt{q}+ [\sqrt{p}-\sqrt{q}]\1_{i\in R}\1_{d_{i,\mathrm{lo}}=0}$. Since $\bB^{(i,j,1)}$ satisfies ($\cB_0$), it follows that the vector $\gamma$ satisfies 
\[
  |\gamma|_{2}=  |\bB^{(i,j,1)}\delta|_2\leq \| \bB^{(i,j,1)}\|_{\infty}|\delta|_2\leq \sqrt{q}+ [\sqrt{p}-\sqrt{q}]\1_{i\in R}\1_{d_{i,\mathrm{lo}}=0}\ .   
\]
This implies that the  $|\beta^{(e')}|_2\leq 1$ and $\beta^{(e')}$ therefore satisfies $(\cB^{-}_{\mathrm{lo}})$.

   
  \end{itemize}
  \end{proof}

  \subsection{Proof of Proposition~\ref{property_graph_multiple}}

\begin{proof}[Proof of Proposition~\ref{property_graph_multiple}]

We shall establish the desired bound under different conditions and then show how to transform $(G', \bB', \beta')$ to satisfy these new conditions. A multigraph $G = (V, E)$ is said to satisfy Property $(\cG^*)$ if:
  \begin{itemize}
    \item[(a)] $G$ does not contain any self-loop. 
    \item[(b)] Any pair of vertices is connected by at most 2 edges.
    \item[(c)] The minimum degree is larger or equal to $2$.
    \item[(d)] If the degree of a vertex is $2$, the this vertex has only one neighbor. In other words, this vertex shares two edges with another vertex.
  \end{itemize}

A matrix  $\bC$ is said to satisfy Condition $(\cB^*)$ if (a) the entries of $\bC$ are non-negative,  if (b) the $l_2$ norm of each row and each column of $\bC$ is at most one, and if  (c) the Frobenius norm of $\bC$ is at most $\sqrt{q}$. 

Finally, we say that $(\bG, \bB)$ where $\bB= (\bB^{(e)})$ is a collection of matrices indexed by the edges of $G$ satisfies Property $(\cA^*)$ is $G$ satisfies $(\cG^*)$ and all matrices $\bB^{(e)}$ satisfy $(\cB^*)$.

\begin{prp}\label{property_graph_multiple_2}
  For any $(G,\bB)$ satisfying Condition ($\cA^*$), we have 
  \beq\label{eq:upper_W2}
  W_{G,\bB}\leq q^{|V|/2}\left(\frac{p}{q}\right)^{(|R|-1)_+/2}\ ,
  \eeq
  where we recall that $[x]_+= \max(x,0)$ and $R\subset V$ is the subset of vertices of type '$r$'. 
  \end{prp}

Let us show how to apply this result to $W_{G',\bB',\beta'}$. First, we transform $G'$ into a multigraph $G^{\uparrow}$ by erasing the self-loops of $G'$. Recall that the external degree of each vertex in $G'$ is at least $3$. For  any vertex $i\in V$, we pick an external edge $e=(k,i,a)$ incident to $i$ and we modify the corresponding matrix $(\bB')^{(e)}$ into $(\bB^{\uparrow})^{(e)}$ by setting $(\bB^{\uparrow})^{(e)}_{\omega, \omega'}= (\bB')^{(e)}_{\omega, \omega'}\prod_{e\in E_{l_0,i}}\beta^{(e)}_{\omega}$. One checks that the resulting matrix  $(\bB^{\uparrow})^{(e)}$ still satisfies $(\cB_{0})$. For all the other external edges $e$, we set 
$(\bB^{\uparrow})^{(e)}= \bB^{(e)}$.
As a consequence,  the resulting  multigraph $G^{\uparrow}= (V, E^{\uparrow})$ is such that $(G^{\uparrow}, \bB^{\uparrow})$ satisfies $(\cA_0)$ and 
$W_{G^{\uparrow}, \bB^{\uparrow}}=   W_{G',\bB',\beta'}$. Next, we modify $G^{\uparrow}$ by erasing some edges in such a way that any pair of vertices is connected by at most twice edges. We call $G^{\ddagger}$ the resulting multigraph. Since the degree of any vertex in  
$G^{\uparrow}$ is larger or equal to $3$, one easily checks that $G^{\ddagger}$ satisfies Condition ($\cG^*$). Then, we define the collection of matrix $\bB^{\ddagger}$ as follows. For any pair of vertices that are connected by at most 2 edges, we simply define $(\bB^{\ddagger})^{(e)}$ by 
$(\bB^{\ddagger})^{(e)}_{\omega,\omega'}= \big|(\bB^{\uparrow})^{(e)}_{\omega,\omega'}\big|$ for the corresponding edges. If there are more than $2$ edges, then the first matrix $(\bB^{\ddagger})^{(e)}$ is defined as above, whereas the second matrix $(\bB^{\ddagger})^{(e)}$ corresponding to the second edge
is defined by the product of the absolute values of the entries of the remaining matrices $(\bB^{\uparrow})^{(e')}$. Equipped with these notation, we straightforwardly have 
\beq\label{eq:upper_W7}
  \big|W_{G', \bB',\beta'}\big|=  \big|W_{G^{\uparrow}, \bB^{\uparrow}}\big|\leq   W_{G^{\ddagger}, \bB^{\ddagger}}\ .
\eeq
It remains to show that the matrices $(\bB^{\ddagger})^{e}$ all satisfy Condition $(\cB^*)$. Since the maximum entry (in absolute value value) of all matrices $(\bB^{\uparrow})^{e}$ is at most one because $\|(\bB^{\uparrow})^{e}\|_{\infty}\leq 1$, it suffices to establish the following claim: if a matrix $\bC$ satisfies Condition $(\cB_0)$, then the matrix $\overline{\bC}$ defined by  $\overline{\bC}_{\omega,\omega'}= |\bC_{\omega,\omega'}|$ satisfies $(\cB^*)$. Indeed, it follows from Lemma~\ref{lem:op+2} that $\|\overline{\bC}\|_2=\|\bC\|_2\leq \sqrt{q}$. Besides, the $l_2$ norm of each row and each column of $\overline{\bC}$ is upper bounded by the operator norm of $\bC$ and is therefore smaller than one. The claim follows. Since $(G^{\ddagger}, \bB^{\ddagger})$ satisfies $(\cA^*)$, we are in position to apply Proposition~\ref{property_graph_multiple_2} and the result follows from~\eqref{eq:upper_W7}.

\end{proof}

Before turning to the proof  Proposition~\ref{property_graph_multiple_2}, we state a simple result that will be used several times. 

\begin{lem}\label{lem:C1_C2}
  Let $\bC_1$ and $\bC_2$ be any two $r\times s$ and $s\times t$ matrices that satisfy Property $(\cB^*)$. Then the product matrix $q^{-1/2}\bC_1\bC_2$ also satisfies Property $(\cB^*)$. 
\end{lem}

\begin{proof}[Proof of Proposition~\ref{property_graph_multiple_2}]

We prove this result by induction on the number $|V|$ of vertices. We start by considering any pair $(G, \bB)$ with $|V|= 2$ vertices that satisfies $(\mathcal{A}^*)$.  Since the two vertices are connected by exactly two edges, we  derive from Young's inequality that 
    \begin{eqnarray*}
      W_{G,\bB} &= &\sum_{\omega_1,\omega_2=1}^{p}\prod_{e\in E} \bB^{(e)}_{\omega_1,\omega_2}\leq \frac{1}{2}\sum_{e\in E}\sum_{\omega_1,\omega_2=1}^{p} (\bB^{(e)}_{\omega_1,\omega_2})^2
      \leq q\leq q^{|V|/2}\left(\frac{p}{q}\right)^{(|R|-1)_+/2}\ ,
    \end{eqnarray*}
    where we used that  the Frobenius norms of the matrices $\bB^{(e)}$ is at most $\sqrt{q}$ by Condition~($\mathcal{B}$).

    \medskip

    \begin{figure}[h]
      \centering 
      \subfloat[Case 1.a]{{\includegraphics[width=3.5cm]{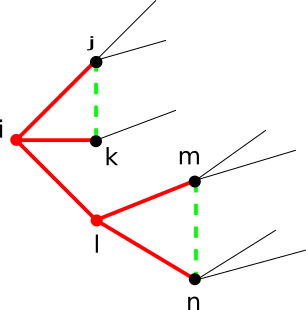}}}\quad 
      \subfloat[Case 1.b.i]{{\includegraphics[width=3.5cm]{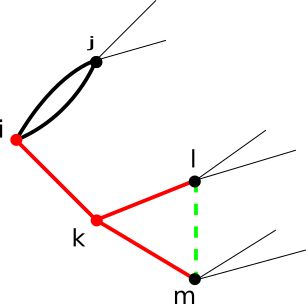}}}\quad 
      \subfloat[Case 1.b.ii]{{\includegraphics[width=3.5cm]{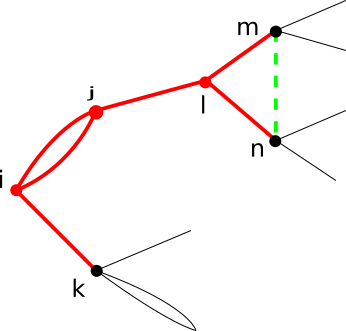}}}\quad 
      \subfloat[Case 3]{{\includegraphics[width=3cm]{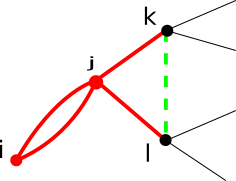}}}
      \caption{Examples of pruned graphs in the proof of Proposition~\ref{property_graph_multiple_2}. 
        Vertices and edges in red are removed, whereas edges in dotted green are inserted. \label{fig:proof_multiple_graph} }    
   \end{figure}

  Let us now consider any $(G,\bB)$ satisfying $(\cA^*)$ with $|V|\geq 3$. 
  If the graph $G$ is not connected, we write $G_1=(V_1,E_1)$, $G_2=(V_2,E_2)$,\ldots, $G_K=(V_K,E_K)$ its decomposition into connected components and $\bB_1$, \ldots, $\bB_K$ the corresponding collection of matrices. Obviously, each $(G_j,\bB_j)$ satisfies the property $(\cA^*)$. Besides, we have  $W_{G,\bB} = \prod_{j=1}^K   W_{G_j,\bB_j}$. By induction, we conclude that  $W_{G,\bB} \leq q^{|V|/2}\left(\frac{p}{q}\right)^{(|R|-1)_+/2}$. We focus henceforth on the situation where $G$ is connected. Since $|V|\geq 3$, this implies that at least one vertex has a degree larger or equal to $3$. The analysis is divided into three cases.

  \noindent
  {\bf CASE (1)}: there exists a least one vertex whose degree is exactly $3$. Then, we pick any degree $3$ vertex $i$ of type $r$  if there exists one and otherwise we pick any degree $3$ vertex $i$ of type $c$. There are two sub-cases depending on the number of distinct neighbors of $i$. 
  
  \medskip
  
  \noindent 
  {\bf CASE (1.a)}: the vertex $i$ has three distinct neighbors $j$, $k$, and $l$ in $G$ --see Figure~\ref{fig:proof_multiple_graph}. Since $W_{G,\bB}$ is decreasing with the number of edges, we start by removing the edge between $i$ and $l$ from $E$ and the corresponding matrix $\bB^{(i,l,1)}$ from $\bB$. We write $G^p$ the corresponding pruned graph and $\bB^p$ the corresponding collection of matrices. Focusing on the sum over $w_i$ in $W_{G,\bB}$, we have 
  \begin{eqnarray*}
 W_{G,\bB} \leq   W_{G^{p},\bB^{p}} = \sum_{\omega_t,\ t\in V\setminus\{i\}}\prod_{e\in E^{(-i)}} \bB^{(e)}_{\omega_{e_1},\omega_{e_2}} \sum_{\omega_i}\prod_{e\in E^{p}_{i}}\bB^{(e)}_{\omega_{e_1},\omega_{e_2}}\ ,
  \end{eqnarray*}
  where $E^{p}_{i}$ is the set of edges incident to $i$. 
  We denote $e'=(i,j,1)$ and $e''=(i,k,1)$ the two edges incident to $i$ in $G^p$. Then, we introduce the matrix $\bC$ defined by 
 $ \bC_{\omega,\omega'}=\sum_{\omega_i} \bB^{(e')}_{\omega_i,\omega}\bB^{(e'')}_{\omega_i,\omega'}$, 
  where the sum over $\omega_i$ runs over $[p]$ if $i$ is of type '$r$' and over $[q]$ if $j$ is of type '$c$'. From Lemma~\ref{lem:C1_C2}, we deduce that 
  $\bC/\sqrt{q}$ satisfies Condition $(\cB^*)$, so that 
  \[
    W_{G^p,\bB}\leq \sum_{\omega_t,\ t\in V\setminus\{i\}}\prod_{e\in E^{(-i)}} \bB^{(e)}_{\omega_{e_1},\omega_{e_2}} \bC_{\omega_{j},\omega_{k}}\ .
  \]
  We now consider the modification $\overline{G}$ and $\overline{\bB}$ of $G^p$ and $\bB^p$ by removing the vertex $i$ and, if there are at most one edge between $i$ and $k$, we insert a new edge $\overline{e}$ between $j$ and $k$ and associate the matrix $\overline{\bB}^{(\overline{e})}= \bC/\sqrt{q}$. All the other matrices $\bB^{(e)}$ are left unchanged. Then, we have 
  \[
    W_{G,\bB}\leq \sqrt{q}W_{\overline{G},\overline{\bB}}\ .
  \]
  Among the remaining vertices in $\overline{G}$, only the vertex $l$ possibly does not meet the conditions in $(\cG^*)$. If $l$ satisfies those condition, then ($\overline{G},\overline{\bB}$) satisfies $(\cA^*)$, and by induction, we conclude that 
  \[
    W_{G,\bB}\leq \sqrt{q}q^{|\overline{V}|/2}\left(\frac{p}{q}\right)^{(|\overline{R}|-1)_+/2}\leq q^{|V|/2}\left(\frac{p}{q}\right)^{(|R|-1)_+/2}\ ,
  \]
  where we denote $\overline{R}$ the collection of vertices of type '$r$' in $\overline{G}$. If $l$ does not satisfy the degree condition in $(\cG^*)$, then this implies that the degree of $l$ in $G^p$ is $2$ and that $l$ is connected to two distinct vertices, say $m$ and $n$. Denoting  $e'=(l,m,1)$ and $e''=(l,n,1)$ these two edges, we consider as above the matrix $\bC'$ defined by  $\bC'_{a,b}=\sum_{\omega_l} \bB^{(e')}_{\omega_l,a}\bB^{(e'')}_{\omega_l,b}$.
Again, the matrix $\bC'/\sqrt{q}$ satisfies $(\cB^*)$ by Lemma~\ref{lem:C1_C2}. Then, 
  \[
    W_{\overline{G},\overline{\bB}}\leq \sum_{\omega_s,\ s\in V\setminus \{i,l\}}\prod_{e\in \overline{E}^{-l}} \overline{\bB}^{(e)}_{e_1,e_2} \bC_{\omega_{m},\omega_{n}}\ .
  \]
  Finally, we consider the transformations $(\overline{G}',\overline{\bB}')$  of $(\overline{G},\overline{\bB})$ where we remove the vertex $l$ and, if the number of edges between $m$ and $n$ is at most one, we insert an edge $\tilde{e}$ between $m$ and $n$ and introduce the corresponding matrix $\overline{\bB}'^{(\tilde{e})}= \bC/\sqrt{q}$. Then, $(\overline{G}',\overline{\bB}')$ satisfies Condition ($\cA^*$) and  we conclude  by induction that
  \[
    W_{G,\bB}\leq qq^{|\overline{V}'|/2}\left(\frac{p}{q}\right)^{(|\overline{R}'|-1)_+/2}\leq q^{|V|/2}\left(\frac{p}{q}\right)^{(|R|-1)_+/2}\ .
  \]

  \noindent 
  {\bf CASE (1.b)}: the vertex $i$ has only two neighbors $j$ and  $k$ in $G$. Without loss of generality, we can assume that there are two edges between $i$ and $j$ and only one between $i$ and $k$. Since $W_{G,\bB}$ is non-decreasing with the number of edges, we start by removing the edge between $i$ and $k$ and call $G^p$ the corresponding pruned graph. We consider again two sub-cases depending on whether the vertex $k$ satisfies the degree condition of  $(\cG^*)$ in  the pruned graph $G^p$. 
  \medskip 
  
  \noindent 
  {\bf CASE (1.b.i)} The vertex $k$ does not satisfy the degree condition of $(\cG^*)$ in the pruned graph $G^p$  --see Figure~\ref{fig:proof_multiple_graph}. Then, this implies that, in $G^p$, $d_k=2$ and that $k$   is connected to two distinct vertices, say $l$ and $m$.  argue as in Case 1.a by introducing the matrix $\bC$ defined by 
  \[
    \bC_{\omega,\omega'}=\sum_{\omega_k} \bB^{(e)}_{\omega_k,\omega}\bB^{(e')}_{\omega_k,\omega'}\ .
  \]
  where $e$ and $e'$ respectively stand for the edges $e= (k,l,1)$ and $e'=(k,m,1)$. The matrix $\bC/\sqrt{q}$ satisfies Condition $(\cB^*)$ by Lemma~\ref{lem:C1_C2}. Then,  we consider the transformation $(\overline{G},\overline{\bB})$ of $(G^p,\bB^p)$, where we remove the vertex $k$ and, if the number of edges between $l$ and $m$ is at most one, we insert an edge $\overline{e}$ between $l$ and $m$  with the matrix $\overline{\bB}^{(\overline{e})}= \bC/\sqrt{q}$. Then, $(\overline{G},\overline{\bB})$ satisfies Condition $(\cA^*)$ and $W_{G,\bB}\leq \sqrt{q}W_{\overline{G},\overline{\bB}}$.  By induction, we conclude that 
  \[
    W_{G,\bB}\leq \sqrt{q}W_{\overline{G},\overline{\bB}}\leq  q^{|V|/2}\left(\frac{p}{q}\right)^{(|R|-1)_+/2}\ .
  \]
  
  \medskip  
  
  \noindent 
  {\bf CASE (1.b.ii)} In the pruned graph $G^p$, the vertex $k$ satisfies the conditions of $(\cG^*)$. Hence, $(G^{P},\bB^{P})$ satisfy Condition $(\mathcal{A}^*)$. We now focus on the vertices $i$ and $j$ --recall that they are connected by two edges. If the degree of $j$ is equal to $2$, then $i$ and $j$ form a connected component of size $2$ in $G^p$. Decomposing $G^p$ into an union of at least two connected components, we conclude by induction  as in the beginning of the proof that 
  \[
    W_{G,\bB}\leq  W_{G^p,\bB^P}\leq    q^{|V|/2}\left(\frac{p}{q}\right)^{(|R|-1)_+/2} \ .
  \]
  We now assume that the degree of $j$ is at most $3$ in $G^p$  --see Figure~\ref{fig:proof_multiple_graph}. Let us denote $e=(i,j,1)$ and $e'=(i,j,2)$ the two edges between $i$ and $j$. We derive from Condition $(\cB^*)$ and Young's inequality that, for any $\omega_j$, 
  \[
    \sum_{\omega_i} \bB^{(e)}_{\omega_i,\omega_j}\bB^{(e')}_{\omega_i,\omega_j}\leq \frac{1}{2}\sum_{\omega_i} \left(\bB^{(e)}_{\omega_i,\omega_j}\right)^2+ \left(\bB^{(e')}_{\omega_i,\omega_j}\right)^2\leq 1\ .
  \]
  As a consequence, the new pruned graph $G^{(-i)}$ where we have removed the vertex $i$  and all edges incident to $i$ satisfies 
  \[
    W_{G,\bB}\leq  W_{G^{(-i)},\bB^{(-i)}}\ . 
  \]
  If, in $G^{(-i)}$, the vertex $j$ satisfies the degree conditions of~$(\cG^*)$, then $(G^{(-i)},\bB^{(-i)})$ satisfies $(\cA^*)$ and we conclude by induction. If the vertex $j$ does not satisfy the degree conditions~$(\cG^*)$, then either $d_j=1$ or $d_j=2$ and $j$ has two distinct neighbors. In the latter case, we have already explained (e.g. in Case (1.b.i)) how to factorize the sum  with respect to $\omega_j$ in $ W_{G^{(-i)},\bB^{(-i)}}$  to build a multi-graph $\overline{G}$ with vertex set $V\setminus\{i,j\}$  and a collection  of matrices $\overline{\bB}$ satisfying Condition $(\cA^*)$ and such that 
  \[
     W_{G^{(-i)},\bB^{(-i)}}\leq \sqrt{q}W_{\overline{G},\overline{\bB}}\ .
  \]
  Then, we conclude by induction. It remains to consider the case where $j$ has exactly one neighbor in $G^{(-i)}$. Let us denote $l$ this neighbor and $e=(j,l,1)$ the corresponding edge. In $W_{G^{(-i)},\bB^{(-i)}}$ we focus on the sum
  \[
  \sum_{\omega_j}\bB^{(e)}_{\omega_j,\omega_l} \ . 
  \]
  By Cauchy-Schwarz inequality and Condition~$(\cB^*)$, this sum is bounded by $\sqrt{q}$ if $j$ is of type '$c$' and by $\sqrt{p}$ if $j$ is of type '$r$'. Since the degree of $j$ in $G^{(-i)}$ is 1, this entails that the degree of $j$ in $G$ is equal to 3. Hence, if $j$ is of type '$r$', then this implies that $i$ is of type '$r$' because we have picked $i$ as one of the degree 3 vertices of type '$r$' if there exists at least one. Hence, considering the new pruned graph $(G^{(-i,-j)},\bB^{(-i,-j)})$ where we have also removed the vertex $j$, we arrive at 
  \[
    W_{G,\bB}\leq \sqrt{q}\sqrt{\frac{p}{q}}^{\1_{i,j\in R}} W_{G^{(-i,-j)},\bB^{(-i,-j)}}\ .
  \]
  In $G^{(-i,-j)}$, only the vertex $l$ possibly does not satisfy the degree condition in  $(\cG^*)$. Hence, if $l$ satisfies the degree conditions in $(\cG^*)$, then $(G^{(-i,-j)},\bB^{(-i,-j)})$ satisfies $(\cA^*)$ and we conclude by induction. Otherwise, the vertex $l$ has two distinct neighbors say $m$ and $n$  in $G^{(-i,-j)}$ with corresponding edges $e$ and $e'$. As already  argued several times, we focus the sum over $\omega_l$  of $\bB^{(e)}_{\omega_l,\omega_m} \bB^{(e')}_{\omega_l,\omega_{n}}$. This allows us to build a graph $\overline{G}$ with vertex set $V\setminus \{i,j,l\}$ and a collection   $\overline{\bB}$ of matrices that satisfy Condition~$(\cA^*)$ and such that 
  \[
    W_{G^{(-i,-j)},\bB^{(-i,-j)}}\leq \sqrt{q}W_{\overline{G},\overline{\bB}}\ . 
  \]
  Then, we conclude by induction.
  
  \medskip 
  
  \noindent 
  {\bf  CASE $(2)$}: The degree of all vertices in $G$ is at least $4$. Since the sum $W_{G,\bB}$ does not decrease when we remove edges in $G$, we can iteratively prune the graph by removing some edges until one vertex has a degree $3$. If the pruned graph $G^p$ is not connected, we conclude as in the beginning of the proof. If the pruned graph is connected, then we come back to Case $1$.

  \medskip 
  
  \noindent 
  {\bf CASE $(3)$}: No vertex has a degree $3$ and there exists at least one vertex with degree $2$. We pick one of the degree $2$ vertices, say $i$. Denoting $j$ the neighbor of $i$ in $G$, we observe that the degree of $j$ in $G$ is at least $4$ since $G$ is connected and since $d_j\neq 3$. Arguing exactly as in Case 1.ii.b, we observe that  $(G^{(-i)},\bB^{(-i)})$ where we have removed the vertex $i$ and the two edges incident to $i$ satisfies
  \[
    W_{G,\bB}\leq  W_{G^{(-i)},\bB^{(-i)}}\ . 
  \]
  If, in $G^{(-i)}$, $j$ satisfies the degree conditions of $(\cG^*)$, then $(G^{(-i)},\bB^{(-i)})$ satisfies Condition~$(\cA^*)$ and we conclude by induction. If $j$ does not satisfy the  degree conditions of $\cG$ in $G^{(-i)}$, then $d_j=2$ and $j$ has two distinct neighbors, say $k$ and $l$ --see Figure~\ref{fig:proof_multiple_graph}. As argued several times, we prune the vertex $j$ and possibly create an edge between $k$ and $l$ to build $(\overline{G},\overline{\bB})$ satisfying $(\cA^*)$ and such that 
  \[
   W_{G^{(-i)},\bB^{(-i)}}\leq \sqrt{q}W_{\overline{G},\overline{\bB}}\ .
  \]
  Then, we conclude by induction.

  \end{proof}

  \begin{proof}[Proof of Lemma~\ref{lem:C1_C2}]
    Let us denote $\bD= \bC_1\bC_2$. We only have to prove that $\|\bD\|_2\leq q$ and that the $l_2$ norms of each row and column is smaller or equal to $\sqrt{q}$. Since the Frobenius norm is sub-multiplicative (as long as all Schatten norms), we deduce that 
    \[
    \|\bD\|_2\leq \|\bC_1\|_2\|\bC_2\|_2  \leq q \ , 
    \]
    since both $\bC_1$ and $\bC_2$ satisfy $(\cB^*)$. Consider that $a$-th row of $\bD$. Then, 
    \beqn
    \sum_{b=1}^t \bD^2_{ab}&=& \sum_{b=1}^t \left[\sum_{c=1}^s (\bC_1)_{a,c}(\bC_2)_{c,b} \right]^2 \leq \sum_{b=1}^t \left[\sum_{c=1}^s(\bC_1)^2_{a,c}\right]\left[\sum_{c'=1}^s(\bC_2)^2_{c',b}\right]\\
    &\leq &  \sum_{b=1}^t \sum_{c'=1}^s(\bC_2)^2_{c',b}\leq  q \enspace , 
    \eeqn 
    where we used in the second line that the $l_2$ norm of the $a$-th row of $\bC_1$ is at most one and  that the Frobenius norm of $\bC_2$ is at most $q$. 
    \end{proof}

  \subsection{Proof of Technical Lemmas}

  \begin{proof}[Proof of Lemma~\ref{lem:hermite_subgaussian}]
    Given a positive integer $k$, we write $a_{l,k}$ for the coefficients of the decomposition of the $k$-th Hermite polynomial in the canonical basis, that is $H_k(x)= \sum_{l=0}^{k} a_{l,k}x^{l}$.  First, we consider the case where $k$ is even. Since $H_{k}$ is a symmetric polynomial, we derive that 
    \begin{eqnarray*}
      \E[H_k(Y)] &=&\sum_{l=0}^{k/2} \E[a_{2l,k} (\theta+E)^{2l}] =  \sum_{s=0}^{k} \theta^{s} \sum_{l=\lceil s/2 \rceil }^{k/2}  a_{2l,k} \binom{2l}{2l-s}\E[E^{2l-s}]\\ 
      & = &a_{k,k} \theta^{k}+ a_{k,k} \theta^{k-1} k\E[E] +  \theta^{k-2}\left[a_{k,k}\frac{k(k-1)}{2}\E[E^2]+ a_{k-2,k}  \right]\\ & &+ \sum_{s=0}^{k-3} \theta^{s} \sum_{l=\lceil s/2 \rceil }^{k/2}  a_{2l,k} \binom{2l}{2l-s}\E[E^{2l-s}] \\ 
      & =& \theta^{k} + \sum_{s=0}^{k-3} \theta^{s} \sum_{l=\lceil s/2 \rceil }^{k/2}  a_{2l,k} \binom{2l}{2l-s}\E[E^{2l-s}] \ , 
    \end{eqnarray*}
    because $\E[E]=0$, $\E[E^2]=1$, $a_{k,k}=1$ and $a_{k,k}\frac{k(k-1)}{2}+a_{k-2,k}=0$. The two latter identities are a consequence of the equality $\E[H_{k}(\theta+ Z)]=\theta^k$ when $Z\sim \cN(0,1)$. Similarly, if $k$ is odd, we derive that 
    \[
      \E[H_k(Y)] = \theta^{k} + \sum_{s=0}^{k-3} \theta^{s} \sum_{l=\lceil (s-1)/2 \rceil }^{(k-1)/2}  a_{2l+1,k} \binom{2l+1}{2l+1-s}\E[E^{2l+1-s}] \ . 
    \]
    Next, we upper bound the coefficients of the polynomial $P_{k}$. For $l\equiv k [2]$, it is well known that the coefficient $a_{l,k}$ of the Hermite polynomial satisfies
    \[
    a_{l,k}=  \frac{k!}{2^k l! [(k-l)/2]!}(-1)^{(k-l)/2} \ , 
    \]
    so that 
    \beq\label{eq:upper_bound_hermite_coefficient}
    |a_{l,k}|\leq \frac{1}{2^k}\binom{k}{l}(k-l)^{(k-l)/2}\leq (k-l)^{(k-l)/2}\leq k^{k/2}\ . 
    \eeq
    Besides, we know (e.g.~\cite{vershynin2018high}) that  $|\E[E^k]|\leq (c\|E\|_{\psi_2}\sqrt{k})^k$ for some positive numerical constant $c$. 
    Writing $P_{k}(x)= \sum_{l=0}^{k-3}b_{l,k}x^k$, we derive that $|b_{l,k}|\leq c^k\|E\|_{\psi_2}^k k^{k+1}2^{k}$.
    Let us turn to $\E[H_kH_l(Y)]$. 
    \begin{eqnarray*} 
      \E[H_kH_l(Y)]&= & \sum_{s=0}^{k+l} \sum_{t=(s-l)\vee 0}^{s\wedge k} a_{t,k}a_{s-t,l}\E[(\theta+E)^{s}]\\
      & = & \sum_{r=0}^{k+l}\theta^{r}  \left[ \sum_{s=r}^{k+l} \sum_{t=(s-l)\vee 0}^{s\wedge k} a_{t,k}a_{s-t,l}  \binom{s}{r}\E[E^{s-r}]\right]\\
      & =& \theta^{k+l}+ \sum_{r=0}^{k+l-2}\theta^{r}  \left[ \sum_{s=r}^{k+l} \sum_{t=(s-l)\vee 0}^{s\wedge k} a_{t,k}a_{s-t,l}  \binom{s}{r}\E[E^{s-r}]\right]\ ,
    \end{eqnarray*}
    because $a_{k,k}=a_{l,l}=1$ and $a_{k-1,k}=a_{l-1,l}=0$. The rhs in the above identity corresponds to a polynomial of degree at most $k+l-2$. Arguing as above, we derive that its coefficients are uniformly bounded (in absolute value) by $(c\|E\|_{\psi_2})^{(k+l)}(k+l)^{2+(k+l)}$. 
    
    Finally, we consider the case where 
    $k+l=2=k\vee l$. Then, we observe that $\E[H_k(Y)H_{l}(Y)]= \E[H_2(Y)]= \E[(\theta+E)^2-1]= \theta^2$ so that $Q_{k,l}=0$. 
    
    \end{proof}

  \begin{proof}[Proof of Lemma~\ref{lem:decomposition_G_P}] 
    We prove the first result by induction on the size of   $v_{\cP}= \sum_{r}|\cP_r|\1_{|\cP_r|>1}$. If $v_{\cP}=0$, then $G_{\cP}=0$ by definition. Consider any partition $\cP$ such that $v_{\cP}>0$, $\cP$ contains a group, say $\cP_1$, of size $2$ satisfying $\cP_1\subset [2k]$ or $\cP_1\subset[2k+1,4k]$. We build the finer collection $\cP^{(-1)}\triangleleft \cP$ by breaking $\cP_1$ into two singletons. Then, 
    \begin{eqnarray*} 
    G_{\cP}(m)&= &F_{\cP}(m) - G_{\cP^{(-1)}}(m)  - \sum_{\cP'\triangleleft \cP,\  \cP'\ntrianglelefteq \cP^{(-1)} }G_{\cP'}(m) - \sum_{ \cP'\triangleleft \cP^{(-1)} }G_{\cP'}(m) \\
    & = & F_{\cP}(m) - G_{\cP^{(-1)}}(m) - \sum_{ \cP'\triangleleft \cP^{(-1)} }G_{\cP'}(m) \\
    & =& F_{\cP}(m) -F_{\cP^{(-1)}}(m) \  , 
    \end{eqnarray*}
    where  we used the definition of $G_{\cP}(m)$ and that, by induction,  $G_{P'}(m)=0$ for $P'$ such $P'\triangleleft \cP$ and   $P'\ntrianglelefteq \cP^{(-1)}$. Finally, since $Q_{2,0}=Q_{0,2}=P_{2}=P_0=0$, we deduce from~\eqref{eq:definition_F_P} that $F_{\cP}(m)= F_{\cP^{(-1)}}(m)$ so that $G_{\cP}(m)=0$.

    We now turn to the second result. Let us apply the generalized M\"obius inversion formula for posets~\cite[Sect.3.7]{stanley2011enumerative} to obtain an explicit expression for $G_{\cP}(m)$: 
    \beq\label{eq:mobius_inversion}
    G_{\cP}(m)= \sum_{\cP'\trianglelefteq \cP} (-1)^{|\cP|-|\cP'|} F_{\cP'}(m)\mu(\cP,\cP') \ , 
    \eeq
    where $\mu$ is the M\"obius function for partitions. More precisely, for $\cP'\trianglelefteq \cP$ where $\cP$ has $t$ groups which respectively split into $s_1$,\ldots, $s_t$ block in $\cP'$, we have $\mu(\cP,\cP')= (-1)^{|\cP|-|\cP'|}\prod_{r=1}^{t}(s_r-1)!$\ . 
    
    We denote $\kappa$ the number of groups of $\cP$ of size larger to one and, without loss of generality, we denote these groups $\cP_1,\ldots, \cP_\kappa$. Writing $\cG_{\cP}$ as a polynomial of $|\cP|$ variables, we then deduce from~\eqref{eq:definition_F_P} that 
    \begin{eqnarray}
    \lefteqn{G_{\cP}[x_1,\ldots, x_{|\cP|}]}&& \nonumber \\ \nonumber
    &=&  \left(\prod_{l=\kappa+1}^{|\cP|} x_l
    \right)\sum_{\underline{\cP'_1} \trianglelefteq \cP_1}  
    \sum_{\underline{\cP'_2}\trianglelefteq \cP_2} 
    \ldots \sum_{\underline{\cP'_\kappa}\trianglelefteq \cP_\kappa} 
    (-1)^{t-\sum_{l=1}^{\kappa}|\underline{\cP'_l|}}\prod_{l=1}^{\kappa} (\underline{|\cP'_l|}-1)!\\ 
    && \cdot \left[\prod_{l=1}^\kappa  x_l^{|\cP_l|} + \prod_{l=1}^{\kappa}\prod_{\cQ \in \underline{\cP'_l}} \left[ x_l^{|\cQ|} + Q_{|\cQ\cap [2k]|,|\cQ\cap [2k+1:4k]|}(x_l)\right]  \right.\nonumber \\ 
    &-& \left.  \prod_{l=1}^{\kappa}\prod_{\cQ \in \underline{\cP'_l}} \left[  x_l^{|\cQ|}+ x_l^{|\cQ\cap [2k]|}P_{|\cQ\cap [2k+1:4k]|}(x_l) \right] 
    -   \prod_{l=1}^{\kappa}\prod_{\cQ \in \underline{\cP'_l}} \left[  (x_l^{|\cQ|}+ x_l^{|\cQ\cap [2k+1:4k]|}P_{|\cQ\cap [2k]|}(x_l) \right]  \right]\ . \nonumber \\ \label{eq:expression_G_P}
    \end{eqnarray}
    where we consider any finer partition $\cP'\trianglelefteq \cP$ as a combination of $t$ partitions of the groups $\cP_1,\ldots, \cP_t$. We now prove that, for 
    $r=1,\ldots, t$, the maximum degree of $x_l$ in  $G_{\cP}$ is at most $|\cP_r|-2$ and is at most $|\cP_r|-3$ when $\cP_l\subset [2k]$ or $\cP_r\subset [2k+1,4k]$. Without loss of generality, we fix $r=1$. We only have to consider monomials with degree $|\cP_1|$ in $x_1$ as there is no monomial of degree $|\cP_1|-1$ in $x_1$ and no monomial of degree $|\cP_1|-2$ in $x_1$ when  $\cP_1\subset [2k]$ or $\cP_1\subset [2k+1,4k]$. Gathering all the monomials in $G_{\cP}$ of degree $|\cP_1|$ in $x_1$,  we arrive at 
    \beqn 
    && x_{1}^{|\cP_1|}\left[\sum_{\underline{\cP'_1} \trianglelefteq \cP_1} (-1)^{1-|\underline{\cP'_1}|} (|\underline{\cP'_1}|-1)! \right] \left[\prod_{l=\kappa+1}^{|\cP|} x_l\right] Z \ ;\\
    Z&:=& \sum_{\underline{\cP'_2} \trianglelefteq \cP_2} \ldots \sum_{\underline{\cP'_\kappa} \trianglelefteq \cP_\kappa}\prod_{l=2}^{\kappa} (-1)^{1-|\cP'_l|}(\underline{|\cP'_l|}-1)!  \left[\prod_{l=2}^\kappa  x_l^{|\cP_l|} + \prod_{l=1}^{\kappa}\prod_{\cQ \in \underline{\cP'_l}} \left[ x_l^{|\cQ|} + Q_{|\cQ\cap [2k]|,|\cQ\cap [2k+1:4k]|}(x_l)\right]  \right.\\ 
     &&- \prod_{l=2}^{\kappa}\prod_{\cQ \in \underline{\cP'_l}} \left[  x_l^{|\cQ|}+ x_l^{|\cQ\cap [2k]|}P_{|\cQ\cap [2k+1:4k]|}(x_l) \right] 
     -   \left.  \prod_{l=2}^{\kappa}\prod_{\cQ \in \underline{\cP'_l}} \left[  (x_l^{|\cQ|}+ x_l^{|\cQ\cap [2k+1:4k]|}P_{|\cQ\cap [2k]|}(x_l) \right]  \right]\ .
    \eeqn 
    Finally, we need to compute  $M_{\cP_1}:= \sum_{\underline{\cP'_1} \trianglelefteq \cP_1} (-1)^{1-|\underline{\cP'_1}|} (|\underline{\cP'_1}|-1)!$. Applying again the generalized inverse M\"obius inverse formula for posets with the constant functional $A_{\cP'}=1$ for any partition $\cP'$ of $\cP_1$, we conclude that $M_{\cP_1}=0$. As a consequence, the degree of $G_P$ in $x_1$ is at most $|\cP_1|-2$ and is at most $\cP_1-3$ if $\cP_1\subset [2k]$ or $\cP_1\subset [2k+1,4k]$.

    \end{proof}

    \begin{proof}[Proof of Lemma~\ref{lem:G_P}]
  It follows from~\eqref{eq:expression_G_P} that $G_{\cP}$ decomposes into a sum of at most $(4k)^{4k}$ -- upper bound for the number of sub-partitions-- sum of four polynomials (see the two last lines in~\eqref{eq:expression_G_P}). Each of these four polynomial decomposes into at most $2^{2k}$ monomials. Finally, it follows from Lemma~\ref{lem:hermite_subgaussian} that the coefficient associated to each of these monomials is smaller or equal $(4k-1)!(c\|E\|_{\psi_2})^{4k}(4k)^{4k+2}$.
  \[ 
  \sum_{D\in \mathrm{Deg}(\cP)} |\alpha_{D,\cP}| \leq 4\cdot 2^{2k} (4k)^{12k+2} (c\|E\|_{\psi_2})^{4k}\leq (c'\|E\|_{\psi_2})^{4k}k^{12k+2}\enspace . 
  \]
      \end{proof}

    \subsection{Proof of Corollary~\ref{cor:schatten_non_gaussian}}

  \begin{proof}[Proof of Corollary \ref{cor:schatten_non_gaussian}]
    The proof follows the same steps as for Corollary~\ref{cor:schatten}. First we observe that $\|\bA\|_{\infty}\leq \|\bA\|_{2k}$ so that 
    \[
      \E\left[\left(U_k - \|\bA\|_{2k}^{2k}\right)^2\right]\leq   (c\|E\|_{\psi_2}k)^{c'k} \left[   (pq)^{k} +  pq \|\bA\|_{2k}^{4k-4}+ q\|\bA\|_{2k}^{4k-2}\right]  
    \]
    We consider two cases depending on $\|\bA\|_{2k}$. If  $\|\bA\|_{2k} \leq  (pq)^{1/4}$, then the above risk bound simplifies in 
    $\E\left[|(U_k-\|\bA\|_{2k}^{2k})^2\right]\leq   3(c\|E\|_{\psi_2}k)^{c'k}(pq)^k$. Then, together with the inequality $|x^{1/(2k)}-y^{1/(2k)}|\leq |x-y|^{1/(2k)}$ and H\"older's inequality, we arrive at
    \[ \E\left[|(U_k)_+^{1/(2k)}- \|\bA\|_{2k}|\right]\leq \E\left[|U_k- \|\bA\|^{2k}_{2k}|\right]^{1/2k}\leq c''(\|E\|_{\psi_2}k)^{c_3} (pq)^{1/4}\enspace .
    \]
        Turning to the case where  $\|\bA\|_{2k} \geq  (pq)^{1/4}$,
        we use   $|(1+x)^{1/(2k)}-1|\leq  2|x|$ for any $x\geq - 1/2$, we have 
        \begin{eqnarray*}
        \E\left[|(U_k)_+^{1/(2k)}- \|\bA\|_{2k}|\right]&=&\E\left[ \|\bA\|_{2k} \big|[U_k/\|\bA\|^{2k}_{2k}]^{\frac{1}{2k}} - 1\big|\right] \nonumber\\
        &\leq & \|\bA\|_{2k}\P\left[U_k\leq \|\bA\|_{2k}^{2k}/2\right]+ 2\|\bA\|^{1-2k}_{2k}\E\left[|U_k- \|\bA\|_{2k}^{2k}|\right] \nonumber \\
        &\leq & \|\bA\|_{2k}\left[4\frac{\E\left[(U_k- \|\bA\|_{2k}^{2k})^2\right]}{\|\bA\|^{4k}_{2k}}+ 2\frac{\E\left[(U_k- \|\bA\|_{2k}^{2k})^2\right]^{1/2}}{\|\bA\|^{2k}_{2k}} \right]\nonumber \\
        &\lesssim &  (c\|E\|_{\psi_2}k)^{c'k} \left[\frac{(pq)^k}{\|\bA\|^{4k-1}_{2k}}+ \frac{(pq)^{k/2}}{\|\bA\|^{2k-1}_{2k}}+ \frac{pq+\sqrt{pq}\|\bA\|_{2k}}{\|\bA\|_{2k}^{2}}+  \frac{q}{\|\bA\|_{2k}}+\sqrt{q} \right] \nonumber
        \\
        &\lesssim & (c\|E\|_{\psi_2}k)^{c'k} (pq) ^{1/4}
        \ .
        \end{eqnarray*}
    The result follows. 
        \end{proof}

\subsection{Proofs of Propositions~\ref{prp:control_operator_subgaussian} and \ref{prp:upper_risk_nuclear_covariance_subgaussian}}

        \begin{proof}[Proof of Proposition~\ref{prp:control_operator_subgaussian}]
          As in proof of  Proposition~\ref{prp:control_operator}, we start from~\eqref{eq:upper_error_operator}
          \[
            \big|\|\bY^T\bY\|_{\infty} - p - \|\bA^T\bA\|_{\infty}\big|\leq   2\|\bA^T\bE\|_{\infty}+ \|\bE^T\bE - p\bI\|_{\infty} \leq   2\|\bA\|_{\infty}\|\bPi_{I(\bA)}\bE\|_{\infty}+ \|\bE^T\bE - p\bI\|_{\infty} \enspace .
          \]
          We control $\|\bE^T\bE - p\bI\|_{\infty}$ using \eqref{eq:convergence_specture_empirical_covariance_subgaussian} and $\|\bPi_{I(\bA)}\bE\|_{\infty}$ as in Lemma~\ref{lem:concentration_pi_A_E_subgaussian}. The remainder of the proof is as in the proof of Proposition~\ref{prp:control_operator} and we obtain
          \[
            \mathbb{E}\cro{|(\|\bY^T \bY\|_{\infty}-p)_+^{1/2} - \|\bA\|_{\infty}|} \leq c_{\|E\|_{\psi_2}}(pq)^{1/4}\ . 
          \]
           
             \end{proof}

              \begin{proof}[Proof of Proposition~\ref{prp:upper_risk_nuclear_covariance_subgaussian}]
                We argue as in the proof of Proposition~\ref{prp:upper_risk_nuclear_covariance}. Starting from~\eqref{eq:upper_T_S}, we have
                \beqn 
                |T^s_s - \|\bA\|^s_s|
                &\leq & 2^{s/2}s \sum_{i=1}^{q}\left[\|\Pi_{I(\bA)}\bE\|^2_{\infty}+ \|\bE^T\bE- p \bI_q\|_{\infty}+ 2 \sigma_i(\bA)\|\Pi_{I(\bA)}\bE\|_{\infty}\right]\sigma^{s-2}_i(\bA)\\ & &\quad \quad\quad \quad  + 3^{s/2-1}\left[\|\Pi_{I(\bA)}\bE\|^s_{\infty}+ \|\bE^T\bE- p \bI_q\|^{s/2}_{\infty}+ 2^{s/2} \sigma^{s/2}_i(\bA)\|\Pi_{I(\bA)}\bE\|^{s/2}_{\infty}\right]\ ,
                \eeqn 
                and we only have to control moments of $\|\Pi_{I(\bA)}\bE\|_{\infty}$ and $\|\bE^T\bE- p \bI_q\|_{\infty}$.
          
          \begin{lem}\label{lem:concentration_pi_A_E_subgaussian}
            There exists a numerical constant $c>0$, such that for any $t>0$, we have
            \beq\label{eq:concentration_pi_A_E_subgaussian}
            \P\left[\|\Pi_{I(\bA)}\bE\|_{\infty} \leq c\|E\|_{\psi_2} \left[\sqrt{q}+ \sqrt{t}\right]\right]\geq 1- e^{-t}\ . 
            \eeq
            \end{lem}
            
            Next, we turn to the control of $\|\bE^T\bE -p \bI_q\|_{\infty}$ which is a straightforward application of Theorem~9 in~\cite{koltchinskii2014concentration}. There exists a constant $c_{\|E\|_{\psi_2}}$ that only depends on $\|E\|_{\psi_2}$ such that for any $t$,
            \beq\label{eq:convergence_specture_empirical_covariance_subgaussian}
            \P\left[\|\bE^T\bE -p \bI_q\|_{\infty}\leq c_{\|E\|_{\psi_2}}\left[\sqrt{pq}+ q + \sqrt{pt}+ t\right]\right]\geq 1-e^{-t}\ .
            \eeq
          Integrating the deviation bounds of~\eqref{eq:concentration_pi_A_E_subgaussian} and~\eqref{eq:convergence_specture_empirical_covariance_subgaussian}, we arrive as in~\eqref{eq:upper_risk_moment1} and \eqref{eq:upper_risk_moment2}
           \[
                 \E\left[|T^s_s - \|\bA\|^s_s|\right]\leq c^s_{\|E\|_{\psi_2}} \left[q[(p\vee s)(q\vee s)]^{s/4}+ \sqrt{q}\|\bA\|^{s-1}_{s-1}+ \1_{s\geq 2}\sqrt{pq}\|\bA\|_{s-2}^{s-2}\right]\ . 
            \]
          Then, the remainder of the proof is exactly as in the proof of Proposition~\ref{prp:upper_risk_nuclear_covariance_subgaussian}.
           \end{proof}

           \begin{proof}[Proof of Lemma~\ref{lem:concentration_pi_A_E_subgaussian}]
            We provide a simple but pedestrian proof based on $\epsilon$-nets. We start from
            \beq\label{eq:def_pi_I_A}
            \|\Pi_{I(\bA)}\bE\|_{\infty} = \sup_{u\in I(\bA), |u|\leq 1}\quad  \sup_{v\in \mathbb{R}^q:\ |v|_2\leq 1 }u^T \bE v  
            \eeq
              Consider an $1/4$-net $\mathcal{N}_1$ of the unit ball of $I(\bA)$ and a $1/4$-net $\mathcal{N}_2$ of the unit ball of $\mathbb{R}^q$.
            Let $u^*$ and $v^*$ be any maximizer of the supremum in~\eqref{eq:def_pi_I_A} and let $u'$ and $v'$ denote the corresponding closest points in $\mathcal{N}_1$ and $\mathcal{N}_2$.  Then, we have 
            
            \beqn 
            \|\Pi_{I(\bA)}\bE\|_{\infty} & = &(u^*)^T \bE v^* = (u')^T \bE v'+ (u^*- u')^T\bE v' + (u^*)^T \bE (v^*-v')\\ &\leq& \sup_{u\in \cN_1}\sup_{v\in \cN_2}u^T \bE v  + \frac{1}{2}\|\Pi_{I(\bA)}\bE\|_{\infty}
            \eeqn
            so that 
            \[
              \|\Pi_{I(\bA)}\bE\|_{\infty}\leq 2 \sup_{u\in \cN_1}\sup_{v\in \cN_2}u^T \bE v
            \]
            Since the mean zero random variable $u^T \bE v$ is sub-Gaussian with sub-Gaussian norm $\|E\|_{\psi_2}$, we deduce that, with probability higher than $1- e^{-t}$, one has 
            \[
              \|\Pi_{I(\bA)}\bE\|_{\infty}\leq c\|E\|_{\psi_2}\left[\sqrt{t}+ \sqrt{\log(|\cN_1|)+\log(|\cN_2|) }\right]\ , 
            \]
            Finally, we conclude by observing that $\log(|\cN_2|)+\log(|\cN_1|)\leq c'q$ since the $\epsilon$-covering number of a $d$-dimensional unit ball is at most $d\log(3/\epsilon)$.
            \end{proof}

\appendix

\section{Technical inequalities}

For bounding the singular values of $\bE$, we shall rely on the following results (taken e.g. from  \cite{Davidson2001}). 
\begin{lem} \label{lem:spectrum}
Let  $\bE$ be a $p \times q$ matrix  whose entries are independent standard Gaussian random variables. Then, 
\[ \mathbb{E}\cro{|\sigma_1(\bE)|}\leq \pa{\sqrt{p} + \sqrt{q}}\enspace . \]
\end{lem}

\begin{lem}\label{lem:DavSza}
Let  $\bE$ be a $p \times q$ matrix  whose entries are independent standard Gaussian random variables. Then,  for any $t > 0$, we have 
\begin{equation}
\max\ac{\mathbb{P}\pa{\sigma_1\pa{\boldsymbol{\bE}}\geq \sqrt{p} + \sqrt{q} + \sqrt{2t}},\mathbb{P}\pa{\sigma_q \pa{\bE}\leq\sqrt{p} - \sqrt{q} - \sqrt{2t} }} \leq e^{-t}\ . 
\end{equation}
\end{lem}

The following lemma is taken from~\cite{laurent00}

\begin{lem}\label{lem:chi2}
 Let $Z$ be distributed as $\chi^2(p)$ random variable. For any $t>0$, we have 
 \[
  \P[Z\geq p + 2\sqrt{pt}+ 2t ]\leq e^{-t}\ . 
 \]
\end{lem}

\bibliography{biblio}
\bibliographystyle{plain}

\end{document}